\documentclass[opre]{informs3}

\OneAndAHalfSpacedXI

\usepackage{endnotes}
\usepackage{url}
\let\footnote=\endnote

\usepackage[skip=2pt,font=scriptsize]{subcaption}
\usepackage{amsmath,amssymb,mathtools}
\usepackage{bm}
\usepackage{microtype}
\usepackage{natbib}
\usepackage{graphicx}
\usepackage{booktabs,tabularx}
\usepackage{algorithm,algpseudocode}

\usepackage{tikz}
\usetikzlibrary{backgrounds,calc,arrows.meta,decorations.markings}
 
\bibpunct[, ]{(}{)}{,}{a}{}{,}
\def\bibfont{\small}

\TheoremsNumberedThrough
\EquationsNumberedThrough

\newcommand{\R}{\mathbb{R}}
\newcommand{\E}{\mathbb{E}}
\newcommand{\Pnom}{\hat{\mathbb{P}}}    
\newcommand{\Pprob}{\mathbb{P}}        
\newcommand{\Qprob}{\mathbb{Q}}        
\newcommand{\Amb}{\mathbb{B}_\rho} 
\newcommand{\cX}{\mathcal{X}}          
\newcommand{\cZ}{\mathcal{Z}}          
\newcommand{\XiSet}{\Xi}               
\newcommand{\Loss}{\ell}               

\newcommand{\proofgap}{\par\vspace{0.35\baselineskip}}

\newcommand{\conv}{\mathrm{conv}}
\newcommand{\cone}{\mathrm{cone}}
\newcommand{\diam}{\mathrm{diam}}
\newcommand{\Lip}{\mathrm{Lip}}
\newcommand{\chg}[1]{#1}

\begin{document}

\RUNAUTHOR{Meng, Cory-Wright, and Wiesemann}
\RUNTITLE{A Shrinkage Path Heuristic for Wasserstein DRO}
\RRHSecondLine{}
\LRHSecondLine{}
\def\setoddRH{\hbox to \textwidth{\fs.7.8.\tabcolsep0pt
  \begin{tabular*}{\textwidth}[b]{l@{\extracolsep\fill}r}
  {\theRRHFirstLine}&\raisebox{0pt}[0pt][0pt]{\fs.10.10.\thepage}\\[-4pt]
  \rlap{\VRHDW{0.5pt}{0pt}{\textwidth}}&\\
  \end{tabular*}}}
\def\setevenRH{\hbox to \textwidth{\fs.7.8.\tabcolsep0pt
  \begin{tabular*}{\textwidth}[b]{l@{\extracolsep\fill}r}
  \raisebox{0pt}[0pt][0pt]{\fs.10.10.\thepage}&{\theLRHFirstLine}\\[-4pt]
  \rlap{\VRHDW{0.5pt}{0pt}{\textwidth}}&\\
  \end{tabular*}}}

\TITLE{A Shrinkage Path Heuristic for \\ Wasserstein Distributionally Robust Optimization}

\ARTICLEAUTHORS{
\AUTHOR{Lingjun Meng, Ryan Cory-Wright, Wolfram Wiesemann}
\AFF{Imperial Business School, Imperial College London, London SW7 2AZ, United Kingdom
\\ \EMAIL{\{l.meng23, r.cory-wright, ww\}@imperial.ac.uk}}
}

\ABSTRACT{Wasserstein distributionally robust optimization (DRO) is a versatile and widely adopted framework for decision-making under uncertainty, yet its standard deterministic reformulations generally contain non-convex inner subproblems that are challenging to solve. To address this issue, we propose a shrinkage path heuristic that reduces the solution of a DRO problem to a one-dimensional search over the line segment connecting the (typically benign) sample average approximation (SAA) and the (more demanding but practically solvable) classical robust optimization solution. We derive a priori suboptimality bounds in stylized settings and, for the general case, a posteriori bounds obtained by applying a similar heuristic to a dual formulation. Numerical experiments on a multi-item newsvendor and an appointment scheduling problem show that the shrinkage path heuristic attains $85$--$110\%$ (resp. $45$--$70\%$) of the out-of-sample performance improvements of Wasserstein DRO over SAA, at a fraction of the computational cost.}

\maketitle

\section{Introduction}\label{sec:intro}
Decision problems under uncertainty are often solved via sample average approximation (SAA), which involves optimizing against the empirical distribution of these observations. However, SAA is prone to overfitting, as the resulting decisions are tailored to the sample rather than reflecting the underlying distribution that generated it, often resulting in poor out-of-sample performance. To counteract this overfitting, distributionally robust
optimization (DRO) hedges against all distributions in a neighborhood of the empirical distribution. This protection comes at a price, however. Although DRO problems can be reformulated in certain special cases, they are generally more computationally challenging to solve than their SAA counterparts. Additionally, the size of the neighborhood is a hyperparameter that is rarely known in advance and usually requires tuning through cross-validation, which involves resolving the DRO problem for each candidate neighborhood size across multiple folds. This paper proposes a simple shrinkage path heuristic that replaces DRO with a one-dimensional search along the line segment connecting the solutions of two well-studied problems: the SAA problem and its classical robust optimization counterpart. Since the line segment itself does not depend on any hyperparameters, tuning the degree of robustness via cross-validation requires only the two endpoint solutions and inexpensive evaluations along the path. The heuristic thereby captures much of the protection of DRO over SAA while reducing computational costs by orders of magnitude.

Specifically, we study DRO problems of the form
\begin{equation} \label{prob:main}
    \begin{aligned}
        & \mathop{\text{minimize}}_{\bm{x} \in \mathcal{X}} && \sup_{\Qprob \in \Amb(\Pnom)} \E_{\Qprob} \big[ \Loss(\bm{x}, \bm{\tilde{\xi}}) \big] \\
        & \text{subject to} && g_m(\bm{x}, \bm{\tilde{\xi}}) \le \gamma_m \qquad \Qprob\text{-a.s.} \; \forall \Qprob \in \Amb(\Pnom), \; \forall m \in [M],
    \end{aligned}
\end{equation}
where the decision $\bm{x}$ is selected from a closed convex set $\mathcal{X} \subseteq \R^n$, the objective minimizes the worst-case expectation of a loss function $\Loss$ that is closed convex in $\bm{x}$ for every $\bm{\xi}\in\Xi$ and upper semicontinuous in $\bm{\xi}$ for every $\bm{x}\in\mathcal{X}$, and the constraints indexed by $m \in [M]:=\{1, \ldots, M\}$ are closed convex in $\bm{x}$ for every fixed realization $\bm{\xi}$ and hold $\Qprob$-almost surely. The uncertainty is captured by the random vector $\bm{\tilde{\xi}} \in \XiSet$, which may be governed by any distribution $\Qprob$ from the type-1 Wasserstein ambiguity set
\begin{equation*}
    \Amb(\Pnom) \coloneqq \{ \Qprob \in \mathcal{P}(\XiSet) : W(\Qprob, \Pnom) \le \rho \}
    \quad \text{with} \quad
    W(\Qprob, \Pnom) = \inf_{\pi \in \Pi(\Qprob, \Pnom)} \int_{\XiSet \times \XiSet} d(\bm{\xi}, \bm{\xi}') \, \pi(\mathrm{d}\bm{\xi}, \mathrm{d}\bm{\xi}').
\end{equation*}
Here, $d$ is a lower semicontinuous ground metric, and the empirical distribution $\Pnom = \frac{1}{N} \sum_{i=1}^{N} \delta_{\chg{\bm{\hat{\xi}}}_i}$ is constructed from $N$ independent samples $\chg{\bm{\hat{\xi}}}_1, \ldots, \chg{\bm{\hat{\xi}}}_N$ drawn from the unknown true data-generating distribution $\Pprob^\star$. We assume throughout that the support $\Xi \subseteq \mathbb{R}^k$ is non-empty, closed, and convex.

We refer to Problem~\eqref{prob:main} as the Wasserstein DRO (WDRO) problem. It arises in a wide range of applications, including inventory management \citep{hanasusanto2015distributionally, xin2022distributionally}, machine learning \citep{bennouna2023certified}, energy systems \citep{xie2017distributionally, esteban2023distributionally}, and control theory \citep{shafieezadeh2018wasserstein}. We refer to \citet{rahimian2019distributionally} and \citet{kuhn2025distributionally} for detailed reviews.

Unfortunately, Problem~\eqref{prob:main} is often not \emph{practically tractable}—that is, solvable in a practical amount of time at instance sizes relevant to the application \citep{bertsimas2019machine}. By strong duality, the worst-case expectation objective admits a finite-dimensional dual reformulation. In this reformulation, one optimizes over the dual multiplier corresponding to the Wasserstein radius and solves $N$ auxiliary maximization problems, one for each sample point, over the support $\Xi$. Each auxiliary problem maximizes a loss-like term minus a transport-cost penalty \citep[Theorem~4.2]{mohajerin2018data}. The tractability of Problem~\eqref{prob:main}, therefore, hinges on the structure of these inner maximization problems. When $\Xi$ is unbounded (e.g., $\Xi = \R^k$ or $\Xi = \R^k_+$) and the ground metric is a norm, each supremum is finite only if the loss function is Lipschitz continuous on $\Xi$. In the special case where $\Xi = \R^k$, this finite problem reduces to a regularized sample average approximation (SAA) \citep[Theorem 4(ii)]{shafieezadehabadeh2019regularization}. However, many losses of practical interest---including quadratic, exponential, and other smooth superlinear losses---fail the Lipschitz condition and therefore can render the worst-case expectation infinite along unbounded rays. When $\Xi$ is compact, the inner suprema are always finite, but their tractability depends on the loss structure. If the loss function is \emph{a pointwise maximum of finitely many concave functions} of $\bm{\xi}$ \citep[Assumption~4.1]{mohajerin2018data}, each inner problem maximizes a piecewise concave function over a convex set, which is tractable. In contrast, when the loss function is \emph{convex} in $\bm{\xi}$---as is the case for quadratic and other smooth losses commonly encountered in practice---each inner subproblem becomes a $k$-dimensional difference-of-convex program; in general, solving such problems is NP-hard, even for highly structured quadratic instances \citep{pardalos1991quadratic}. 

Several strategies aim to circumvent the intractability of WDRO with general loss functions. When the loss is convex in $\bm{\xi}$, Corollary~4.3 of \citet{mohajerin2018data} yields finite convex conservative approximations of the inner suprema, although their sizes typically grow exponentially in the dimension of $\bm{\xi}$. \citet{cheramin2022computationally} propose inner and outer approximations for moment and Wasserstein ambiguity sets based on block decompositions of $\bm{\xi}$ and principal component reduction, together with gap bounds that quantify the resulting accuracy--runtime trade-off. \citet[Theorem~2]{gao2024wasserstein} connect type-1 Wasserstein DRO to variation regularization and show that the gap between the worst-case expectation and the empirical loss is bounded by $\rho$ times a variation norm of the loss function. 

A second stream of work solves Problem \eqref{prob:main} algorithmically. \citet{shafieezadehabadeh2025nash} study gradient-based methods for Wasserstein DRO when $\Xi$ is unconstrained and the loss $\chg{\Loss}$ is nonconvex in $\bm{x}$. \citet{wang2025sinkhorn} replace the Wasserstein distance with its entropically regularized counterpart; the resulting Sinkhorn dual is amenable to optimization via stochastic mirror descent. Two further works propose dual decomposition algorithms for special problem classes: \citet{li2019first} propose an ADMM framework that exploits a dual decomposition for the Wasserstein distributionally robust logistic regression. \citet{luo2019decomposition} develop a cutting-surface decomposition algorithm for the dual reformulation of the Wasserstein DRO with bounded support for a class of statistical learning problems. Finally, \citet{blanchet2022optimal} study optimal-transport DRO with quadratic transport costs and linear decision rules, which together reduce each inner adversarial subproblem to a line search, and develop stochastic gradient descent schemes for the outer decision with convergence guarantees. The GitHub repository accompanying this paper classifies the literature in greater depth according to the scope of each approach, the nature of the resulting reformulation or algorithm, and whether the approach is exact\endnote{GitHub repository: \url{https://github.com/lingjun-meng/minimax-shrinkage}.}.


All of the above approaches aim to solve Problem~\eqref{prob:main}---or an approximation thereof---for a \emph{single fixed} Wasserstein radius $\rho$. In contrast, we exploit the structure of the solution path \emph{as $\rho$ varies}. We adopt this perspective because, in most practical applications, $\rho$ is not selected a priori (\emph{e.g.}, from a concentration inequality), but rather it is tuned via cross-validation on the data. To this end, we propose a shrinkage path heuristic (Algorithm~\ref{alg:iph}) that reduces Problem~\eqref{prob:main} to a one-dimensional search over the line segment that connects two benchmark decisions: the solution of the sample average approximation (SAA), which arises at $\rho = 0$ and is typically easy to compute, and the solution of the classical robust optimization (RO) problem, which arises as $\rho \rightarrow \infty$ and—although harder—is routinely solved in practice. The line segment itself is independent of $\rho$: instead of solving a full DRO problem at every candidate radius, one tunes the position on this fixed path by cross-validation, saving orders of magnitude in runtime. Although simple, the heuristic admits a priori guarantees in stylized settings (Theorems~\ref{thm:bounded}--\ref{thm:endpoint} and Proposition~\ref{thm:unbounded} in Section~\ref{sec:apriori}), a posteriori guarantees in general (Theorems~\ref{thm:lb-algorithm} and \ref{thm:posteriori} in Section~\ref{sec:aposteriori}), and competitive empirical performance (Section~\ref{sec:numerics}); these three advances form the core of this paper.

More specifically, we summarize our contributions as follows.
\begin{enumerate}
  \item Under suitable stylized assumptions on the loss function and uncertainty support, we establish \emph{a priori suboptimality bounds} that serve as a theoretical motivation for our heuristic.
  \item For the general setting of Problem~\eqref{prob:main}, we apply the same shrinkage idea to a dual formulation, yielding a companion procedure (Algorithm~\ref{alg:lower-bound}) that mirrors Algorithm~\ref{alg:iph} on the dual side. The gap between the primal and dual heuristic values provides computable (per-instance) \emph{a posteriori suboptimality bounds} without requiring an exact solution of Problem~\eqref{prob:main}.
  \item We demonstrate the effectiveness of our heuristic on a multi-item newsvendor and an appointment scheduling problem, where our a~posteriori bounds certify that the heuristic incurs only a small in-sample suboptimality gap for the DRO objective. Out of sample, for the multi-item newsvendor setting (resp.\ the appointment scheduling problem), the shrinkage path heuristic attains $85$--$110\%$ (resp.\ $45$--$70\%$) of the median out-of-sample benefits of Wasserstein DRO over SAA, at a fraction of the computational cost.
\end{enumerate}

The remainder of this paper is organized as follows. Section~\ref{sec:heuristics} introduces our heuristic. Section~\ref{sec:apriori} derives a priori suboptimality bounds under several structural assumptions on the loss function. Section~\ref{sec:aposteriori} develops a posteriori performance guarantees for the general setting by constructing a dual shrinkage path whose objective values provide computable lower bounds on the true DRO optimum. Section~\ref{sec:numerics} illustrates our heuristic on a multi-item newsvendor and an appointment scheduling problem. All proofs are deferred to the electronic companion.  We relegate an extended literature review and a generalization of Problem~\eqref{prob:main} to ambiguous conditional value-at-risk constraints to the GitHub repository accompanying this work.

\begin{remark}[Multi-Stage Problems]
    Our heuristic extends to two-stage and multi-stage variants of Problem~\eqref{prob:main} in two natural ways. The first is to restrict the recourse decisions to affine (or piecewise-affine) functions of the uncertainty, which collapses the problem to a single-stage instance of Problem~\eqref{prob:main} in an expanded first-stage decision vector. The second is to substitute the optimal recourse value function for $\Loss$ directly, whenever this value function remains convex in $(\bm{x}, \bm{\xi})$; this is the route taken in our numerical experiments.
\end{remark}

\section{The Shrinkage Path Heuristic}\label{sec:heuristics}

We begin by simplifying the feasible set of Problem~\eqref{prob:main}. For any $\rho > 0$ and each point $\bm{\xi} \in \Xi$, the Wasserstein ball $\Amb(\Pnom)$ contains a distribution that places positive mass at $\bm{\xi}$, so the almost-sure constraints decouple from the radius and reduce to the support-only robust requirement that $g_m(\bm{x}, \bm{\xi}) \le \gamma_m$ for all $\bm{\xi} \in \Xi$ and every $m \in [M]$. Accordingly, the feasible set of Problem~\eqref{prob:main} reduces to
\begin{equation}\label{eq:feasible-set}
    \mathcal{X}(\Xi) \coloneqq \big\{ \bm{x} \in \mathcal{X} \, : \, g_m(\bm{x}, \bm{\xi}) \le \gamma_m \; \; \forall \bm{\xi} \in \Xi, \; \forall m \in [M] \big\}.
\end{equation}
We assume throughout that $\mathcal{X}(\Xi)$ is non-empty, otherwise, Problem~\eqref{prob:main} is trivially infeasible. Convexity of $\mathcal{X}$ together with convexity of each $g_m(\cdot, \bm{\xi})$ ensures that $\mathcal{X}(\Xi)$ is itself convex. The set $\mathcal{X}(\Xi)$ is described by semi-infinite constraints. When each $g_m(\bm{x}, \cdot)$ is concave on $\Xi$ and $\Xi$ admits a conic representation, classical robust counterpart techniques yield a tractable finite-dimensional reformulation \citep{ben2009robust}. More broadly, $\mathcal{X}(\Xi)$ remains tractable whenever it admits a polynomial-time separation oracle, that is, whenever for any candidate $\bm{x}^\star$ a violated constraint can be identified by maximizing $g_m(\bm{x}^\star, \bm{\xi})$ over $\bm{\xi} \in \Xi$ in polynomial time.

The \emph{sample average approximation} (SAA) solution is any optimal solution $\bm{x}(0)$ to
\begin{subequations}\label{eq:extremes}
\begin{equation}\label{eq:extremes:saa}\begin{aligned}
    & \mathop{\text{minimize}}_{\bm{x} \in \mathcal{X} (\Xi)} && \E_{\Pnom} \big[ \Loss(\bm{x}, \bm{\tilde{\xi}}) \big].
\end{aligned}\end{equation}
We assume that Problem~\eqref{eq:extremes:saa} attains a finite optimal value, ensuring that $\bm{x}(0)$ is well-defined. Since $\Pnom \in \Amb(\Pnom)$, this further implies that the WDRO problem~\eqref{prob:main} is bounded below. Whenever $\Loss$ is convex in $\bm{x}$ and $\mathcal{X}(\Xi)$ admits a polynomial-time separation oracle, Problem~\eqref{eq:extremes:saa} is solvable in polynomial time; the Wasserstein ambiguity set plays no role at this extreme.

The \emph{robust optimization} (RO) solution is any optimal solution $\bm{x}(\infty)$ to
\begin{equation}\label{eq:extremes:ro}\begin{aligned}
    & \mathop{\text{minimize}}_{\bm{x} \in \mathcal{X} (\Xi)} && \sup_{\bm{\xi} \in \Xi} \; \Loss(\bm{x}, \bm{\xi}).
\end{aligned}\end{equation}
\end{subequations}
We assume that Problem~\eqref{eq:extremes:ro} attains a finite optimal value, ensuring that $\bm{x}(\infty)$ is well-defined. This holds, for instance, when $\Xi$ is compact and $\Loss$ is upper semicontinuous in $\bm{\xi}$, or more generally when $\Loss(\bm{x}, \cdot)$ is bounded above on $\Xi$ for every $\bm{x} \in \mathcal{X}(\Xi)$. Problem~\eqref{eq:extremes:ro} is tractable when $\Loss(\bm{x}, \cdot)$ is concave on $\Xi$. In the convex-in-$\bm{\xi}$ regime of interest in this paper, the problem is NP-hard in general \citep{pardalos1991quadratic}---the same source of intractability that afflicts the inner suprema in the dual reformulation of Problem~\eqref{prob:main}. Even so, as in classical robust optimization, this worst-case separation is routinely handled in practice by cutting-plane methods that iteratively add violated worst-case scenarios \citep{mutapcic2009cutting}.

Define the \emph{shrinkage path} as
\begin{equation}\label{eq:path}
    \widehat{\mathcal{X}}(\Xi) \; \coloneqq \; \left\{ (1 - \lambda) \bm{x}(0)+ \lambda \bm{x}(\infty) \, : \, \lambda \in [0, 1] \right\}.
\end{equation}
Because the SAA and RO problems in~\eqref{eq:extremes} need not admit unique optimizers, we fix one solution of each; these endpoints $\bm{x}(0)$ and $\bm{x}(\infty)$ in turn fix the shrinkage path, and we refer henceforth to \emph{the} SAA solution, \emph{the} RO solution, and \emph{the} shrinkage path. The shrinkage path is a tractable proxy for the WDRO solution path $\{\bm{x}^\star(\rho) : \rho \ge 0\}$, where $\bm{x}^\star(\rho)$ denotes any optimal solution of Problem~\eqref{prob:main} at radius $\rho$. As shown by \citet{hao2025robust} for robust linear optimization, $\bm{x}^\star(\rho)$ traces a one-dimensional Bregman-type regularization path; in our setting, however, it is generally intractable to compute. We refer to $\lambda$ as the \textit{shrinkage parameter}: $\lambda = 0$ corresponds to the SAA solution $\bm{x}(0)$, $\lambda = 1$ recovers the RO solution $\bm{x}(\infty)$, and intermediate values balance nominal fit against robustness. Since both endpoints lie on the shrinkage path, optimizing over it weakly dominates the \textit{SAA--RO heuristic} of selecting the better of $\bm{x}(0)$ and $\bm{x}(\infty)$. Figure~\ref{fig:paths} illustrates the shrinkage path alongside the true WDRO optimizer path and the resulting suboptimality gap.

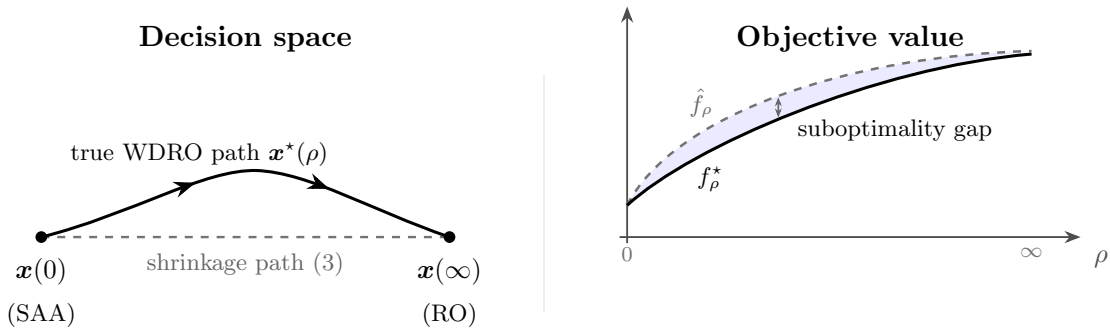
\begin{figure}[h!]
    \centering
    \begin{tikzpicture}[
    >=Stealth,
    every node/.style={font=\small}
]

\begin{scope}[xshift=0cm]
    \node[font=\bfseries\normalsize] at (2.7,2.62) {Decision space};

    \coordinate (SAA) at (0,0);
    \coordinate (RO)  at (5.4,0);

    \draw[black!55, dashed, line width=0.9pt] (SAA) -- (RO);

    \draw[black, line width=1.15pt,
          postaction={decorate, decoration={
              markings,
              mark=at position 0.38 with {\arrow{>}},
              mark=at position 0.70 with {\arrow{>}}
          }}]
        (SAA)
        .. controls (1.20,0.30) and (2.15,0.90) ..
        (2.85,0.88)
        .. controls (3.55,0.85) and (4.35,0.36) ..
        (RO);

    \node[font=\footnotesize, anchor=west] at (0.25,1.10)
        {true WDRO path $\bm{x}^\star(\rho)$};
    \node[font=\footnotesize, text=black!60] at (2.7,-0.36)
        {shrinkage path~(3)};

    \fill (SAA) circle (2.2pt);
    \fill (RO) circle (2.2pt);
    \node[below=4pt, align=center] at (SAA)
        {$\bm{x}(0)$\\[-1pt]{\footnotesize(SAA)}};
    \node[below=4pt, align=center] at (RO)
        {$\bm{x}(\infty)$\\[-1pt]{\footnotesize(RO)}};
\end{scope}

\draw[black!12] (6.65,-1.0) -- (6.65,2.15);

\begin{scope}[xshift=7.75cm]
    \node[font=\bfseries\normalsize] at (2.95,2.62) {Objective value};

    \draw[->, black!70, line width=0.7pt] (-0.1,0) -- (6.0,0) node[below right=1pt] {$\rho$};
    \draw[->, black!70, line width=0.7pt] (0,-0.08) -- (0,3.05);

    \fill[blue!8]
        (0,0.42)
        .. controls (0.60,1.60) and (2.90,2.40) .. (5.35,2.46)
        .. controls (3.30,2.25) and (0.85,1.20) .. (0,0.42);

    \draw[black!55, dashed, line width=0.95pt]
        (0,0.42)
        .. controls (0.60,1.60) and (2.90,2.40) .. (5.35,2.46);

    \draw[black, line width=1.15pt]
        (0,0.42)
        .. controls (0.85,1.20) and (3.30,2.25) .. (5.35,2.42);

    \node[font=\footnotesize, text=black!60, fill=white, inner sep=1pt, anchor=south]
        at (0.98,1.54) {$\hat f_{\rho}$};
    \node[font=\footnotesize, fill=white, inner sep=1pt, anchor=north]
        at (1.10,0.98) {$f_{\rho}^{\star}$};

    \draw[black!60, {Stealth[length=3.5pt]}-{Stealth[length=3.5pt]}, line width=0.6pt] (2.0,1.58) -- (2.0,1.85);
    \node[font=\footnotesize, anchor=west] at (2.12,1.42)
        {suboptimality gap};

    \node[font=\scriptsize, text=black!55, below] at (0,0) {$0$};
    \node[font=\scriptsize, text=black!55, below] at (5.35,0) {$\infty$};
\end{scope}
\end{tikzpicture}
    \caption{Schematic view of the shrinkage path heuristic. Left: geometric approximation in decision space. The dashed segment is the shrinkage path connecting the SAA and RO endpoints, and the solid curve is the WDRO optimizer path $\bm{x}^\star(\rho)$. Right: induced value gap as a function of the Wasserstein radius $\rho$. The dashed curve depicts the shrinkage path values $\hat f_\rho$, the solid curve denotes the WDRO optima $f_\rho^\star$, and the shaded region represents the induced suboptimality gaps $\hat f_\rho-f_\rho^\star$.}
    \label{fig:paths}
\end{figure}

\begin{algorithm}[h!]\small
\caption{Shrinkage Path Heuristic for Problem~\eqref{prob:main}}
\label{alg:iph}
\begin{algorithmic}[1]
\Require Empirical law $\Pnom := \tfrac1N \sum_{i=1}^N \delta_{\chg{\bm{\hat{\xi}}}_i}$ with samples $\{\chg{\bm{\hat{\xi}}}_i\}_{i=1}^N$; radius $\rho\ge 0$; search grid $\Lambda \subset [0, 1]$;
\Statex $\mspace{50mu}$ worst-case evaluation oracle $\textsc{WC-Eval}(\bm{x}; \rho) = \sup \big\{ \mathbb E_\Qprob [\chg{\Loss}(\bm{x},\chg{\bm{\tilde{\xi}}})] \, : \, \Qprob\in \Amb(\Pnom) \big\}$.
\State Initialize incumbent objective $J^\star\gets +\infty$ and incumbent solution $\bm{x}^\star \gets \emptyset$.
    \For{$\lambda\in\Lambda$}
        \State Compute $\bm{x} \gets (1 - \lambda) \bm{x}(0) + \lambda \bm{x}(\infty)$ and $J \gets \textsc{WC-Eval}(\bm{x}; \rho)$.
        \If{$J < J^\star$}
            \State $J^\star \gets J$ and $\bm{x}^\star \gets \bm{x}$.
        \EndIf
    \EndFor
\State \Return $\bm{x}^\star$ with objective value $J^\star$.
\end{algorithmic}
\end{algorithm}

Algorithm~\ref{alg:iph} implements a one-dimensional search over the shrinkage path. Since $\mathcal{X}(\Xi)$ is convex and both endpoints $\bm{x}(0)$ and $\bm{x}(\infty)$ reside in $\mathcal{X}(\Xi)$, every interpolant $\bm{x}$ is feasible in Problem~\eqref{prob:main}.

For a \emph{fixed} radius $\rho$, Algorithm~\ref{alg:iph} replaces the joint optimization in Problem~\eqref{prob:main} with two endpoint problems plus $|\Lambda|$ calls to the worst-case oracle $\textsc{WC-Eval}(\cdot;\rho)$. Each oracle call fixes the decision~$\bm{x}$ and admits the standard Wasserstein dual reformulation \citep[Theorem~4.2]{mohajerin2018data}, which decomposes the worst-case expectation into~$N$ independent $k$-dimensional subproblems---one per data point, each maximizing the loss minus a penalized transport cost over~$\Xi$. These subproblems are decoupled, parallelize easily, and can be handled by standard global solvers (\emph{e.g.}, branch-and-bound or multistart local search) for moderate~$k$. The in-sample speedup over the joint problem is typically modest, and is not the main benefit of the heuristic. The larger savings arise \emph{out-of-sample}, when $\rho$ is itself tuned by $K$-fold cross-validation. In practice, $\rho$ is rarely chosen a priori; it is selected from a candidate grid $\mathcal{R} \subset \mathbb{R}_+$ to optimize cross-validated performance. For the exact WDRO problem, this requires $K |\mathcal{R}|$ solves of Problem~\eqref{prob:main}, each of which is precisely the joint optimization that motivated the heuristic in the first place. The shrinkage path, by contrast, does not depend on $\rho$, and only one of its two endpoints depends on the data: the typically harder RO endpoint $\bm{x}(\infty)$ is computed once from $\mathcal{X}$, $\Loss$, and $\Xi$ alone, while the typically easier SAA endpoint $\bm{x}(0)$ is recomputed per fold from the corresponding training sample. The entire path is then scored on the validation sample \emph{without any worst-case evaluations}. Concretely, denote by $\Pnom^{\mathrm{tr}}$ and $\Pnom^{\mathrm{val}}$ the empirical distributions on a fold's training and validation samples, write $\bm{x}^{\mathrm{tr}}(0)$ for the SAA endpoint computed from $\Pnom^{\mathrm{tr}}$, replace the in-sample oracle in Algorithm~\ref{alg:iph} with the validation-fold sample average
\begin{equation*}
\textsc{Val-Eval}(\bm{x}) \;:=\; \E_{\Pnom^{\mathrm{val}}} \big[\Loss(\bm{x}, \bm{\tilde{\xi}})\big] \;=\; \frac{1}{N^{\mathrm{val}}} \sum_{i=1}^{N^{\mathrm{val}}} \Loss\big(\bm{x}, \chg{\bm{\hat{\xi}}}_i^{\mathrm{val}}\big),
\end{equation*}
and average $\textsc{Val-Eval}$ across the $K$ folds for each $\lambda$ on the search grid. Feasibility need not be checked separately, since $\mathcal{X}(\Xi)$ does not depend on $\rho$ and every interpolant $(1 - \lambda) \bm{x}^{\mathrm{tr}}(0) + \lambda \bm{x}(\infty)$ remains in $\mathcal{X}(\Xi)$ by convexity. Let $\lambda^\star \in \Lambda$ minimize the cross-validated $\textsc{Val-Eval}$ along the path; the final decision $(1 - \lambda^\star) \bm{x}(0) + \lambda^\star \bm{x}(\infty)$ is then formed by recomputing the SAA endpoint $\bm{x}(0)$ on the full sample $\Pnom$, while the RO endpoint $\bm{x}(\infty)$ is reused. The heuristic thus replaces cross-validated tuning of $\rho$, which requires $K |\mathcal{R}|$ solves of Problem~\eqref{prob:main}, with cross-validated tuning of $\lambda$, which requires one RO solve, $K$ SAA solves, and sample-average evaluations along the path. The resulting savings are typically several orders of magnitude.

\section{A Priori Performance Guarantees}\label{sec:apriori}

To obtain a priori guarantees, this section focuses on the following structured setting.
\begin{assumption}[Lipschitz Loss]\label{ass:apriori}
    Problem~\eqref{prob:main} satisfies the following conditions.
    \begin{enumerate}
        \item[(i)] The ground metric is the Euclidean norm: $d(\bm{\xi}, \bm{\xi}') = \| \bm{\xi} - \bm{\xi}' \|_2$.
        \item[(ii)] The loss is $\Loss(\bm{x}, \bm{\xi}) = L(\bm{\xi}^\top \bm{x})$ with $L : \R \to \R_+$ convex and Lipschitz with constant $\Lip (L)$.
    \end{enumerate}
\end{assumption}
Pairing the Euclidean ground metric in~\emph{(i)} with the bilinear coupling $\bm{\xi}^\top \bm{x}$ in~\emph{(ii)} ensures that, for every fixed decision $\bm{x}$, the loss $\bm{\xi} \mapsto L(\bm{\xi}^\top \bm{x})$ is Lipschitz in $\bm{\xi}$ with the constant $\Lip(L) \|\bm{x}\|_2$, which in turn yields explicit, decision-dependent bounds on how fast the worst-case expectation can grow with $\rho$ relative to its nominal counterpart. This implies in particular that the number of uncertain parameters $k$ coincides with the number of decisions $n$ in this section. The Lipschitz assumption on $L$ in~\emph{(ii)} further keeps the worst-case expectation finite when $\Xi$ is unbounded, and convexity of $L$ ensures that the WDRO problem is convex in $\bm{x}$. Under these conditions, Problem~\eqref{prob:main} admits closed-form suboptimality bounds and structural insights that are difficult to obtain in the general case. In particular, when $\Xi = \R^k$, the WDRO problem reduces to a regularized SAA problem
\citep[Theorem~4(ii)]{shafieezadehabadeh2019regularization}. In such settings, the heuristic is not needed in practice, since the WDRO problem is itself tractable. We study them because they can be analyzed in closed form, and the resulting guarantees help interpret the strong performance of the heuristic on the harder instances of Section~\ref{sec:numerics}.

We define the \emph{additive suboptimality} of a candidate solution at a given Wasserstein radius $\rho$ as the gap between its worst-case expected loss and the optimal value of Problem~\eqref{prob:main} at that radius. The worst-case additive suboptimality of a heuristic is the supremum of this gap over all $\rho \ge 0$. We say that a bound on this quantity is \emph{tight} when equality is attained by some problem instance, and \emph{asymptotically tight} when equality is approached in the limit along a sequence of instances parameterized by some structural quantity. When neither holds, we say the bound has \emph{slack}.

We first bound the additive suboptimality of our shrinkage path heuristic when $\Xi$ is bounded.

\begin{theorem}[Bounded Support]\label{thm:bounded}
    Let $\diam(\Xi) = \sup\{\| \bm{\xi} - \bm{\xi}' \|_2 : \bm{\xi}, \bm{\xi}' \in \Xi\} < \infty$.  Then, for any $\rho \in [0, +\infty)$, the additive suboptimality of the shrinkage path heuristic is at most
    \begin{equation*}
        \rho \left[1-\frac{\rho}{\diam(\Xi)}\right]_{+} \Lip (L) \| \bm{x}(0) \|_2,
    \end{equation*}
    which is maximized at $\rho = \diam(\Xi) / 2$. Moreover, the bound is pointwise tight: for every $\rho \in [0, \infty)$, the bound is attained by some problem instance.
\end{theorem}

The bound in Theorem~\ref{thm:bounded} vanishes at both endpoints of $\rho \in [0, \diam(\Xi)]$ and is largest near the midpoint. Thus, the heuristic is provably near-optimal for small $\rho$ and again as $\rho$ approaches $\diam(\Xi)$. The small-$\rho$ regime is the practically relevant one: the out-of-sample performance of DRO typically improves over SAA at small
radii and worsens past a critical threshold \citep{anderson2022improving}. The heuristic therefore suits the common goal of improving SAA with a modest amount of distributional robustness.

We now prove an analogous result for the unbounded support case.

\begin{proposition}[Unbounded Support]\label{thm:unbounded}
    Let $\Xi = \R^k$, and let $\bar r := \frac{1}{N}\sum_{i=1}^{N} \|\chg{\bm{\hat{\xi}}}_i\|_2$ denote the \emph{effective diameter} of the support. \mbox{The additive suboptimality of the shrinkage path heuristic is at most}
    \begin{equation*}
        \Lip(L) \left( \rho \left(1 - \frac{\rho}{\bar r}\right) \|\bm{x}(0)\|_2
        + \frac{\rho^2}{\bar r} \|\bm{x}(\infty)\|_2 \right)
        \qquad \forall \rho \in [0, \bar r],
    \end{equation*}
    which is maximized at $\rho = \bar r / 2$ when $\bm{x}(\infty) = \bm{0}$.
\end{proposition}

Proposition~\ref{thm:unbounded} parallels Theorem~\ref{thm:bounded}, with the (finite) effective diameter $\bar r$ serving as a data-driven surrogate for the (infinite) support diameter $\diam(\Xi)$: below $\bar r$, the bound is a weighted sum of two terms---one proportional to the SAA endpoint $\|\bm{x}(0)\|_2$ with weight $\rho(1-\rho/\bar r)$, and one governed by the RO endpoint $\|\bm{x}(\infty)\|_2$ with weight $\rho^2/\bar r$---and it reduces to a parabola with peak at $\rho = \bar r / 2$ in the case where $\bm{x}(\infty) = \bm{0}$. Unlike Theorem~\ref{thm:bounded}, the bound is not in general tight. The bound can be extended to $\rho > \bar r$ and admits a closed-form worst case over all $\rho \ge 0$; we relegate these refinements to the GitHub companion.

The bounds in Theorem~\ref{thm:bounded} and Proposition~\ref{thm:unbounded} depend only on coarse global parameters---the Lipschitz constant $\Lip(L)$, the support diameter $\diam(\Xi)$ or its effective counterpart $\bar r$, and the endpoint norms $\|\bm{x}(0)\|_2$ and $\|\bm{x}(\infty)\|_2$. In particular, they do not reflect the curvature of the loss function or the conditioning of the data. We now specialize to a quadratic loss setting that admits sharper, data-dependent guarantees. In addition to an improved additive bound, this setting yields a \emph{multiplicative suboptimality} bound: a bound on the ratio of the worst-case additive suboptimality of the shrinkage path heuristic to that of the SAA--RO heuristic, which selects the better of the two endpoints $\bm{x}(0)$ and $\bm{x}(\infty)$. This bound shows that the shrinkage path heuristic strictly improves upon the SAA--RO heuristic whenever the data are not perfectly conditioned. This justifies exploring the full shrinkage path rather than restricting attention to its two endpoints.

\begin{assumption}[Quadratic Loss]\label{ass:quadratic}
    Let $\cZ := \{\bm{x}^\top \chg{\bm{\hat{\xi}}}_i : \|\bm{x}\|_2 \le \|\bm{x}(0)\|_2,\; i \in [N]\}$ collect the scalar arguments along the shrinkage path. In addition to Assumption~\ref{ass:apriori}, the following conditions hold.
    \begin{enumerate}
        \item[(i)] The support and feasible set are unconstrained: $\Xi = \R^k$ and $\cX (\Xi) = \R^n$.
        \item[(ii)] The loss $L : \R \to \R_+$ is differentiable, and is quadratic on $\cZ$, that is, $L(z) = z^2 + bz + c$ on $\cZ$ for some $b, c \in \R$.
        \item[(iii)] The empirical second-moment matrix $\bm{H} := \frac{1}{N}\sum_{i=1}^{N} \chg{\bm{\hat{\xi}}}_i \chg{\bm{\hat{\xi}}}_i^\top$ is positive definite, with eigenvalues $0 < \lambda_{\min} \le \cdots \le \lambda_{\max}$ and condition number $\kappa := \lambda_{\max} / \lambda_{\min}$.
    \end{enumerate}
\end{assumption}

Condition~\emph{(i)} reduces the problem geometry to a form where closed-form analysis is possible: $\Xi = \R^k$ collapses the WDRO problem to a regularized SAA problem with a Euclidean norm penalty, and $\cX (\Xi) = \R^n$ forces $\bm{x}(\infty) = \bm{0}$, so the shrinkage path reduces to the segment $\{\lambda \bm{x}(0) : \lambda \in [0,1]\}$. Condition~\emph{(ii)} is consistent with Assumption~\ref{ass:apriori}\emph{(ii)}, provided $L$ extends to a globally Lipschitz function on $\R$. This is satisfied, for example, by the Huber loss $L(z) = z^2$ for $|z| \le \delta$ and $L(z) = 2\delta |z| - \delta^2$ otherwise: as long as $\cZ \subseteq [-\delta, \delta]$, the loss is quadratic on $\cZ$ and globally Lipschitz with constant $2\delta$. Our results readily extend to losses with a continuous, bounded second derivative on $\cZ$, with $\lambda_{\min}$ and $\lambda_{\max}$ in the bounds replaced by the local strong convexity and smoothness constants of $L$ on $\cZ$; we omit the details for brevity.

The following theorem refines the additive bound of Proposition~\ref{thm:unbounded} under the additional structure of Assumption~\ref{ass:quadratic}.

\begin{theorem}[Additive Suboptimality; Quadratic Loss]\label{thm:quadratic}
    Let Assumption~\ref{ass:quadratic} hold. The worst-case additive suboptimality
    of the shrinkage path heuristic is bounded above by
    \begin{equation}\label{eq:quadratic-additive-bound}
        \lambda_{\max} \|\bm{x}(0)\|_2^{2} \frac{(\kappa-1)^2}{\kappa\,[\varphi^5(\kappa-1)+27]},
    \end{equation}
    where $\varphi = (\sqrt{5}+1)/2$ is the golden ratio. The bound is
    asymptotically tight as $\kappa \to 1^+$ and $\kappa \to \infty$.
\end{theorem}

The bound~\eqref{eq:quadratic-additive-bound} factors into two terms: $\lambda_{\max} \|\bm{x}(0)\|_2^{2}$ and a $\kappa$-dependent factor. The latter behaves like $(\kappa - 1)^2 / 27$ to leading order as $\kappa \to 1^+$, so the shrinkage path is provably near-optimal whenever the empirical second-moment matrix is well-conditioned; it increases monotonically with $\kappa$ and approaches $1/\varphi^5 \approx 0.0902$ in the severely ill-conditioned limit $\kappa \to \infty$. The appearance of the golden ratio reflects the limiting worst-case geometry, described in the following corollary.

\begin{corollary}[Worst-case Geometry]\label{cor:golden-ratio-geometry}
    Let Assumption~\ref{ass:quadratic} hold, and let $\rho_{\mathrm{gap}}$, $\rho^{\mathrm h}_0$, and $\rho^{\star}_0$ denote, respectively, a radius at which the suboptimality gap is maximized, the radius at which the shrinkage path solution reaches $\bm{0}$, and the smallest radius at which the WDRO solution reaches $\bm{0}$. As $\kappa \to \infty$, the worst-case bound in Theorem~\ref{thm:quadratic} is approached by instances for which
    \begin{align*}
        \rho_{\mathrm{gap}} : \rho^{\mathrm h}_0 : \rho^{\star}_0 \;\; \to \;\; 1 : \varphi : \varphi^2.
    \end{align*}
\end{corollary}

The corollary describes the configuration that produces the worst case of Theorem~\ref{thm:quadratic}, in which the two forms of shrinkage are hardest to reconcile. The WDRO problem shrinks eigendirections at non-uniform rates---in the limiting worst case, the $\lambda_{\min}$ direction collapses to zero while the $\lambda_{\max}$ direction shrinks by a factor of $1/\varphi$---whereas the shrinkage path heuristic applies a single uniform factor across all coordinates. The heuristic incurs its largest gap at the configuration where this mismatch is  most costly; there, the uniform factor $1/\varphi^2$ and the dominant per-coordinate factor $1/\varphi$ sum to one: $\tfrac{1}{\varphi} + \tfrac{1}{\varphi^2} = 1$. Thus, no single direction is responsible for the bound: it arises because one common shrinkage factor cannot match the WDRO solution in directions of very different magnitude.

We now turn to the multiplicative suboptimality of our heuristic.

\begin{theorem}[Multiplicative Suboptimality]\label{thm:endpoint}
    Let Assumption~\ref{ass:quadratic} hold. The multiplicative suboptimality of the shrinkage path heuristic relative to the SAA--RO heuristic \mbox{is bounded above by}
    \begin{equation}\label{eq:ratio-bound}
        \frac{(\kappa - 1)^2 (\kappa + 1)}{(\kappa^{3/2} + 1)^2},
    \end{equation}
    and the bound is strictly less than one whenever $\kappa \geq 1$.
\end{theorem}

The bound~\eqref{eq:ratio-bound} equals zero at $\kappa = 1$ and increases monotonically to one as $\kappa \to \infty$, so the shrinkage path offers the largest relative improvement over the SAA--RO heuristic for well- to moderately conditioned data. The bound is generally not tight: it relaxes the worst-case suboptimality of the SAA--RO heuristic to a lower bound, and a smaller denominator enlarges the ratio.

\begin{remark}[Connection to James--Stein Shrinkage]\label{rem:james-stein}
Under Assumption~\ref{ass:quadratic}, the solution of the shrinkage path heuristic admits the closed form
\[
    \hat{\bm{x}}(\rho)
    = \left[1 - \frac{\rho\,\|\bm{x}(0)\|_2}{2\,\bm{x}(0)^\top \bm{H}\, \bm{x}(0)}\right]_{+} \bm{x}(0),
\]
a non-negative scalar multiple of the SAA solution, with the multiplier shrinking from one toward zero as $\rho$ grows. This is formally reminiscent of the positive-part James--Stein estimator $\hat{\bm{\theta}}^{JS+} = [1 - (n-2)/\|\bm{Y}\|_2^2]_+\, \bm{Y}$ \citep{james1961estimation,baranchik1970}. In the isotropic case $\bm{H} = \bm{I}$, our shrinkage factor reduces to $1 - \rho/(2\|\bm{x}(0)\|_2)$; identifying $\bm{Y}$ with $\bm{x}(0)$, the heuristic and James--Stein shrinkage factors then coincide at $\rho = 2(n-2)/\|\bm{x}(0)\|_2$. The key difference is that the degree of shrinkage is governed by the ambiguity radius $\rho$ rather than the sample size, suggesting that our shrinkage path heuristic implements an ambiguity-driven analogue of classical shrinkage~regularization.
\end{remark}

\section{A Posteriori Performance Guarantees}\label{sec:aposteriori}

We next develop computable a posteriori certificates for the suboptimality of our shrinkage path heuristic. Since the shrinkage path is a subset of the WDRO feasible set, the objective value $J^\star$ returned by Algorithm~\ref{alg:iph} \emph{upper} bounds the optimal value of Problem~\eqref{prob:main}. This section complements our upper bound with a matching \emph{lower} bound. Our construction mirrors the primal shrinkage path on the dual side: a minimax interchange first recasts the optimal value of Problem~\eqref{prob:main} as a maximization over distributions $\Qprob$ in the Wasserstein ball $\Amb(\Pnom)$, and this maximization is then restricted to the line segment joining the empirical distribution $\Pnom$ to a worst-case distribution at $\rho=\infty$. To secure the minimax interchange and the existence of a worst-case distribution at $\rho=\infty$, this section imposes the following regularity assumption on Problem~\eqref{prob:main}.

\begin{assumption}[A Posteriori Regularity]\label{ass:aposteriori}
Problem~\eqref{prob:main} satisfies the following conditions.
\begin{enumerate}
    \item[(i)] The SAA objective $\bm{x}\mapsto \mathbb{E}_{\Pnom} \big[ \Loss(\bm{x},\bm{\tilde{\xi}}) \big]$ is coercive on $\mathcal{X}$.
    \item[(ii)] There exist a neighborhood $\mathcal{U}$ of the RO solution $\bm{x}(\infty)$ and a compact set $\mathcal{K}\subseteq\Xi$ such that
    \[
        \sup_{\bm{\xi}\in\Xi}\Loss(\bm{x},\bm{\xi})
        =
        \max_{\bm{\xi}\in\mathcal{K}}\Loss(\bm{x},\bm{\xi})
        \qquad\forall\,\bm{x}\in\mathcal{U}\cap\mathcal{X}.
    \]
\end{enumerate}
\end{assumption}

Recall that a function $f:\mathcal{X}\to\R\cup\{+\infty\}$ is \emph{coercive} on $\mathcal{X}$ if all its sublevel sets on $\mathcal{X}$ are bounded, or equivalently, if $f(\bm{x}_k)\to+\infty$ along every sequence $\{\bm{x}_k\}\subseteq\mathcal{X}$ with $\|\bm{x}_k\|\to\infty$. Condition~\emph{(i)} secures the minimax interchange between the decisions $\bm{x}$ and the distributions $\Qprob$ in Problem~\eqref{prob:main}, while Condition~\emph{(ii)} guarantees the existence of a worst-case distribution at $\rho = \infty$ that is supported on finitely many atoms. The latter condition holds, for example, whenever $\Xi$ is compact.

Denote by $J^\star(\rho)$ the optimal value of Problem~\eqref{prob:main} at radius $\rho$. We henceforth denote by $\overline{J}(\rho)$ the optimum of Problem~\eqref{prob:main} restricted to the full shrinkage path $\widehat{\mathcal{X}}(\Xi)$, with the shrinkage parameter $\lambda$ ranging over the full interval $[0,1]$. In the notation of Section~\ref{sec:apriori}, $J^\star(\rho)=f_\rho^\star$ is the WDRO optimum and $\overline{J}(\rho)=\hat f_\rho$ is the value of the shrinkage path heuristic. The objective value returned by Algorithm~\ref{alg:iph} is a finite-grid approximation of $\overline{J} (\rho)$, since Algorithm~\ref{alg:iph} searches only over the grid $\Lambda\subset[0,1]$. The full-path optimum itself is given by
\begin{equation*}
    \overline{J}(\rho) \;\coloneqq\; \min_{\bm{x}\in \widehat{\mathcal{X}}(\Xi)}\,\sup_{\Qprob\in\Amb(\Pnom)} \; \E_{\Qprob}\big[\Loss(\bm{x},\chg{\bm{\tilde{\xi}}})\big]\;\geq\; J^\star(\rho),
\end{equation*}
where the inequality follows from $\widehat{\mathcal{X}}(\Xi)\subseteq\mathcal{X}(\Xi)$. To certify the resulting approximation gap $\overline{J}(\rho)-J^\star(\rho)$ without access to $J^\star(\rho)$, we construct a matching \emph{lower} bound $\underline{J}(\rho)$ via a dual shrinkage path that mirrors the primal one. The construction is motivated by the minimax interchange
\begin{equation}\label{eq:minimax}
    J^\star(\rho)
    \;=\;
    \min_{\bm{x}\in\mathcal{X}(\Xi)} \, \sup_{\Qprob\in\Amb(\Pnom)} \;
    \E_\Qprob\big[\Loss(\bm{x},\chg{\bm{\tilde{\xi}}})\big]
    \;=\;
    \sup_{\Qprob\in\Amb(\Pnom)} \, \inf_{\bm{x}\in\mathcal{X}(\Xi)} \;
    \E_\Qprob\big[\Loss(\bm{x},\chg{\bm{\tilde{\xi}}})\big],
\end{equation}
which holds for every $\rho \in \mathbb{R}_+ \cup \{ +\infty \}$ under Assumption~\ref{ass:aposteriori}\emph{(i)} \citep[Corollary~5.16]{kuhn2025distributionally}. Note that the inner infimum on the right-hand side may not be attained because,  for a fixed $\Qprob$, the map $\bm{x}\mapsto\E_\Qprob \big[\Loss(\bm{x},\chg{\bm{\tilde{\xi}}}) \big]$ need not be coercive. This min--sup $=$ sup--inf identity instantiates the \emph{primal worst equals dual best} principle \citep{beck2009duality, zhen2025unified}; the interchange and the existence of an associated worst-case (least-favorable) distribution are well studied in this setting \citep{blanchet2019quantifying}.

The dual shrinkage path is the line segment between the two saddle distributions at $\rho=0$ and $\rho=\infty$. At $\rho=0$, the Wasserstein ball $\Amb(\Pnom)$ collapses to the empirical distribution $\Pnom$. At $\rho=\infty$, we fix any optimizer of the dual maximization in~\eqref{eq:minimax} as $\rho\to\infty$, namely
\begin{equation}\label{eq:Qinf}
   \Qprob(\infty)\in \mathop{\arg\max}_{\Qprob\in\mathcal{P}(\Xi)} \,
    \inf_{\bm{x}\in\mathcal{X}(\Xi)} \;
    \E_\Qprob\big[\Loss(\bm{x},\chg{\bm{\tilde{\xi}}})\big],
\end{equation}
and call it the \emph{RO distribution} $\Qprob(\infty)$, paralleling the RO solution $\bm{x}(\infty)$ on the primal side. The GitHub companion establishes that under Assumption~\ref{ass:aposteriori}, such an optimizer exists and can be chosen to be supported on at most $n+1$ atoms. Fixing $\Qprob(\infty)$ in turn fixes the dual shrinkage path; although both are determined only up to the choice of optimizer, we henceforth speak of \emph{the} RO distribution and \emph{the} dual shrinkage path. The latter is defined as
\begin{equation*}
    \widehat{\mathcal{B}}(\Pnom)
    \;\coloneqq\;
    \big\{(1-\lambda) \Pnom+\lambda \Qprob(\infty) \, : \, \lambda\in[0,1]\big\}.
\end{equation*}
The primal and dual shrinkage paths are symmetric: the primal path $\widehat{\mathcal{X}}(\Xi)$ is the line segment between the SAA solution $\bm{x}(0)$ and the RO solution $\bm{x}(\infty)$, while the dual path $\widehat{\mathcal{B}}(\Pnom)$ is the line segment between the empirical distribution $\Pnom$ and the RO distribution $\Qprob(\infty)$. By the minimax interchange~\eqref{eq:minimax}, the endpoints of the two paths pair up: $(\bm{x}(0),\Pnom)$ at $\rho=0$ and $(\bm{x}(\infty),\Qprob(\infty))$ at $\rho=\infty$. Both pairs are in fact Nash equilibria of the game in which the decision maker minimizes and nature maximizes $\E_{\Qprob}[\Loss(\bm{x},\chg{\bm{\tilde{\xi}}})]$: at $\rho=0$ the ambiguity set collapses to the singleton $\{\Pnom\}$, leaving neither player a profitable deviation, while at $\rho=\infty$ the minimax interchange~\eqref{eq:minimax} equates the min--sup and max--inf values, so their outer optimizers $\bm{x}(\infty)$ and $\Qprob(\infty)$ are mutual best responses.

Restricting the dual maximization in~\eqref{eq:minimax} from the Wasserstein ball $\Amb(\Pnom)$ to its intersection with the dual shrinkage path $\widehat{\mathcal{B}}(\Pnom)$ yields the matching a posteriori lower bound
\begin{equation}\label{eq:lower-bound}
    \underline{J}(\rho)
    \; \coloneqq \;
    \max_{\Qprob\in\widehat{\mathcal{B}}(\Pnom)\,\cap\,\Amb(\Pnom)} \,
    \inf_{\bm{x}\in\mathcal{X}(\Xi)} \;
    \E_\Qprob\big[\Loss(\bm{x},\chg{\bm{\tilde{\xi}}})\big]
    \;\le\; J^\star(\rho).
\end{equation}
The a posteriori gap $\overline{J}(\rho)-\underline{J}(\rho)$ then provides a suboptimality certificate for the shrinkage path heuristic in Algorithm~\ref{alg:iph}. 

The lower bound $\underline{J}(\rho)$ admits a reduction to a finite-scenario stochastic optimization problem.
\begin{theorem}[Dual Shrinkage Closed Form]\label{thm:lb-algorithm}
Whenever $\Pnom \neq \Qprob(\infty)$, the a posteriori lower bound~\eqref{eq:lower-bound} admits the finite-scenario reformulation
\begin{equation}\label{eq:Qhat}
    \underline{J}(\rho) \;=\; \min_{\bm{x}\in\mathcal{X}(\Xi)}\,
    \E_{\widehat{\Qprob}(\rho)}\big[\Loss(\bm{x},\chg{\bm{\tilde{\xi}}})\big]
    \qquad \text{with} \qquad
    \widehat{\Qprob}(\rho) \;=\; \mu(\rho) \Pnom \;+\; \big(1 - \mu(\rho)\big) \Qprob(\infty)
\end{equation}
under the mixture $\mu(\rho) \coloneqq \big[1 - \rho / W(\Pnom,\Qprob(\infty))\big]_+$. In the degenerate case where $\Pnom = \Qprob(\infty)$, the dual shrinkage path collapses to $\{\Pnom\}$, and~\eqref{eq:Qhat} holds with $\widehat{\Qprob}(\rho) = \Pnom$ for every $\rho \ge 0$.
\end{theorem}

Theorem~\ref{thm:lb-algorithm} identifies $\widehat{\Qprob}(\rho)$ as the distribution on the dual shrinkage path that lies as far from $\Pnom$ as the radius permits: it sits on the boundary of the Wasserstein ball, $W(\widehat{\Qprob}(\rho),\Pnom)=\rho$, whenever $\rho \le W(\Pnom,\Qprob(\infty))$, and coincides with $\Qprob(\infty)$ otherwise; equivalently, $W(\widehat{\Qprob}(\rho),\Pnom) = \min\{\rho,\,W(\Pnom,\Qprob(\infty))\}$. Geometrically, the dual shrinkage path $\widehat{\mathcal{B}}(\Pnom)$ stays fixed while the Wasserstein ball $\Amb(\Pnom)$ expands with $\rho$, so $\widehat{\Qprob}(\rho)$ slides along the path at the ball's boundary until it reaches $\Qprob(\infty)$. Like the primal shrinkage path, the dual path does not vary with $\rho$: evaluating $\underline{J}(\rho)$ requires only a one-off RO distribution construction, a one-off Wasserstein-distance evaluation between two finitely supported distributions, and a single finite-scenario stochastic optimization per radius, as summarized in Algorithm~\ref{alg:lower-bound}. Sweeping a grid of radii therefore traces the entire lower-bound curve at negligible additional cost per radius. Together with $\overline{J}(\rho)$, this brackets $J^\star(\rho)$ from above and below, and so certifies the quality of the heuristic at every radius without solving Problem~\eqref{prob:main} exactly.

\begin{algorithm}[h!]
\caption{A Posteriori Lower Bound Computation}
\label{alg:lower-bound}
\begin{algorithmic}[1]
\Require Empirical law $\Pnom := \tfrac1N \sum_{i=1}^N \delta_{\chg{\bm{\hat{\xi}}}_i}$ with samples $\{\chg{\bm{\hat{\xi}}}_i\}_{i=1}^N$; radius $\rho \ge 0$.
\State Construct the RO distribution $\Qprob(\infty)$ as defined in~\eqref{eq:Qinf}.
\State Compute $W(\Pnom,\Qprob(\infty))$.
\State Form the dual shrinkage distribution $\widehat{\Qprob}(\rho)$ according to~\eqref{eq:Qhat}.
\State Solve $\underline{J}(\rho) \gets \min_{\bm{x}\in\mathcal{X}(\Xi)}\E_{\widehat{\Qprob}(\rho)}\!\big[\Loss(\bm{x},\chg{\bm{\tilde{\xi}}})\big]$.
\State \Return $\underline{J}(\rho)$.
\end{algorithmic}
\end{algorithm}

The following theorem establishes further structural properties shared by $\overline{J}(\rho)$, $J^\star(\rho)$, and $\underline{J}(\rho)$.

\begin{theorem}[A Posteriori Guarantee]\label{thm:posteriori}
The a posteriori upper bound $\overline{J}(\rho)$, the WDRO optimum $J^\star(\rho)$, and the a posteriori lower bound $\underline{J}(\rho)$ satisfy:
\begin{enumerate}
    \item[(i)] \emph{\textbf{Monotonicity:}} All three values are nondecreasing in $\rho$.
    \item[(ii)] \emph{\textbf{Concavity:}} All three values are concave in $\rho$.
    \item[(iii)] \emph{\textbf{Endpoint tightness:}} $\overline{J}(\rho) = J^\star(\rho) = \underline{J}(\rho)$ at $\rho = 0$ and for every $\rho \ge W(\Pnom, \Qprob(\infty))$.
\end{enumerate}
\end{theorem}

Theorem~\ref{thm:posteriori} 
shows that the a posteriori bounds and the exact WDRO value share the same shape: all three start from $J(0)$ (the optimal value of the SAA problem~\eqref{eq:extremes:saa}), rise concavely with $\rho$, and saturate at $J(\infty)$ (the optimal value of the RO problem~\eqref{eq:extremes:ro}) by $\rho = W(\Pnom, \Qprob(\infty))$.

Algorithm~\ref{alg:lower-bound} requires constructing the RO distribution $\Qprob(\infty)$, defined in~\eqref{eq:Qinf}. The following proposition provides three such constructions, with details deferred to Section \ref{appsec:proofs}.

\begin{proposition}[RO Distribution Construction]\label{prop:saddle-summary}
Under the standing assumptions, a finitely supported RO distribution $\Qprob(\infty)$ can be recovered from a finite-dimensional convex problem in each of the following settings.
\begin{enumerate}
  \item[(i)] \emph{Convex loss, polyhedral support.} The support $\Xi$ is a pointed polyhedron, and $\Loss(\bm{x},\cdot)$ is convex on $\Xi$ with nonpositive recession along every extreme ray for every $\bm{x}\in\mathcal{X}(\Xi)$.
  \item[(ii)] \emph{Finite-max refinement.} In addition to~(i), $\mathcal{X}(\Xi)$ is polyhedral and $\Loss=\max_{l\in[L]}\Loss_l$, with each $\Loss_l(\cdot,\bm{\xi})$ convex and differentiable. The resulting distribution has at most $n+1$ atoms.
  \item[(iii)] \emph{Convex--concave loss.} The support $\Xi$ is compact and $\Loss=\max_{l\in[L]}\Loss_l$, where every $\Loss_l$ is closed convex with full domain in $\bm{x}$ and concave and upper semicontinuous in $\bm{\xi}$ on $\Xi$.
\end{enumerate}
\end{proposition}

Construction~\emph{(i)} yields an RO distribution that is supported on the vertices of $\Xi$. Construction~\emph{(ii)} specializes~\emph{(i)} and obtains $\Qprob(\infty)$ from any basic feasible solution of a linear program by exploiting the finite-max differentiable structure of the loss. Construction~\emph{(iii)}, finally, recovers $\Qprob(\infty)$ from a Fenchel dual of the RO problem~\eqref{eq:extremes:ro} and applies to the finite-max convex--concave structure underlying tractable convex reformulations of Wasserstein DRO~\citep{mohajerin2018data,shafieezadehabadeh2025nash}. The constructions can differ in support size: the GitHub companion proves that an RO distribution on at most $n+1$ atoms always exists and describes the post-processing routine used to sparsify constructions~\emph{(i)} and~\emph{(iii)} to $n+1$ atoms. Construction~\emph{(ii)} meets this bound directly.

\section{Numerical Experiments}\label{sec:numerics}

In this section, we evaluate the effectiveness of our shrinkage path heuristic in both a multi-item newsvendor (Section \ref{ssec:newsvendor}) and an appointment scheduling setting (Section \ref{ssec:appointmentscheduling}). All experiments were conducted on a \chg{high-performance} computing environment, which hosts AMD EPYC 7742 processors with 64 GB of RAM and runs Python~3.11, CVXPY~1.7.3, MOSEK~11.0.29, and Gurobi~12.0.3. All reported runtimes correspond to single-threaded execution. All solver parameters are set to their default values except where explicitly stated otherwise. We evaluate our shrinkage heuristic along three dimensions: the quality of its a posteriori guarantees, its out-of-sample performance compared to unregularized and regularized SAA \citep{kleywegt2002sample}, Wasserstein DRO \citep{kuhn2019wasserstein}, and Sinkhorn DRO \citep{wang2025sinkhorn}, and its scalability relative to these methods. Supplementary implementation details are provided in the GitHub companion.

\subsection{Multi-Item Newsvendor}\label{ssec:newsvendor}

Consider the multi-item newsvendor problem, a classical inventory control model in which the decision maker selects order quantities $\bm x$ under uncertain demand $\chg{\bm{\tilde{\xi}}}$ supported on the box $\Xi=[\bm 0,\bm U]$ for some $\bm U\in\mathbb R^n_{++}$, and then incurs holding costs for excess inventory and backlogging costs for unmet demand \citep{zipkin2000foundations}. Given an empirical distribution $\chg{\Pnom}=\frac1N\sum_{j=1}^N\delta_{\chg{\bm{\hat{\xi}}}_j}$ and a \chg{type-1} Wasserstein ball with radius $\rho$ and $\ell_2$ ground metric, we benchmark the performance of different methods for (approximately) solving the WDRO newsvendor problem \citep{mohajerin2018data,gao2023distributionally}
\begin{equation}\label{prob:multiitemnewsvendor}
    \begin{aligned}
        & \mathop{\text{minimize}}_{\bm x} && \sup_{\Qprob \in \Amb(\Pnom)} \E_{\Qprob} \big[\bm h^\top[\bm x-\chg{\bm{\tilde{\xi}}}]_+ + \bm b^\top[\chg{\bm{\tilde{\xi}}}-\bm x]_+\big] \\
        & \text{subject to} && \bm 0 \le \bm x \le \bm V.
    \end{aligned}
\end{equation}
where $[\,\cdot\,]_+ := \max\{\,\cdot\,,0\}$ is applied componentwise, and $\bm h,\bm b,\bm V \in \mathbb R_{++}^n$ denote the marginal holding costs, the marginal backlogging costs, and the order capacities, respectively.

The newsvendor problem is a useful diagnostic because both endpoints of the shrinkage path are available in closed form, while the full WDRO problem remains practically tractable. The \emph{critical ratio} $q_i:=b_i/(h_i+b_i)\in(0,1)$---the cost-optimal service level (order fractile) of the classical newsvendor---increases in the backlogging cost $b_i$ and decreases in the holding cost $h_i$, so costlier shortages call for higher order quantities. Let $\widehat\xi_{(k),i}$ denote the $k$th order statistic of the $i$th demand coordinate, i.e., the $k$th smallest of its $N$ samples. Then, at $\rho=0$, the SAA endpoint is the empirical $q_i$-quantile of demand---the cost-optimal service level of the nominal newsvendor, clipped at capacity,
$
        x_i(0)=\min\{\widehat \xi_{(\lceil q_iN\rceil),i},V_i\}, \ i\in[n].
$
As $\rho\to\infty$, the WDRO problem reduces to the robust problem
$
        \min_{\bm x\in[\bm 0,\bm V]}\sum_{i=1}^n \max\{h_ix_i,b_i(U_i-x_i)\},
$
whose solution is $x_i(\infty)=\min\{q_iU_i,V_i\}, \ i\in[n].$ Finally, for finite $\rho$, Problem~\eqref{prob:multiitemnewsvendor} admits an exact conic reformulation \citep[Corollary~5.1]{mohajerin2018data}: the loss can be written as the maximum of $2^n$ affine functions of $\bm\xi$, and the box support and $\ell_2$ ground metric yield a second-order cone program. For completeness, we provide this SOC reformulation in the GitHub companion.

Throughout this section, all experiments draw from a common master instance of $n=20$ items. We generate the cost and support parameters by tiling a length-five pattern four times: the holding costs are $h_i=1$ for all $i$, the backlogging costs $\bm b$ and support upper bounds $\bm U$ cycle through $(2,6,8,10,10)$ and $(10,8,6,4,2)$ respectively, and the order capacities are $V_i=0.6\,U_i$; each pattern thus repeats across the four consecutive blocks of five items. Demands are sampled independently as $\tilde\xi_i=U_i\tilde Z_i$ with $\tilde Z_i\sim\mathrm{Beta}(\alpha_i,\beta_i)$, where $(\alpha_i,\beta_i)$ cycles through $(0.2,2),(0.2,1),(0.2,3),(0.4,2),(0.3,2)$, and the empirical distribution $\chg{\Pnom}$ is built from $N=15$ i.i.d.\ samples. Experiments that use fewer than $20$ items take the data of the first $n$ items.

We first benchmark our a posteriori guarantees with $n=15$ items, evaluating the worst-case expected loss along a logarithmic grid of $30$ values of $\rho$ between $10^{-2}$ and $10^{2}$. Figure~\ref{fig:newsvendor_insample} (left) plots the a posteriori envelopes of Section \ref{sec:aposteriori} and the exact WDRO objective and demonstrates that our shrinkage heuristic is near-optimal across all radii and exact at the two endpoints. In addition, the right panel of Figure~\ref{fig:newsvendor_insample} reports the average per-$\rho$ runtime of the primal approximation, the dual approximation, and the exact WDRO solve.

\begin{figure}[H]
  \centering
  \begin{subfigure}{0.45\textwidth}
      \includegraphics[width=\textwidth]{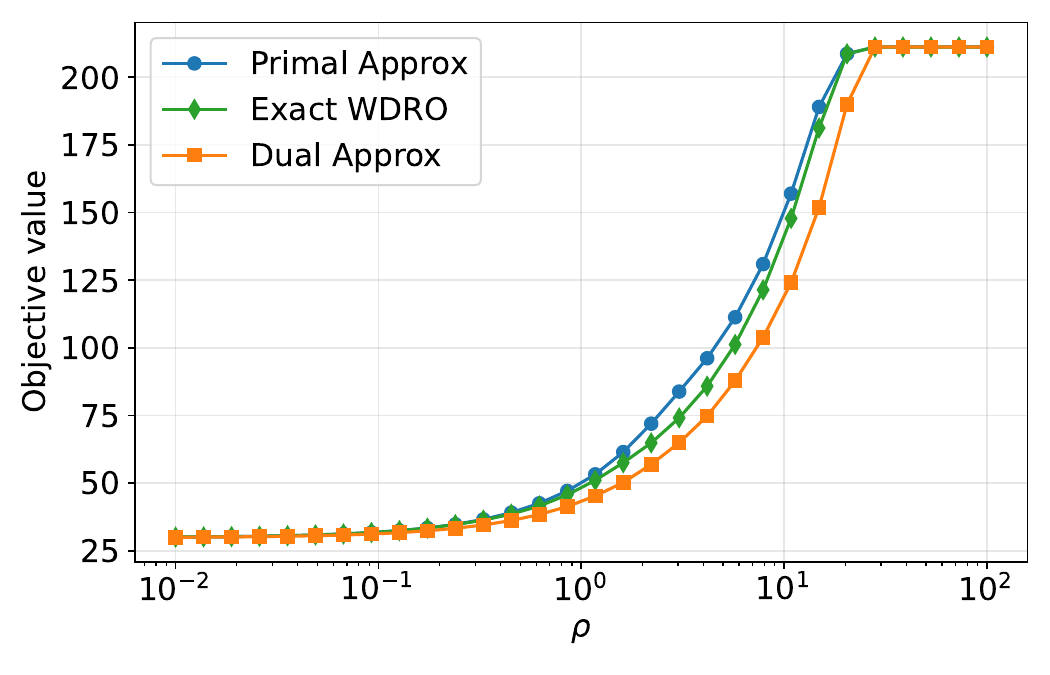}
  \end{subfigure}
    \begin{subfigure}{0.45\textwidth}
      \includegraphics[width=\textwidth]{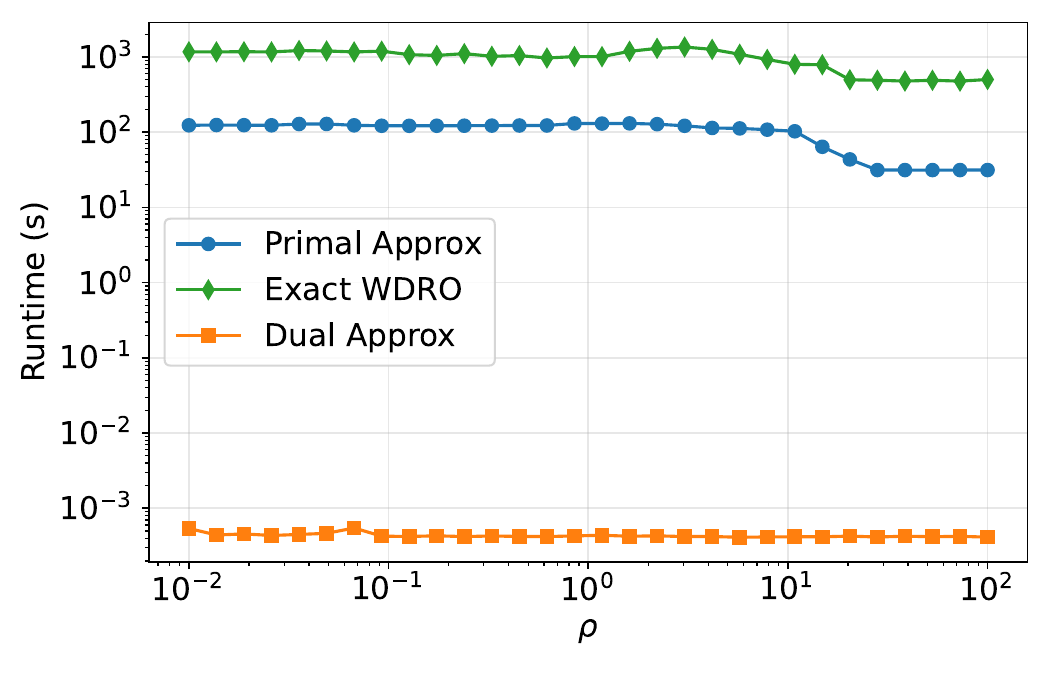}
  \end{subfigure}
  \caption{A~posteriori upper bound (primal approximation), exact WDRO objective, and a~posteriori lower bound (dual approximation) (left panel), and runtimes (right panel) for the 15-item newsvendor problem.}
  \label{fig:newsvendor_insample}
\end{figure}

Next, we evaluate the out-of-sample performance of the following methods. First, we include SAA. Second, we include two regularized SAA benchmarks, ``$\ell_1$-RSAA-CV'' and ``$\ell_2$-RSAA-CV'', which solve $\min_{\bm x\in\mathcal X}\frac1N\sum_{j=1}^N\Loss(\bm x,\chg{\bm{\hat{\xi}}}_j)+\nu\,\|\bm x\|_p^p$ for $p=1$ and $p=2$, respectively, where $\nu\ge0$ is a regularization weight tuned by five-fold cross-validation over $25$ values logarithmically spaced between $10^{-3}$ and $10^{2}$ and augmented with $\nu=0$. In view of the well-documented connections between Wasserstein DRO and norm regularization \citep{shafieezadehabadeh2019regularization,gao2024wasserstein}, these benchmarks test whether shrinking the decision per se already confers the benefits of distributional robustness, or whether the direction of shrinkage toward the RO solution matters. Third, we solve the exact WDRO reformulation introduced above, with $\rho$ tuned by five-fold cross-validation over $30$ values logarithmically spaced between $0.1$ and $15$ and augmented with $\rho=0$; we refer to this approach as ``WDRO-CV''. We additionally solve the exact WDRO problem with the cutting-plane algorithm described in the GitHub companion, denoted ``WDRO-CV (CP)''; since it returns the same solutions as WDRO-CV, we report its out-of-sample performance only for instances that WDRO-CV cannot solve within our runtime and memory budget, but report its runtime for all $n$. Fourth, we implement the Sinkhorn DRO method of \citet{wang2025sinkhorn}, which we refer to as ``SDRO-CV'', with the nominal radius $\bar\rho$ and entropic strength $\epsilon$ jointly tuned by five-fold cross-validation over the $22\times 14 = 308$ candidate $(\bar\rho,\epsilon)$ pairs of their default newsvendor configuration (implementation details are provided in the GitHub companion). Fifth, we evaluate our shrinkage path heuristic, with $\lambda$ tuned by five-fold cross-validation over $30$ uniformly spaced values in $[0,1]$; we refer to this approach as ``Shrinkage-CV''. Sixth, ``SAA-RO-CV'' uses five-fold cross-validation to choose the better of the SAA and RO decisions. Finally, we compute oracle counterparts of the shrinkage heuristic and of exact WDRO, with $\lambda$ and $\rho$ tuned directly on the test samples; these oracles, ``Shrinkage-Oracle'' and ``WDRO-Oracle'', give the best out-of-sample performance each approach can attain.


When the training and test data are drawn from the same distribution, SAA is already a strong benchmark for the data-driven newsvendor problem. The value of distributional robustness surfaces primarily under a distribution shift. Accordingly, to assess robustness to such a shift between training and deployment, the test data follow the coordinatewise contaminated mixture of \citet{hanasusanto2015distributionally},
\[
\mathbb P_{\mathrm{test}}
=\bigotimes_{i=1}^{n}
\Bigl[
(1-c) \mathbb P_{\mathrm{train},i}
+
c \mathrm{Unif}[0,U_i]
\Bigr],
\]
where $\mathbb P_{\mathrm{train},i}$ is the $i$th training marginal, $\mathrm{Unif}[0,U_i]$ is the uniform distribution on $[0,U_i]$, and the contamination level is $c=0.2$. We measure out-of-sample performance through the suboptimality of each method, that is, the percentage by which the expected cost of its decision on $50{,}000$ test samples exceeds the smallest expected cost attainable by any feasible decision under the same samples.

Figure~\ref{fig:newsvendor_oos} reports the out-of-sample suboptimality and average runtime across $100$ random seeds, as the number of items varies over $n\in\{8,10,12,15,20\}$ with the sample size fixed at $N=15$.
The left panel reports out-of-sample suboptimality. The boxes span interquartile ranges, the whiskers reach the most extreme seeds within $1.5$ interquartile ranges of the boxes, and seeds beyond the whiskers are omitted.
SAA and the two regularized SAA benchmarks perform worst across all problem sizes.
Both $\ell_1$-RSAA-CV and $\ell_2$-RSAA-CV underperform SAA; since $\nu=0$ belongs to the tuning grid, we attribute this shortfall to cross-validation error under the small sample size $N=15$.
Wasserstein DRO and Sinkhorn DRO are marginally better than the shrinkage path heuristic for $n \leq 15$.
At $n=20$, the shrinkage path heuristic outperforms the cutting-plane variant of Wasserstein DRO but remains marginally worse than Sinkhorn DRO.
The SAA-RO endpoint heuristic (SAA-RO-CV) also performs among the worst methods across all problem sizes.
Overall, for the multi-item newsvendor problem, the shrinkage path heuristic captures around $85\%$--$110\%$ of the median out-of-sample benefit of Wasserstein DRO over SAA.
The percentages are computed from the median relative suboptimalities reported in Figure~\ref{fig:newsvendor_oos}: at each $n$, the gap between SAA and the shrinkage path heuristic is divided by the gap between SAA and Wasserstein DRO, and the range is taken over $n$.

The right panel of Figure~\ref{fig:newsvendor_oos} reports the runtimes for each method. Here, the (barely visible) shaded bands cover one standard deviation on either side of the mean runtime.
Results for Wasserstein DRO and Sinkhorn DRO are reported only when the corresponding solve finishes within the $72$-hour runtime budget.
The exact reformulation of Wasserstein DRO exhausts our $64$~GB memory budget at $n=20$, where only its cutting-plane variant remains tractable.
The shrinkage path heuristic runs three orders of magnitude faster than either Wasserstein DRO variant or Sinkhorn DRO.
These computational savings allow the heuristic to achieve its competitive out-of-sample performance at a fraction of the computational cost and make it especially attractive in large-scale settings.
\begin{figure}[H]
  \centering
  \begin{subfigure}{0.6\textwidth}
  \includegraphics[width=\textwidth]{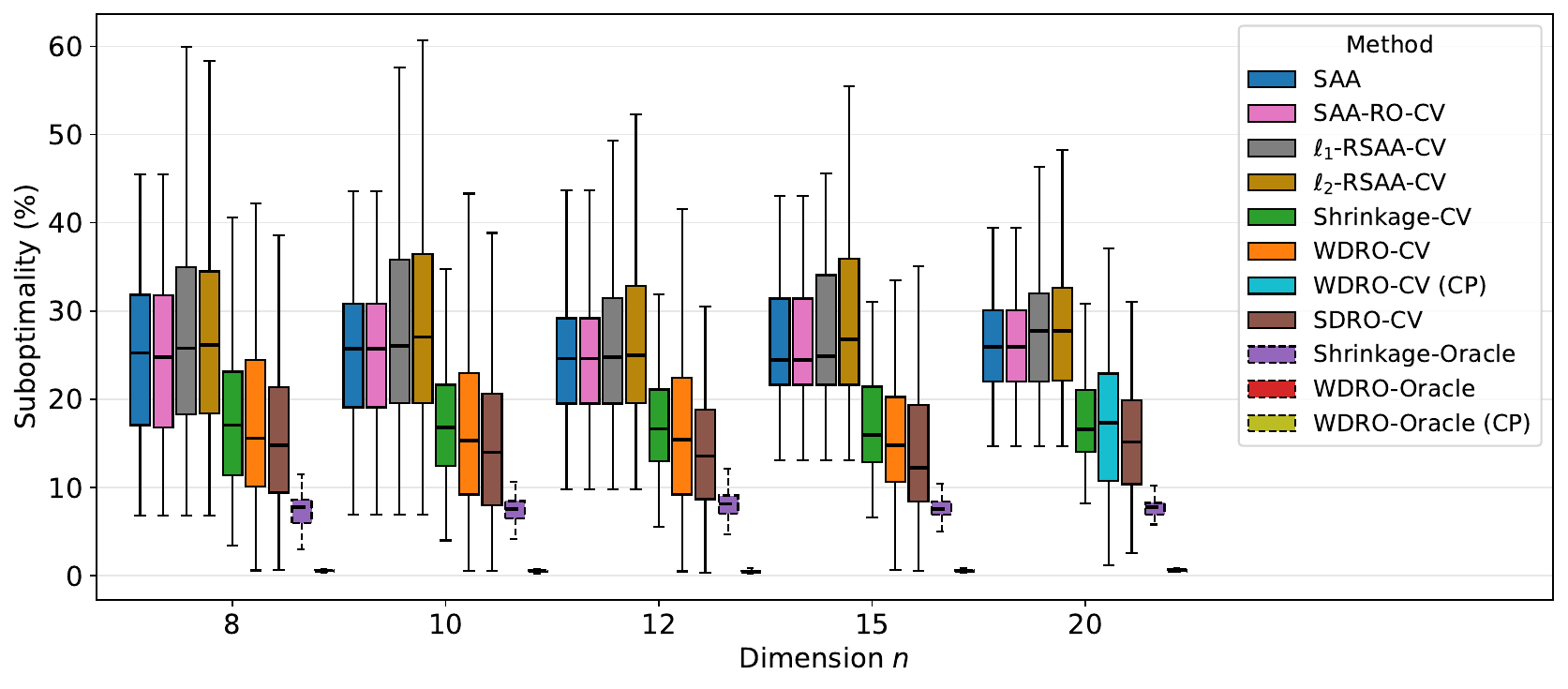}
  \end{subfigure}
  \begin{subfigure}{0.37\textwidth}
  \includegraphics[width=\textwidth]{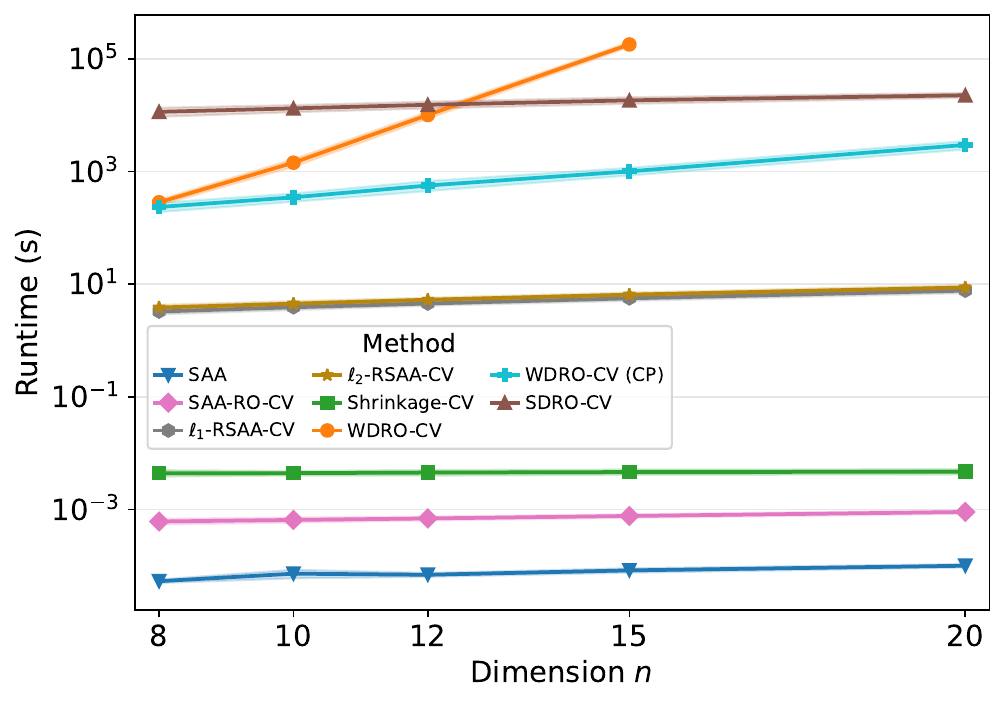}
\end{subfigure}
    
    \caption{Out-of-sample suboptimality (left) and average runtime (right) of the different methods on the multi-item newsvendor problem with $N=15$ samples, as the number of items varies over $n \in \{8, 10, 12, 15, 20\}$. The shrinkage path heuristic captures most of the out-of-sample benefit of Wasserstein or Sinkhorn DRO over SAA, at orders of magnitude less computational cost.}  \label{fig:newsvendor_oos}
\end{figure}

\subsection{Appointment Scheduling}\label{ssec:appointmentscheduling}

We consider an appointment scheduling problem in the spirit of \citet{denton2003sequential} and \citet{jiang2019data}. A service provider schedules $n$ consecutive appointment slots within a session of maximum duration $D$ and must fix the slot lengths $\bm{x}\in\mathbb{R}_+^n$ before observing the uncertain service durations $\chg{\bm{\tilde{\xi}}}$. The durations are supported on the box $\Xi=[\underline{\bm\xi},\overline{\bm\xi}]$. Given the empirical distribution $\chg{\Pnom}=\frac1N\sum_{j=1}^N\delta_{\chg{\bm{\hat{\xi}}}_j}$, we study the type-1 Wasserstein DRO problem with $\ell_2$ ground metric
\begin{subequations}
\begin{equation}
\label{eq:appointment-wdro}
\begin{aligned}
& \mathop{\text{minimize}}_{\bm x} && \sup_{\Qprob\in\Amb(\Pnom)} \E_{\Qprob}\big[\chg{\Loss}(\bm{x},\chg{\bm{\tilde{\xi}}})\big]\\
& \text{subject to} && \bm x\ge\bm 0,\quad \sum_{t\in[n]}x_t\leq D,
\end{aligned}
\end{equation}
where $\chg{\Loss}(\bm{x},\bm\xi)$ is the optimal value of the second-stage problem
\begin{equation}
\label{eq:appointment-recourse}
\begin{aligned}
& \mathop{\text{minimize}}_{\bm w,\bm s,o} && \rlap{$\displaystyle\sum_{t=2}^n\big(c_w w_t^2+c_s s_t^2\big)+c_o o^2$}\\
& \text{subject to} && w_1=s_1=0,\\
&&& w_{t+1}\ge w_t+\xi_t-x_t, && w_{t+1}\ge0 && \forall t\in[n-1],\\
&&& s_{t+1}\ge x_t-w_t-\xi_t, && s_{t+1}\ge0 && \forall t\in[n-1],\\
&&& o\ge w_n+\xi_n-x_n, && o\ge0.
\end{aligned}
\end{equation}
\end{subequations}

Here, the recourse variables $w_t$ and $s_t$ denote, respectively, the waiting time of patient $t$ and the server's idle time immediately before appointment $t$, while $o$ is the overtime beyond the scheduled session end $\sum_{t\in[n]}x_t$. The constants $c_w,c_s,c_o>0$ are the penalty weights on the squared waiting times, idle times, and overtime. Since the second-stage objective is jointly convex in $(\bm x,\bm\xi,\bm w,\bm s,o)$ and the constraints are jointly affine, $\chg{\Loss}(\bm x,\bm\xi)$ is jointly convex in $(\bm x,\bm\xi)$.

Unlike the newsvendor loss, which is a maximum of finitely many affine functions of $\bm\xi$, the appointment scheduling loss $\chg{\Loss}(\bm x,\bm\xi)$ is a maximum of \emph{infinitely many} affine functions. The extreme-point reduction that would render it finite fails because the second-stage dual is concave-quadratic, rather than linear, in its multipliers (cf.\ the GitHub companion). Since no finite, exact, convex reformulation is available, we solve the exact WDRO problem using the same cutting-plane algorithm as in the previous subsection, which we denote ``WDRO-CP''. As the worst-case evaluation oracle, $\textsc{WC-Eval}(\cdot;\rho)$, (cf. Section  \ref{sec:heuristics}) does not admit an exact finite reformulation either, we approximate each oracle call, including the separation problems within WDRO-CP, by a multi-start alternating maximization procedure (see the GitHub companion for more details). The shrinkage path endpoints, by contrast, remain practically tractable. Writing $\mathcal X=\{\bm x\in\mathbb R_+^n:\sum_{t=1}^n x_t\leq D\}$, the SAA and RO endpoints solve the convex quadratically constrained programs
\begin{align*}
\bm x(0) &\in \mathop{\arg\min}_{\bm x\in\mathcal X,\ \bm\tau\in\mathbb R^N}\ \Big\{\,\tfrac1N\textstyle\sum_{j=1}^N \tau_j \ :\ \chg{\Loss}(\bm x,\chg{\bm{\hat{\xi}}}_j)\le \tau_j\ \ \forall j\in[N]\,\Big\},\\[2pt]
\bm x(\infty) &\in \mathop{\arg\min}_{\bm x\in\mathcal X,\ \tau\in\mathbb R}\ \Big\{\,\tau \ :\ \chg{\Loss}(\bm x,\bm v_k)\le \tau\ \ \forall k\in[2^n]\,\Big\},
\end{align*}
where $\bm v_k$ ranges over the vertices of $\Xi$. Optimizing over these vertices suffices because $\bm\xi\mapsto\chg{\Loss}(\bm x,\bm\xi)$ is convex, so its maximum over the box $\Xi$ is attained at a vertex. Since enumerating all $2^n$ vertices is impractical, we compute $\bm x(\infty)$ using a cutting-plane method. In particular, we solve a relaxation over a subset of the vertex constraints and add the most violated one at each iteration, until no vertex constraint is violated.

Throughout this section, all experiments draw from a common master instance of $n=20$ appointments. We set the support bounds to $\underline{\bm\xi}=\bm 0$ and $\overline{\bm\xi}=(2,\ldots,2)$, the session length to $D=\sum_{i=1}^n(\underline{\xi}_i+\overline{\xi}_i)/2$, and the penalty weights to $c_w=1.0$ and $c_s=c_o=0.5$. Service durations are sampled as $\tilde{\xi}_t=\overline{\xi}_t\,\tilde Z_t$, where each $\tilde Z_t\in[0,1]$ follows a two-component Beta mixture. Its shape parameters $(\alpha^s_t,\beta^s_t)$ cycle through $\{(0.2,2),(0.2,1),(0.2,3),(0.4,2),(0.3,2)\}$, with $\alpha^s_t<\beta^s_t$ so that $\mathrm{Beta}(\alpha^s_t,\beta^s_t)$ favors short durations. We draw $\tilde{\bm Z}\sim(1-c) \mathrm{Beta}(\bm\alpha^s,\bm\beta^s)+c \mathrm{Beta}(\bm\beta^s,\bm\alpha^s)$ with $c=0.05$. The parameter-swapped second component favors long durations, so sessions have a 5\% chance of running long. The empirical distribution $\chg{\Pnom}$ is built from $N=10$ i.i.d.\ samples. Experiments that use fewer than $20$ appointments take the data of the first $n$.

We first construct the a posteriori envelopes of Section~\ref{sec:aposteriori} for the full $n=20$ instance and evaluate the worst-case expected loss along a geometric grid of $30$ values of $\rho$ between $10^{-3}$ and $30$. Figure~\ref{fig:appoint_insample} reports the resulting primal and dual approximation bounds on the optimal value of \eqref{eq:appointment-wdro} and demonstrates that our shrinkage path heuristic is near-optimal across all radii and exact at the two endpoints. For comparison, it also reports the objective value of the exact WDRO problem, solved by WDRO-CP, together with the runtime of all three methods, showing that the dual bound is two orders of magnitude cheaper than the primal.
\begin{figure}[H]
  \centering
\begin{subfigure}{0.45\textwidth}
      \includegraphics[width=\textwidth]{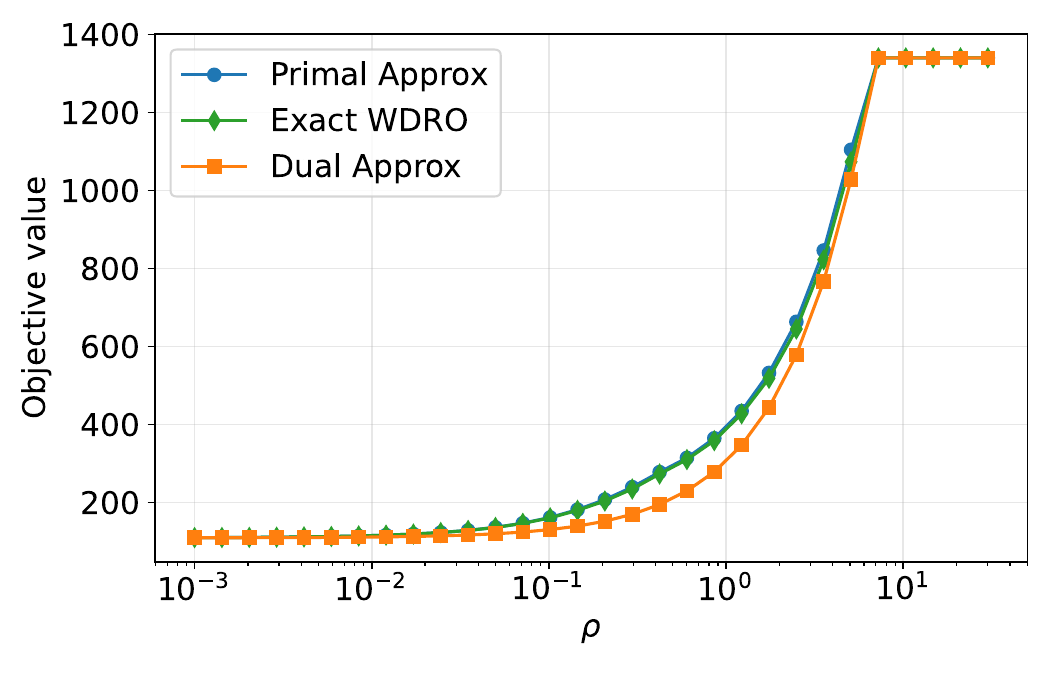}
  \end{subfigure}
    \begin{subfigure}{0.45\textwidth}
      \includegraphics[width=\textwidth]{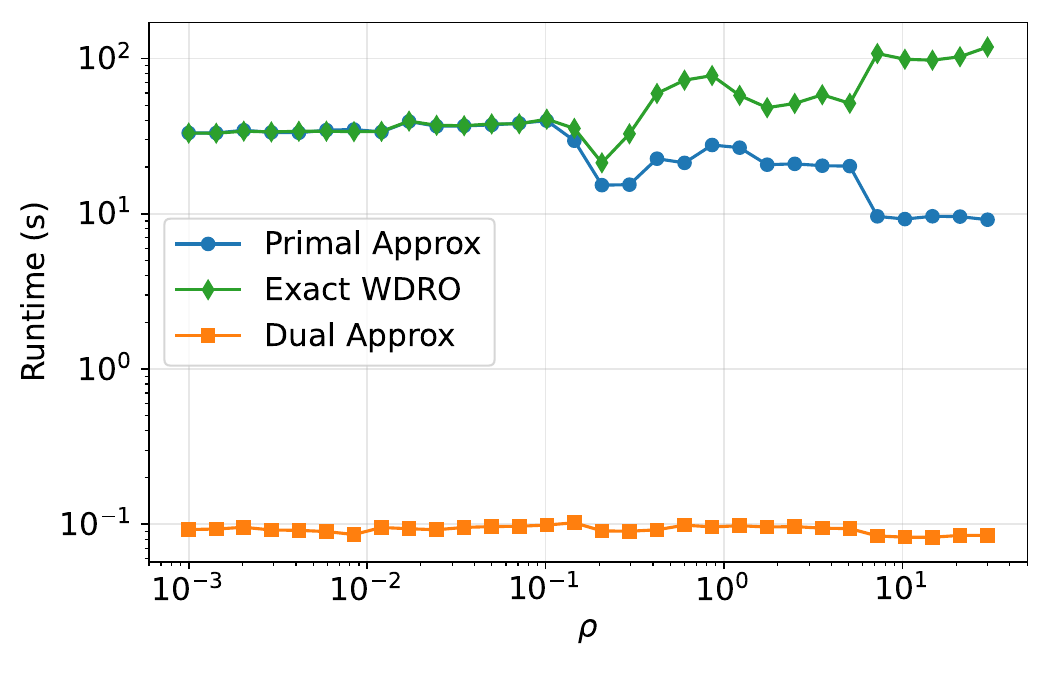}
  \end{subfigure}
  \caption{A~posteriori upper and lower bounds (left) and average runtimes (right) with varying $\rho$ for the appointment scheduling WDRO problem with $n=20$ appointments and $N=10$ samples.}
  \label{fig:appoint_insample}
\end{figure}

Next, we compare the out-of-sample performance of our shrinkage path heuristic against the same set of methods as in the newsvendor experiment (Section~\ref{ssec:newsvendor}), with Wasserstein DRO now solved by WDRO-CP. All hyperparameters are again tuned by five-fold cross-validation, and the number of appointments is swept over $n\in\{8,10,12,15,20\}$. The setup mirrors the newsvendor case except for the tuning grids: the shrinkage parameter $\lambda$ ranges over a linear grid of $25$ values in $[0,1]$, the Wasserstein radius $\rho$ over a logarithmic grid of $25$ values in $[0.001,5]$, and the Sinkhorn pair $(\bar\rho,\epsilon)$ over the $25$ combinations
$
    \bar\rho\in\{0.001,\,0.005,\,0.01,\,0.05,\,0.1\}, \ \epsilon\in\{0.01,\,0.05,\,0.1,\,0.5,\,1.0\}.
$
Training and test samples are drawn i.i.d.\ \mbox{from the same short/long mixture as above, with $N=10$ held fixed.}


Figure~\ref{fig:appoint_oos} (left panel) reports the out-of-sample relative suboptimality over $100$ random seeds, while Figure~\ref{fig:appoint_oos} (right panel) reports the average runtime, both as functions of $n$. The shrinkage path heuristic outperforms SAA, both regularized SAA benchmarks, and the SAA-RO endpoint heuristic at every $n$, while Wasserstein DRO and Sinkhorn DRO, in turn, outperform the shrinkage path heuristic. $\ell_1$-RSAA-CV and $\ell_2$-RSAA-CV are nearly indistinguishable from SAA at every $n$, indicating that cross-validation typically selects a vanishing regularization weight; shrinking the schedule toward $\bm 0$ thus confers no robustness benefit here, whereas shrinking toward the RO solution does. However, Wasserstein and Sinkhorn DRO are three to four orders of magnitude slower than the shrinkage path heuristic across all $n$. The appointment scheduling experiment thus complements the newsvendor one:  the shrinkage path heuristic yields competitive robust decisions that capture around $45\%$--$70\%$ of the median out-of-sample benefits of Wasserstein DRO over SAA, at a fraction of the computational cost. These percentages follow the same convention as for Figure~\ref{fig:newsvendor_oos}.

\begin{figure}[H]\centering
    \begin{subfigure}{0.6\textwidth}
    \includegraphics[width=\textwidth]{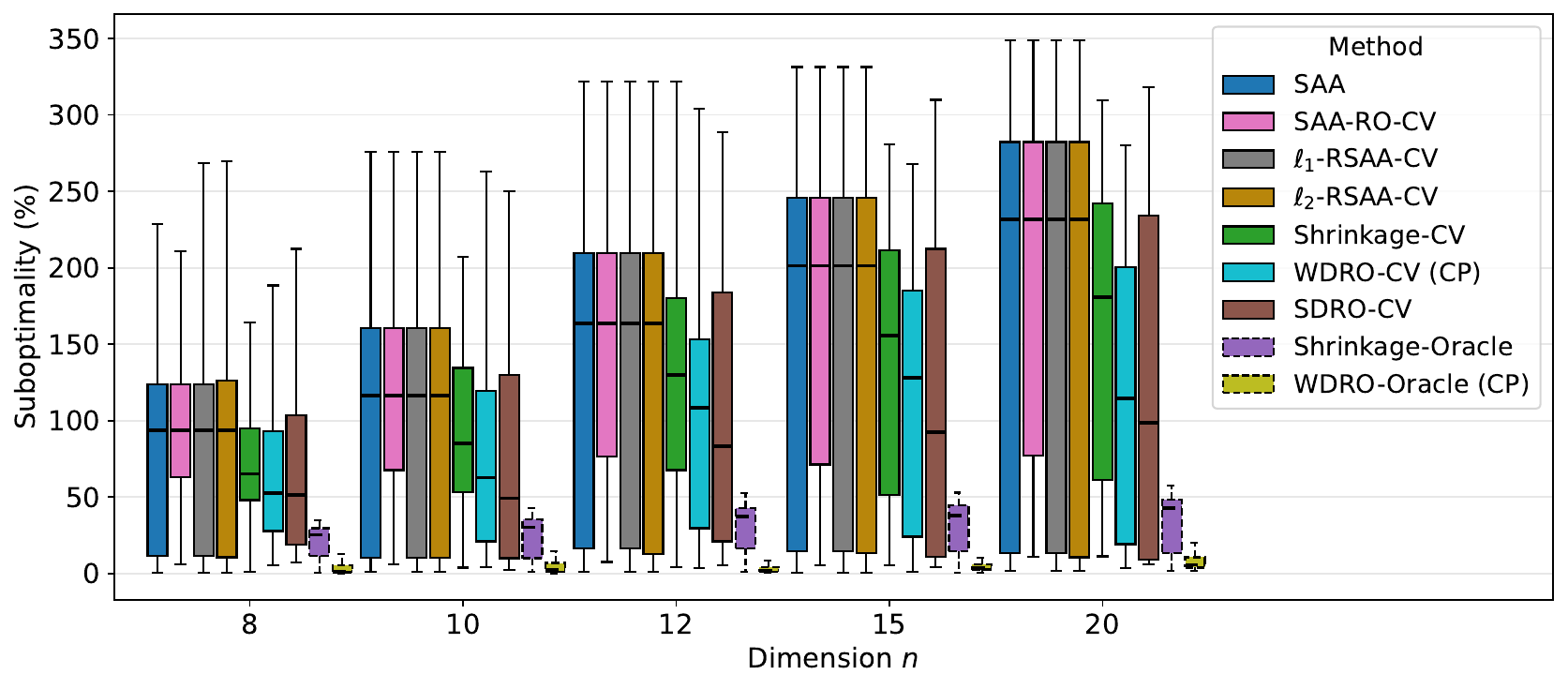}
    \end{subfigure}
    \begin{subfigure}{0.39\textwidth}
    \includegraphics[width=\textwidth]{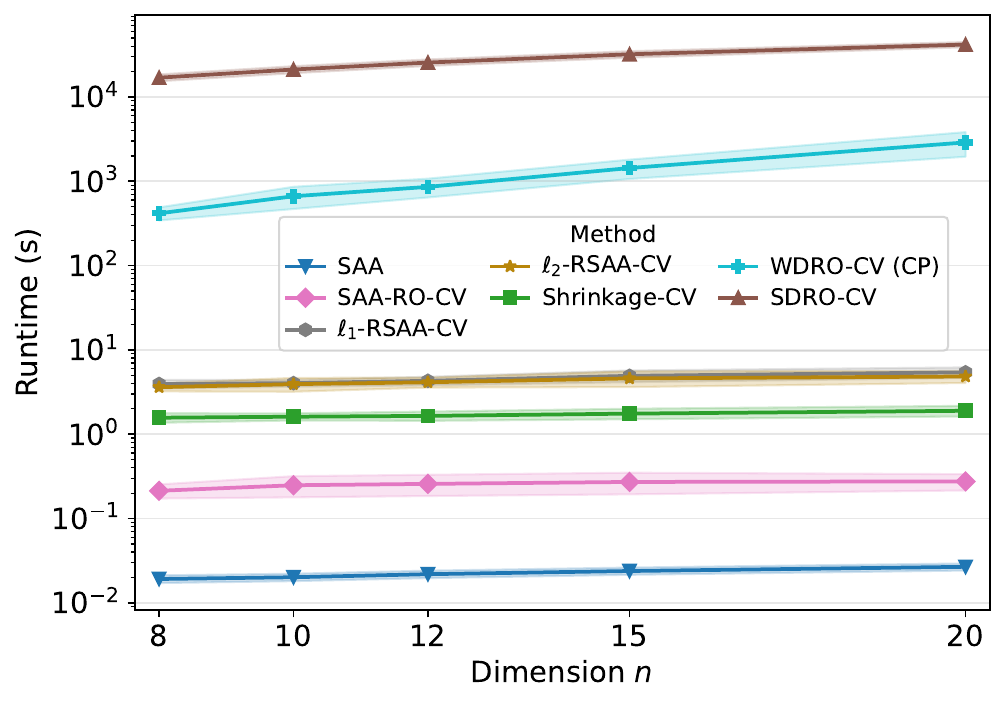}
    \end{subfigure}
  \caption{Out-of-sample suboptimality (left panel) and average runtime (right panel) of the different methods on the appointment scheduling problem with $N=10$ samples, as the number of appointments varies over $n\in\{8,10,12,15,20\}$. The shrinkage path heuristic captures a significant part of the out-of-sample benefit of Wasserstein and Sinkhorn DRO over SAA, at orders of magnitude less computational cost, whereas the regularized SAA benchmarks ($\ell_1$-RSAA-CV and $\ell_2$-RSAA-CV) barely improve on SAA. Boxes, whiskers, and shaded bands follow the conventions of Figure~\ref{fig:newsvendor_oos}.}
  \label{fig:appoint_oos}
\end{figure}

\section*{Use of Large Language Models}
Large language models (including Claude Fable $5$ and Chat-GPT-Pro $5.6$) were used to assist in writing the manuscript and implementing the numerical experiments. The authors have carefully reviewed and edited all generated content and take full responsibility for any remaining errors.

\theendnotes
\bibliography{ref}

\begin{thebibliography}{36}
\providecommand{\natexlab}[1]{#1}
\providecommand{\url}[1]{\texttt{#1}}
\expandafter\ifx\csname urlstyle\endcsname\relax
  \providecommand{\doi}[1]{doi: #1}\else
  \providecommand{\doi}{doi: \begingroup \urlstyle{rm}\Url}\fi

\bibitem[Anderson and Philpott(2022)]{anderson2022improving}
Edward Anderson and Andy Philpott.
\newblock Improving sample average approximation using distributional robustness.
\newblock \emph{INFORMS Journal on Optimization}, 4\penalty0 (1):\penalty0 90--124, 2022.

\bibitem[Baranchik(1970)]{baranchik1970}
Alvin~J. Baranchik.
\newblock A family of minimax estimators of the mean of a multivariate normal distribution.
\newblock \emph{The Annals of Mathematical Statistics}, 41\penalty0 (2):\penalty0 642--645, 1970.

\bibitem[Beck and Ben-Tal(2009)]{beck2009duality}
Amir Beck and Aharon Ben-Tal.
\newblock Duality in robust optimization: Primal worst equals dual best.
\newblock \emph{Operations Research Letters}, 37\penalty0 (1):\penalty0 1--6, 2009.

\bibitem[Ben-Tal et~al.(2009)Ben-Tal, El~Ghaoui, and Nemirovski]{ben2009robust}
Aharon Ben-Tal, Laurent El~Ghaoui, and Arkadi Nemirovski.
\newblock \emph{Robust Optimization}.
\newblock Princeton Series in Applied Mathematics. Princeton University Press, Princeton, NJ, 2009.

\bibitem[Bennouna et~al.(2023)Bennouna, Lucas, and Van~Parys]{bennouna2023certified}
Amine Bennouna, Ryan Lucas, and Bart Van~Parys.
\newblock Certified robust neural networks: Generalization and corruption resistance.
\newblock In \emph{International Conference on Machine Learning}, pages 2092--2112. PMLR, 2023.

\bibitem[Bertsimas and Dunn(2019)]{bertsimas2019machine}
Dimitris Bertsimas and Jack Dunn.
\newblock \emph{Machine Learning Under a Modern Optimization Lens}.
\newblock Dynamic Ideas LLC, Charlestown, MA, 2019.

\bibitem[Blanchet and Murthy(2019)]{blanchet2019quantifying}
Jose~H. Blanchet and Karthyek R.~A. Murthy.
\newblock Quantifying distributional model risk via optimal transport.
\newblock \emph{Mathematics of Operations Research}, 44\penalty0 (2):\penalty0 565--600, 2019.

\bibitem[Blanchet et~al.(2022)Blanchet, Murthy, and Zhang]{blanchet2022optimal}
Jose~H. Blanchet, Karthyek R.~A. Murthy, and Fan Zhang.
\newblock Optimal transport-based distributionally robust optimization: Structural properties and iterative schemes.
\newblock \emph{Mathematics of Operations Research}, 47\penalty0 (2):\penalty0 1500--1529, 2022.

\bibitem[Cheramin et~al.(2022)Cheramin, Cheng, Jiang, and Pan]{cheramin2022computationally}
Meysam Cheramin, Jianqiang Cheng, Ruiwei Jiang, and Kai Pan.
\newblock Computationally efficient approximations for distributionally robust optimization under moment and {W}asserstein ambiguity.
\newblock \emph{INFORMS Journal on Computing}, 34\penalty0 (3):\penalty0 1768--1794, 2022.

\bibitem[Denton and Gupta(2003)]{denton2003sequential}
Brian Denton and Diwakar Gupta.
\newblock A sequential bounding approach for optimal appointment scheduling.
\newblock \emph{IIE Transactions}, 35\penalty0 (11):\penalty0 1003--1016, 2003.

\bibitem[Esteban-P{\'e}rez and Morales(2023)]{esteban2023distributionally}
Adri{\'a}n Esteban-P{\'e}rez and Juan~M. Morales.
\newblock Distributionally robust optimal power flow with contextual information.
\newblock \emph{European Journal of Operational Research}, 306\penalty0 (3):\penalty0 1047--1058, 2023.

\bibitem[Gao and Kleywegt(2023)]{gao2023distributionally}
Rui Gao and Anton~J. Kleywegt.
\newblock Distributionally robust stochastic optimization with {W}asserstein distance.
\newblock \emph{Mathematics of Operations Research}, 48\penalty0 (2):\penalty0 603--655, 2023.

\bibitem[Gao et~al.(2024)Gao, Chen, and Kleywegt]{gao2024wasserstein}
Rui Gao, Xi~Chen, and Anton~J. Kleywegt.
\newblock {W}asserstein distributionally robust optimization and variation regularization.
\newblock \emph{Operations Research}, 72\penalty0 (3):\penalty0 1177--1191, 2024.

\bibitem[Hanasusanto et~al.(2015)Hanasusanto, Kuhn, Wallace, and Zymler]{hanasusanto2015distributionally}
Grani~A. Hanasusanto, Daniel Kuhn, Stein~W. Wallace, and Steve Zymler.
\newblock Distributionally robust multi-item newsvendor problems with multimodal demand distributions.
\newblock \emph{Mathematical Programming}, 152\penalty0 (1--2):\penalty0 1--32, 2015.

\bibitem[Hao and Zhang(2025)]{hao2025robust}
Hao Hao and Peter Zhang.
\newblock Robust paths: Geometry and computation.
\newblock \emph{arXiv preprint arXiv:2508.20039}, 2025.

\bibitem[James and Stein(1961)]{james1961estimation}
William James and Charles Stein.
\newblock Estimation with quadratic loss.
\newblock In \emph{Proceedings of the Fourth Berkeley Symposium on Mathematical Statistics and Probability, Volume 1: Contributions to the Theory of Statistics}, pages 361--379, Berkeley, CA, 1961. University of California Press.

\bibitem[Jiang et~al.(2019)Jiang, Ryu, and Xu]{jiang2019data}
Ruiwei Jiang, Minseok Ryu, and Guanglin Xu.
\newblock Data-driven distributionally robust appointment scheduling over {W}asserstein balls.
\newblock \emph{arXiv preprint arXiv:1907.03219}, 2019.

\bibitem[Kleywegt et~al.(2002)Kleywegt, Shapiro, and {Homem-de-Mello}]{kleywegt2002sample}
Anton~J. Kleywegt, Alexander Shapiro, and Tito {Homem-de-Mello}.
\newblock The sample average approximation method for stochastic discrete optimization.
\newblock \emph{SIAM Journal on Optimization}, 12\penalty0 (2):\penalty0 479--502, 2002.

\bibitem[Kuhn et~al.(2019)Kuhn, Mohajerin~Esfahani, Nguyen, and Shafieezadeh-Abadeh]{kuhn2019wasserstein}
Daniel Kuhn, Peyman Mohajerin~Esfahani, Viet~Anh Nguyen, and Soroosh Shafieezadeh-Abadeh.
\newblock {W}asserstein distributionally robust optimization: Theory and applications in machine learning.
\newblock In \emph{Operations Research \& Management Science in the Age of Analytics}, INFORMS TutORials in Operations Research, pages 130--166. INFORMS, 2019.

\bibitem[Kuhn et~al.(2025)Kuhn, Shafiee, and Wiesemann]{kuhn2025distributionally}
Daniel Kuhn, Soroosh Shafiee, and Wolfram Wiesemann.
\newblock Distributionally robust optimization.
\newblock \emph{Acta Numerica}, 34:\penalty0 579--804, 2025.

\bibitem[Li et~al.(2019)Li, Huang, and So]{li2019first}
Jiajin Li, Sen Huang, and Anthony Man-Cho So.
\newblock A first-order algorithmic framework for distributionally robust logistic regression.
\newblock \emph{Advances in Neural Information Processing Systems}, 32, 2019.

\bibitem[Luo and Mehrotra(2019)]{luo2019decomposition}
Fengqiao Luo and Sanjay Mehrotra.
\newblock Decomposition algorithm for distributionally robust optimization using {W}asserstein metric with an application to a class of regression models.
\newblock \emph{European Journal of Operational Research}, 278\penalty0 (1):\penalty0 20--35, 2019.

\bibitem[Mohajerin~Esfahani and Kuhn(2018)]{mohajerin2018data}
Peyman Mohajerin~Esfahani and Daniel Kuhn.
\newblock Data-driven distributionally robust optimization using the {W}asserstein metric: Performance guarantees and tractable reformulations.
\newblock \emph{Mathematical Programming}, 171\penalty0 (1--2):\penalty0 115--166, 2018.

\bibitem[Mutapcic and Boyd(2009)]{mutapcic2009cutting}
Almir Mutapcic and Stephen Boyd.
\newblock Cutting-set methods for robust convex optimization with pessimizing oracles.
\newblock \emph{Optimization Methods and Software}, 24\penalty0 (3):\penalty0 381--406, 2009.

\bibitem[Pardalos and Vavasis(1991)]{pardalos1991quadratic}
Panos~M. Pardalos and Stephen~A. Vavasis.
\newblock Quadratic programming with one negative eigenvalue is {NP}-hard.
\newblock \emph{Journal of Global Optimization}, 1\penalty0 (1):\penalty0 15--22, 1991.

\bibitem[Rahimian and Mehrotra(2022)]{rahimian2019distributionally}
Hamed Rahimian and Sanjay Mehrotra.
\newblock Frameworks and results in distributionally robust optimization.
\newblock \emph{Open Journal of Mathematical Optimization}, 3:\penalty0 1--85, 2022.
\newblock Article 4.

\bibitem[Shafiee et~al.(2026)Shafiee, Aolaritei, D{\"o}rfler, and Kuhn]{shafieezadehabadeh2025nash}
Soroosh Shafiee, Liviu Aolaritei, Florian D{\"o}rfler, and Daniel Kuhn.
\newblock Nash equilibria, regularization, and computation in optimal transport-based distributionally robust optimization.
\newblock \emph{Operations Research}, 74\penalty0 (3):\penalty0 1689--1709, 2026.

\bibitem[Shafieezadeh-Abadeh et~al.(2018)Shafieezadeh-Abadeh, Nguyen, Kuhn, and Mohajerin~Esfahani]{shafieezadeh2018wasserstein}
Soroosh Shafieezadeh-Abadeh, Viet~Anh Nguyen, Daniel Kuhn, and Peyman Mohajerin~Esfahani.
\newblock Wasserstein distributionally robust {K}alman filtering.
\newblock \emph{Advances in Neural Information Processing Systems}, 31, 2018.

\bibitem[Shafieezadeh-Abadeh et~al.(2019)Shafieezadeh-Abadeh, Kuhn, and Mohajerin~Esfahani]{shafieezadehabadeh2019regularization}
Soroosh Shafieezadeh-Abadeh, Daniel Kuhn, and Peyman Mohajerin~Esfahani.
\newblock Regularization via mass transportation.
\newblock \emph{Journal of Machine Learning Research}, 20\penalty0 (103):\penalty0 1--68, 2019.

\bibitem[Sion(1958)]{sion1958minimax}
Maurice Sion.
\newblock On general minimax theorems.
\newblock \emph{Pacific Journal of Mathematics}, 8\penalty0 (1):\penalty0 171--176, 1958.

\bibitem[Wang et~al.(2026)Wang, Gao, and Xie]{wang2025sinkhorn}
Jie Wang, Rui Gao, and Yao Xie.
\newblock Sinkhorn distributionally robust optimization.
\newblock \emph{Operations Research}, 74\penalty0 (3):\penalty0 1581--1603, 2026.

\bibitem[Watson(1955)]{watson1955serial}
Geoffrey~S. Watson.
\newblock Serial correlation in regression analysis. {I}.
\newblock \emph{Biometrika}, 42\penalty0 (3--4):\penalty0 327--341, 1955.

\bibitem[Xie and Ahmed(2018)]{xie2017distributionally}
Weijun Xie and Shabbir Ahmed.
\newblock Distributionally robust chance constrained optimal power flow with renewables: A conic reformulation.
\newblock \emph{IEEE Transactions on Power Systems}, 33\penalty0 (2):\penalty0 1860--1867, 2018.

\bibitem[Xin and Goldberg(2022)]{xin2022distributionally}
Linwei Xin and David~Alan Goldberg.
\newblock Distributionally robust inventory control when demand is a martingale.
\newblock \emph{Mathematics of Operations Research}, 47\penalty0 (3):\penalty0 2387--2414, 2022.

\bibitem[Zhen et~al.(2025)Zhen, Kuhn, and Wiesemann]{zhen2025unified}
Jianzhe Zhen, Daniel Kuhn, and Wolfram Wiesemann.
\newblock A unified theory of robust and distributionally robust optimization via the primal-worst-equals-dual-best principle.
\newblock \emph{Operations Research}, 73\penalty0 (2):\penalty0 862--878, 2025.

\bibitem[Zipkin(2000)]{zipkin2000foundations}
Paul~H. Zipkin.
\newblock \emph{Foundations of Inventory Management}.
\newblock McGraw-Hill, Boston, 2000.

\end{thebibliography}
\bibliographystyle{plainnat}


\ECSwitch
\renewcommand\thealgorithm{EC.\arabic{algorithm}}
\setcounter{algorithm}{0} 
\ECHead{Supplementary Material}

\section{Proofs of Main Theoretical Results}
\label{appsec:proofs}

\noindent \textbf{Proof of Theorem~\ref{thm:bounded}.} $\;$
We first show that the shrinkage path heuristic is exact whenever $\rho\ge\diam(\Xi)$. This implies both the validity and the tightness of the stated bound for all $\rho$. We then establish the validity of the bound for $\rho\in[0,\diam(\Xi)]$. Finally, we establish its pointwise tightness on the same interval by constructing, for each radius $\rho\in[0,\diam(\Xi)]$, an instance which attains the bound for that $\rho$. Throughout the proof, we denote by $f_\rho^\star$ the optimal value of the WDRO problem~\eqref{prob:main}, by $\hat{f}_\rho$ the objective value of the shrinkage path heuristic, and by $f_{\rho}(\bm{x})$ the objective value of~\eqref{prob:main} for a fixed decision $\bm{x}$.

First, observe that if $\mathrm{diam}(\Xi)=0$, then $\Xi$ is a singleton and the heuristic is exact for each radius. Therefore, we assume for the remainder of the proof that $\mathrm{diam}(\Xi)>0$. Next, for $\rho\ge\diam(\Xi)$, Assumption~\ref{ass:apriori}(i) and the definition of the Wasserstein distance then yield
\begin{equation*}
    W(\mathbb Q,\chg{\Pnom})
    \;\; = \;\;
    \inf_{\pi \in \Pi(\Qprob, \Pnom)} \; \int_{\XiSet \times \XiSet} \|\bm{\xi}-\bm{\xi}'\|_2 \, \pi(\mathrm{d}\bm{\xi}, \mathrm{d}\bm{\xi}')
    \;\; \le \;\;
    \sup_{\bm{\xi},\,\bm{\xi}'\in\Xi}\|\bm{\xi}-\bm{\xi}'\|_2
    \;\; = \;\;
    \diam(\Xi).
\end{equation*}
Hence, when $\rho\ge\diam(\Xi)$ the constraint $W(\mathbb Q,\chg{\Pnom})\le\rho$ holds vacuously for every $\mathbb Q\in\mathcal{P}(\Xi)$, and $f_\rho(\bm{x})=\sup_{\mathbb Q\in\mathcal{P}(\Xi)}\E_{\mathbb Q}[L(\chg{\bm{\tilde{\xi}}}^\top\bm{x})]=\sup_{\bm{\xi}\in\Xi}L(\bm{\xi}^\top\bm{x})$. The WDRO problem~\eqref{prob:main} thus reduces to the RO problem~\eqref{eq:extremes:ro}, whose optimal value is attained by $\bm{x}(\infty)$. Since $\bm{x}(\infty)$ lies on the shrinkage path, the shrinkage path heuristic is exact. Likewise, our suboptimality bound yields $\rho \left[1-\frac{\rho}{\diam(\Xi)}\right]_{+} \Lip(L) \| \bm{x}(0) \|_2=0$ when $\rho\ge\diam(\Xi)$, which therefore is both valid and tight.

We now prove the validity of our bound for $\rho\in[0,\diam(\Xi)]$. To this end, we derive an upper bound on the heuristic's value $\hat{f}_\rho$ and a matching lower bound on the optimal value $f_\rho^\star$.

As for the upper bound on the heuristic's value $\hat{f}_\rho$, we evaluate the objective $f_\rho(\bm{x})$ of~\eqref{prob:main} at the two extreme solutions $\bm{x}(0)$ and $\bm{x}(\infty)$ and then combine them to obtain a bound. While $\bm{x}(0)$ is optimal at $\rho=0$ (i.e., $f_0(\bm{x}(0))=f_0^\star$), it is generally suboptimal once $\rho>0$ (i.e., $f_\rho(\bm{x}(0))\neq f_\rho^\star$). To bound $f_\rho(\bm{x}(0))$, we control how fast the worst-case objective can grow as the radius increases to $\rho$. By Kantorovich--Rubinstein duality, $W(\mathbb Q,\chg{\Pnom})=\sup_{g:\,\Lip(g)\le1}\big\{\E_{\Qprob}[g(\chg{\bm{\tilde{\xi}}})]-\E_{\chg{\Pnom}}[g(\chg{\bm{\tilde{\xi}}})]\big\}$. Hence, every Lipschitz function $g$ satisfies $\E_{\mathbb Q}\big[g(\chg{\bm{\tilde{\xi}}})\big]-\E_{\chg{\Pnom}}\big[g(\chg{\bm{\tilde{\xi}}})\big]\le \Lip(g) W(\mathbb Q,\chg{\Pnom})$. By Assumption~\ref{ass:apriori}(ii), the loss function $L$ is Lipschitz continuous with constant $\Lip(L)$, so for any $\bm{x}$ the map $\bm{\xi} \mapsto L(\bm{\xi}^\top \bm{x})$ is Lipschitz in $\bm{\xi}$ with constant $\Lip(L) \| \bm{x} \|_2$. Applying the Kantorovich--Rubinstein inequality to $g(\bm{\xi})=L(\bm{\xi}^\top\bm{x})$ and invoking $f_0(\bm{x})=\E_{\chg{\Pnom}}[L(\chg{\bm{\tilde{\xi}}}^\top\bm{x})]$, we obtain $\E_{\mathbb Q}[L(\chg{\bm{\tilde{\xi}}}^\top\bm{x})]\le f_0(\bm{x})+\Lip(L) \|\bm{x}\|_2 W(\mathbb Q,\chg{\Pnom})$. Taking the supremum over all $\mathbb Q\in\Amb(\Pnom)$ with $W(\mathbb Q,\chg{\Pnom})\le\rho$ yields $f_\rho(\bm{x})\le f_0(\bm{x})+\rho \Lip(L) \|\bm{x}\|_2$. Applying the preceding inequality at $\bm{x}=\bm{x}(0)$ and observing that $f_0(\bm{x}(0))=f_0^\star$ by construction, we have
\begin{equation}
f_\rho(\bm{x}(0)) \;\; \le \;\; f_0(\bm{x}(0))+\rho \Lip(L) \|\bm{x}(0)\|_2 \;\; = \;\; f_0^\star+\rho \Lip(L) \|\bm{x}(0)\|_2.
\label{eq:KR-envelope}
\end{equation}
To bound $f_\rho(\bm{x}(\infty))$, observe that for each fixed $\bm{x}$ the map $\rho\mapsto f_\rho(\bm{x})$ is nondecreasing, because the Wasserstein ball expands as $\rho$ increases. Hence, for all $\rho\in[0,\diam(\Xi)]$, we have
\begin{equation}
f_\rho(\bm{x}(\infty)) \;\; \le \;\; f_{\diam(\Xi)}(\bm{x}(\infty)) \;\; = \;\; f_{\diam(\Xi)}^\star,
\label{eq:winfty-mono}
\end{equation}
where the final equality holds because $\bm{x}(\infty)$ optimally solves the WDRO problem~\eqref{prob:main} at $\rho=\diam(\Xi)$. It remains to combine the bounds~\eqref{eq:KR-envelope} and~\eqref{eq:winfty-mono}. Note that $f_\rho(\bm{x})$ is convex in $\bm{x}$ since it constitutes a pointwise supremum of weighted sums of the convex functions $\bm{x}\mapsto L(\bm{\xi}^\top\bm{x})$. We fix $\lambda=\frac{\rho}{\diam(\Xi)}\in[0,1]$ on the shrinkage path and apply Jensen's inequality to $f_\rho$ to obtain
\begin{align}\label{eq:heuristic-ub}
    \hat{f}_\rho
    &\;\; \le \;\; f_\rho\!\left(\Big(1-\tfrac{\rho}{\diam(\Xi)}\Big) \bm{x}(0)+\tfrac{\rho}{\diam(\Xi)} \bm{x}(\infty)\right) \nonumber\\
    &\;\; \le \;\; \Big(1-\tfrac{\rho}{\diam(\Xi)}\Big) f_\rho(\bm{x}(0))+\tfrac{\rho}{\diam(\Xi)} f_\rho(\bm{x}(\infty)) \nonumber\\
    &\;\; \le \;\; \Big(1-\tfrac{\rho}{\diam(\Xi)}\Big) \Big(f_0^\star+\rho \Lip(L) \| \bm{x}(0)\|_2\Big)+\tfrac{\rho}{\diam(\Xi)} f_{\diam(\Xi)}^\star,
\end{align}
where the last inequality invokes \eqref{eq:KR-envelope} and \eqref{eq:winfty-mono}.

As for the lower bound on the optimal value $f_\rho^\star$, we exploit the concavity of $f_\rho^\star$ in $\rho$: a concave function lies above the chord through any two of its points, so $f_\rho^\star$ is bounded below on $[0,\diam(\Xi)]$ by the chord joining its values at $\rho=0$ and $\rho=\diam(\Xi)$. To establish this concavity, we invoke the dual formulation of $f_\rho(\bm{x})$ \citep[Theorem~4.2]{mohajerin2018data}, namely $f_\rho(\bm{x})=    \inf_{\eta\ge0}\Big\{\eta\rho + \tfrac{1}{N}\sum_{i=1}^N\sup_{\bm{\xi}\in\Xi}\big[\Loss(\bm{x},\bm{\xi})-\eta\|\bm{\xi}-\chg{\bm{\hat{\xi}}}_i\|_2\big]\Big\}$. As each $f_\rho(\bm{x})$ is an infimum of functions affine in $\rho$, the optimal value $f_\rho^\star = \min_{\bm{x}\in\mathcal X(\Xi)} f_\rho(\bm{x})$ is itself a pointwise infimum of affine functions of $\rho$, and therefore concave in $\rho$. Writing any $\rho\in[0,\diam(\Xi)]$ as the convex combination $\rho=\big(1-\tfrac{\rho}{\diam(\Xi)}\big) 0+\tfrac{\rho}{\diam(\Xi)} \diam(\Xi)$, concavity then yields
\begin{equation}
f_\rho^\star \;\; \ge \;\; \left(1-\frac{\rho}{\diam(\Xi)}\right) f_0^\star+\frac{\rho}{\diam(\Xi)} f_{\diam(\Xi)}^\star.
\label{eq:concave-lb}
\end{equation}
Finally, subtracting the lower bound~\eqref{eq:concave-lb} on $f_\rho^\star$ from the upper bound~\eqref{eq:heuristic-ub} on $\hat{f}_\rho$, we obtain
\[
    \hat{f}_\rho-f_\rho^\star \;\; \le \;\; \left(1-\frac{\rho}{\diam(\Xi)}\right) \rho \Lip(L) \|\bm{x}(0)\|_2.
\]
The bound is parabolic in $\rho$ and is thus maximized at $\rho=\frac{\diam(\Xi)}{2}$.

We close by establishing the pointwise tightness of our bound for $\rho\in[0,\diam(\Xi)]$. As observed above, $f_0^\star=f_0(\bm{x}(0))$ and $f_{\diam(\Xi)}^\star=f_{\diam(\Xi)}(\bm{x}(\infty))$, with both $\bm{x}(0)$ and $\bm{x}(\infty)$ lying on the shrinkage path. Hence, the heuristic is exact---and our bound tight---at both endpoints. For each remaining radius $\bar\rho\in(0,\diam(\Xi))$, we construct an instance for which the bound is attained at $\rho=\bar\rho$. This instance has loss $L(z)=|z|$ (which satisfies Assumption~\ref{ass:apriori}(ii) with $\Lip(L)=1$), a single sample $\Pnom=\delta_{\bm\xi^0}$, and support $\Xi=\conv\{\bm\xi^0,\bm\xi^1\}$, where $\bm\xi^0=(0,\diam(\Xi))^\top$ and $\bm\xi^1=(\diam(\Xi),\diam(\Xi))^\top$. The feasible set is
$\mathcal{X}(\Xi)=\conv\Big\{(1,0)^\top,\;\big(0,\tfrac{\bar\rho}{\diam(\Xi)}\big)^\top,\;\big(\tfrac{\bar\rho}{\diam(\Xi)},\,0\big)^\top\Big\}$.
On this instance, since $\bm{x}\ge\bm 0$ and $\bm\xi\ge\bm 0$, the loss is linear: $L(\bm\xi^\top\bm{x})=\bm\xi^\top\bm{x}$. Because $\Pnom$ is supported on a single atom, the Wasserstein distance reduces to an expectation under $\Qprob\in\mathcal P(\Xi)$: $W(\Qprob,\Pnom)=\E_\Qprob\big[\|\chg{\bm{\tilde{\xi}}}-\bm\xi^0\|_2\big]=\E_\Qprob\big[\tilde\xi_1\big]$, where the second equality follows from the support structure of this instance. Hence, the worst-case objective $f_\rho(\bm{x})$ admits the closed form
\begin{equation*}
\begin{aligned}
f_\rho(\bm{x})&\;\; = \;\;\sup_{\Qprob\in\mathcal P(\Xi)}\big\{\E_{\Qprob}[\tilde\xi_1 x_1+\tilde\xi_2 x_2] \;:\; \E_\Qprob[\tilde\xi_1]\le \rho\big\}
\\&\;\; = \;\;x_2 \diam(\Xi)+x_1 \sup_{\Qprob_1\in\mathcal P([0,\diam(\Xi)])}\big\{\E_{\Qprob_1}[\tilde\xi_1] \;:\; \E_{\Qprob_1}[\tilde\xi_1]\le \rho\big\}
\\&\;\; = \;\;x_2 \diam(\Xi)+x_1 \min\{\rho,\diam(\Xi)\}.
\end{aligned}
\end{equation*}
The two path endpoints are $\bm{x}(0)=(1,0)^\top$, the minimizer of the SAA objective $f_0(\bm{x})=\diam(\Xi) x_2$ over $\mathcal X(\Xi)$, and $\bm{x}(\infty)=\big(0,\tfrac{\bar\rho}{\diam(\Xi)}\big)^\top$, the minimizer of the RO objective $f_{\diam(\Xi)}(\bm{x})=\diam(\Xi) (x_1+x_2)$ over $\mathcal X(\Xi)$. Both endpoints satisfy $f_{\bar\rho}(\bm{x}(0))=f_{\bar\rho}(\bm{x}(\infty))=\bar\rho$, so by linearity $f_{\bar\rho}(\bm{x})=\bar\rho$ for every $\bm{x}$ on the shrinkage path, giving $\hat f_{\bar\rho}=\bar\rho$. To evaluate the optimal value $f_{\bar\rho}^\star=\min_{\bm{x}\in\mathcal X(\Xi)}f_{\bar\rho}(\bm{x})$, note that $f_{\bar\rho}$ is affine in $\bm{x}$ and hence attains its minimum at a vertex of $\mathcal X(\Xi)$. Evaluating $f_{\bar\rho}$ at the three vertices of $\mathcal X(\Xi)$, we find that $(1,0)^\top$ and $\big(0,\tfrac{\bar\rho}{\diam(\Xi)}\big)^\top$ both give $\bar\rho$, whereas $\big(\tfrac{\bar\rho}{\diam(\Xi)},\,0\big)^\top$ gives $\frac{\bar\rho^2}{\diam(\Xi)}<\bar\rho$ since $\bar\rho<\diam(\Xi)$. Hence $f_{\bar\rho}^\star=\frac{\bar\rho^2}{\diam(\Xi)}$, and therefore $\hat f_{\bar\rho}-f_{\bar\rho}^\star
=\bar\rho-\frac{\bar\rho^2}{\diam(\Xi)}
=\bar\rho \Big(1-\frac{\bar\rho}{\diam(\Xi)}\Big) \Lip(L) \|\bm{x}(0)\|_2$. This is exactly the bound from the statement of the theorem, thereby establishing tightness. \hfill \Halmos

\proofgap
\noindent \textbf{Proof of Proposition \ref{thm:unbounded}.} $\;$ 
Throughout the proof, we use the same notation $f_\rho^\star$, $\hat{f}_\rho$, and $f_\rho(\bm{x})$ as in the proof of Theorem~\ref{thm:bounded}. If $\bar{r}=0$, then each sample is zero and the only radius in $[0, \bar{r}]$ at which the gap is zero is $\rho=0$. We therefore assume throughout the proof that $\bar{r}>0$. We first recall a closed-form identity for the WDRO objective value $f_\rho(\bm{x})$. Under Assumption~\ref{ass:apriori} with $\Xi = \R^k$, the WDRO problem reduces to a regularized SAA problem \citep[Theorem~4(ii)]{shafieezadehabadeh2019regularization}, yielding, for every $\bm{x} \in \mathcal{X}(\Xi)$,
\begin{equation}\label{eq:regSAA}
    f_\rho(\bm{x}) \;\; = \;\; \frac{1}{N}\sum_{i=1}^{N} L(\chg{\bm{\hat{\xi}}}_i^\top \bm{x})
    + \rho \Lip(L) \|\bm{x}\|_2.
\end{equation}
When $\rho\rightarrow\infty$, the penalty on $\|\bm{x}\|_2$ dominates $f_\rho(\bm{x})$. The WDRO endpoint solution is therefore the minimum-norm solution $\bm{x}(\infty) = \mathop{\arg\min}_{\bm{x} \in \mathcal{X}(\Xi)} \|\bm{x}\|_2$, which is exists and is unique since $\mathcal{X}(\Xi)$ is nonempty, closed and convex.

To prove the bound, we combine an upper bound on the heuristic's value $\hat f_\rho$, obtained from the identity~\eqref{eq:regSAA} at $\bm{x}(0)$ and $\bm{x}(\infty)$, with a lower bound on the optimal value $f_\rho^\star$, obtained from the concavity of $\rho\mapsto f_\rho^\star$ on $[0,\bar r]$. Subtracting the latter from the former bounds the suboptimality $\hat f_\rho - f_\rho^\star$ up to a residual term proportional to $f_0(\bm{x}(\infty)) - f_{\bar r}^\star$, which we show to be nonpositive. Together these steps give the claimed bound.

To bound the heuristic's value $\hat f_\rho$ from above, we evaluate~\eqref{eq:regSAA} at $\bm{x}(0)$ and $\bm{x}(\infty)$, which yields
\begin{equation}\label{eq:reg_ub_0}
    f_\rho(\bm{x}(0)) \;\; = \;\; f_0^\star+\rho \Lip(L) \|\bm{x}(0)\|_2,
\end{equation}
\begin{equation}\label{eq:reg_ub_inf}
    f_\rho(\bm{x}(\infty)) \;\; = \;\; f_0(\bm{x}(\infty))+\rho \Lip(L) \|\bm{x}(\infty)\|_2.
\end{equation}
where the first identity additionally invokes the optimality of $\bm{x}(0)$ at $\rho=0$.

Fixing $\lambda=\frac{\rho}{\bar r}\in[0,1]$ on the shrinkage path and using the convexity of $f_\rho(\bm{x})$ in $\bm{x}$, Jensen's inequality yields an upper bound on the heuristic's value $\hat f_\rho$,
\begin{equation}\label{eq:heuristic_ub}
    \begin{aligned}
        \hat f_\rho
        &\;\; \le \;\; f_\rho\!\left(\Big(1-\tfrac{\rho}{\bar r}\Big) \bm{x}(0)+\tfrac{\rho}{\bar r} \bm{x}(\infty)\right)
        \\&\;\; \le \;\; \Big(1-\tfrac{\rho}{\bar r}\Big) f_\rho(\bm{x}(0)) + \tfrac{\rho}{\bar r} f_\rho(\bm{x}(\infty))
        \\&\;\; = \;\; \Big(1-\tfrac{\rho}{\bar r}\Big) \big(f_0^\star+\rho \Lip(L) \|\bm{x}(0)\|_2\big) + \tfrac{\rho}{\bar r} \big(f_0(\bm{x}(\infty))+\rho \Lip(L) \|\bm{x}(\infty)\|_2\big),
    \end{aligned}
\end{equation}
where the equality invokes~\eqref{eq:reg_ub_0} and~\eqref{eq:reg_ub_inf}.

To bound the optimal value $f_\rho^\star$ from below, we note that by the identity~\eqref{eq:regSAA}, $f_\rho(\bm{x})$ is affine in $\rho$, and therefore $\rho \mapsto f_\rho^\star = \min_{\bm{x} \in \mathcal{X}(\Xi)} f_\rho(\bm{x})$ is concave as the pointwise infimum of affine functions of $\rho$. Writing any $\rho\in[0,\bar r]$ as the convex combination $\rho=\big(1-\tfrac{\rho}{\bar r}\big) 0+\tfrac{\rho}{\bar r} \bar r$, the concavity of $\rho\mapsto f_\rho^\star$ yields the following lower bound on $f_\rho^\star$:
\begin{equation}\label{eq:optimal_concave_lb}
    f_\rho^\star \;\; \ge \;\; \left(1-\frac{\rho}{\bar r}\right) f_0^\star + \frac{\rho}{\bar r} f_{\bar r}^\star.
\end{equation}
Subtracting the lower bound~\eqref{eq:optimal_concave_lb} from the upper bound~\eqref{eq:heuristic_ub} yields
\begin{equation*}
    \begin{aligned}
        \mspace{-30mu}\hat{f}_\rho-f_\rho^\star \;\; \le \;\; & \left(1-\frac{\rho}{\bar r}\right) \big( f_0^\star+\rho \Lip(L) \|\bm{x}(0)\|_2\big) + \frac{\rho}{\bar r} \big(f_0(\bm{x}(\infty))+\rho \Lip(L) \|\bm{x}(\infty)\|_2\big)-\left(1-\frac{\rho}{\bar r}\right) f_0^\star -\frac{\rho}{\bar r} f_{\bar r}^\star
        \\ \;\; = \;\; & \left(1-\frac{\rho}{\bar r}\right) \rho \Lip(L) \|\bm{x}(0)\|_2+ \frac{\rho^2}{\bar r} \Lip(L) \|\bm{x}(\infty)\|_2 + \frac{\rho}{\bar r}\big(f_0(\bm{x}(\infty)) - f_{\bar r}^\star\big).
    \end{aligned}
\end{equation*}
It remains to show that $f_0(\bm{x}(\infty))\le f_{\bar r}^\star$ to yield the claimed bound in Proposition~\ref{thm:unbounded}. Observe that $f_0(\bm{x})=\frac{1}{N}\sum_{i=1}^{N} L(\chg{\bm{\hat{\xi}}}_i^\top \bm{x})$. By the Lipschitz continuity of $L$ and the Cauchy--Schwarz inequality, for all $\bm{x}\in\mathcal X(\Xi)$, we have
\begin{equation}\label{eq:lipschitz_proof}
    \begin{aligned}
        f_0(\bm{x}(\infty))
        &\;\; \le \;\; f_0(\bm{x})+\frac{1}{N}\sum_{i=1}^N\Lip(L) \big|\chg{\bm{\hat{\xi}}}_i^\top(\bm{x}-\bm{x}(\infty))\big|
        \\&\;\; \le \;\; f_0(\bm{x})+\left(\frac{1}{N}\sum_{i=1}^N\|\chg{\bm{\hat{\xi}}}_i\|_2\right) \Lip(L) \|\bm{x}-\bm{x}(\infty)\|_2
        \\&\;\; = \;\; f_0(\bm{x})+\bar r \Lip(L) \|\bm{x}-\bm{x}(\infty)\|_2.
    \end{aligned}
\end{equation}
Since $\bm{x}(\infty)$ minimizes $\|\cdot\|_2$---equivalently $\tfrac12\|\cdot\|_2^2$---over the convex set $\mathcal{X}(\Xi)$, the first-order optimality condition $\nabla\big(\tfrac12\|\bm{x}\|_2^2\big)\big|_{\bm{x}(\infty)}^\top(\bm{x}-\bm{x}(\infty))\ge 0$ gives $\bm{x}(\infty)^\top(\bm{x}-\bm{x}(\infty)) \ge 0$ for all $\bm{x}\in\mathcal{X}(\Xi)$. Therefore, we have $    \|\bm{x}\|_2^2=\|\bm{x}-\bm{x}(\infty)\|_2^2 + 2 \bm{x}(\infty)^\top(\bm{x}-\bm{x}(\infty)) + \|\bm{x}(\infty)\|_2^2\ge \|\bm{x}-\bm{x}(\infty)\|_2^2$. Substituting $\|\bm{x}-\bm{x}(\infty)\|_2 \le \|\bm{x}\|_2$ into~\eqref{eq:lipschitz_proof} and invoking the identity~\eqref{eq:regSAA} at $\rho=\bar r$, we have
\begin{equation*}
    f_0(\bm{x}(\infty)) \;\; \le \;\; f_0(\bm{x})+\bar r \Lip(L) \|\bm{x}\|_2 \;\; = \;\; f_{\bar r}(\bm{x}).
\end{equation*}
The inequality $f_0(\bm{x}(\infty))\le f_{\bar r}(\bm{x})$ holds for all $\bm{x}\in\mathcal X(\Xi)$. Minimizing the right-hand side over $\bm{x}\in\mathcal X(\Xi)$ yields the target inequality $f_0(\bm{x}(\infty))\le f_{\bar r}^\star$, and therefore
\begin{equation*}
\begin{aligned}
    \hat{f}_\rho-f_\rho^\star \;\; &\le  \left(1-\frac{\rho}{\bar r}\right) \rho \Lip(L) \|\bm{x}(0)\|_2+ \frac{\rho^2}{\bar r} \Lip(L) \|\bm{x}(\infty)\|_2 + \frac{\rho}{\bar r}\big(f_0(\bm{x}(\infty)) - f_{\bar r}^\star\big)
    \\\;\; &\le \Lip(L) \left( \rho \left(1 - \frac{\rho}{\bar r}\right) \|\bm{x}(0)\|_2
    + \frac{\rho^2}{\bar r} \|\bm{x}(\infty)\|_2 \right),
\end{aligned}
\end{equation*}
which completes the proof of the claimed bound. For $\bm{x}(\infty) = \bm{0}$, this bound is maximized at $\rho = \bar r/2$.
\hfill \Halmos

\proofgap
The proof of Theorem~\ref{thm:quadratic} relies on the following auxiliary results, which we state and prove first. Throughout, we use the notation $f_\rho^\star$, $\hat f_\rho$, and $f_\rho(\bm{x})$ from the proof of Theorem~\ref{thm:bounded}, we write $\Delta(\rho):=\hat f_\rho-f_\rho^\star$ for the suboptimality gap of the shrinkage path heuristic at the radius $\rho$, and we denote by $\bm{x}^\star(\rho)$ and $\hat{\bm{x}}(\rho)$ the corresponding optimizers of the WDRO problem~\eqref{prob:main} and of the shrinkage path heuristic, respectively. The first lemma expresses the optimal values $f_\rho^\star$ and $\hat f_\rho$ through explicit quadratics; the second lemma locates a gap-maximizing radius $\rho_{\mathrm{gap}}$ and expresses $\Delta(\rho_{\mathrm{gap}})$ as the covariance between two functions of a random eigenvalue of $\bm H$; the third lemma bounds this covariance by an \emph{endpoint envelope}, that is, an upper bound that depends on the spectrum of $\bm H$ only through its extreme eigenvalues $\lambda_{\min}$ and $\lambda_{\max}$; and the fourth lemma maximizes the resulting univariate envelope in closed form.

\begin{lemma}\label{lem:quadratic-reduction}
Let Assumption~\ref{ass:quadratic} hold and suppose $\bm{x}(0)\neq\bm0$. After absorbing $\Lip(L)$ into $\rho$, so that $f_\rho(\bm{x})=f_0(\bm{x})+\rho \|\bm{x}\|_2$, the WDRO optimal value satisfies, for every $\rho\ge0$,
\begin{equation}\label{eq:dro-quadratic}
        f_\rho^\star = f_0^\star + \min_{\bm{x} \in \R^n} \big\{ (\bm{x} - \bm{x}(0))^\top \bm{H} (\bm{x} - \bm{x}(0)) + \rho \|\bm{x}\|_2 \big\}.
\end{equation}
Moreover, $\bm{x}(\infty)=\bm0$, and the shrinkage path solution satisfies $\hat{\bm{x}}(\rho)=\lambda^\star(\rho) \bm{x}(0)$, where
\begin{equation}
    \label{eq:shrink-factor-closed-form}
     \lambda^\star(\rho)= \left[1-\frac{\rho\,\|\bm{x}(0)\|_2}{2\,\bm{x}(0)^\top\bm{H}\,\bm{x}(0)}\right]_+,
\end{equation}
and the corresponding shrinkage path value is
\begin{equation}\label{eq:interp-closed-form}
    \hat f_\rho  =    f_0^\star+    \bm{x}(0)^\top\bm{H}\,\bm{x}(0)    -    \bm{x}(0)^\top\bm{H}\,\bm{x}(0)\left[1-\frac{\rho\,\|\bm{x}(0)\|_2}{2\,\bm{x}(0)^\top\bm{H}\,\bm{x}(0)}\right]_+^2.
\end{equation}
\end{lemma}

\noindent \textbf{Proof of Lemma~\ref{lem:quadratic-reduction}.} $\;$
The normalization of the radius is harmless: under the present assumptions, $\cZ$ contains a nondegenerate interval and $L$ is a nonconstant quadratic on it, so $\Lip(L)>0$ and the substitution $\rho\leftarrow\Lip(L) \rho$ simply relabels the radii $\rho\ge0$. The remainder of the proof proceeds in three steps: the first step shows that on the ball $B:=\{\bm{x}\in\R^n:\|\bm{x}\|_2\le\|\bm{x}(0)\|_2\}$, the empirical loss $f_0$ is a quadratic centered at $\bm{x}(0)$; the second step shows that the WDRO problem and the quadratic problem in~\eqref{eq:dro-quadratic} both attain their minima on $B$, where their objectives coincide up to the constant $f_0^\star$, which establishes~\eqref{eq:dro-quadratic}; and the third step determines $\bm{x}(\infty)$ and optimizes over the shrinkage path, which yields~\eqref{eq:shrink-factor-closed-form} and~\eqref{eq:interp-closed-form}.

For the first step, the definition of $\cZ$ ensures that $\chg{\bm{\hat{\xi}}}_i^\top \bm{x}\in\cZ$ for all $\bm{x}\in B$ and $i\in[N]$, so Assumption~\ref{ass:quadratic}\emph{(ii)} implies that
\[
    f_0(\bm{x})
    =\bm{x}^\top\bm H\bm{x}
    + b\,\bm{x}^\top\left(\frac1N\sum_{i=1}^N\chg{\bm{\hat{\xi}}}_i\right)+c
    \qquad\text{for all }\bm{x}\in B.
\]
Because $L$ is differentiable on $\R$ and agrees with the quadratic $z\mapsto z^2+bz+c$ on $\cZ$, we have $L'(z)=2z+b$ for all $z\in\cZ$: at interior points this is immediate, and at the endpoints $\pm R$ of $\mathcal{Z}:=[-R, R]$, where $R:=\Vert \bm{x}(0)\Vert_2\max_{i \in [N]}\Vert \hat{\bm{\xi}}_i\Vert_2$, differentiability forces $L'$ to coincide with the one-sided derivative of the quadratic taken from inside $\mathcal{Z}$. Since $\chg{\bm{\hat{\xi}}}_i^\top\bm{x}(0)\in\cZ$ for all $i$ and $\mathcal X(\Xi)=\R^n$, the SAA first-order condition $\nabla f_0(\bm{x}(0))=\bm0$ is
\[
    2\bm H\bm{x}(0)+b \left(\frac1N\sum_{i=1}^N\chg{\bm{\hat{\xi}}}_i\right)=\bm0.
\]
Substituting $b \frac1N\sum_{i=1}^N\chg{\bm{\hat{\xi}}}_i=-2\bm H\bm{x}(0)$ into the expression for $f_0$ above and completing the square, we obtain
\[
    f_0(\bm{x})
    =(\bm{x}-\bm{x}(0))^\top\bm H(\bm{x}-\bm{x}(0))+c-\bm{x}(0)^\top\bm H\,\bm{x}(0)
    =f_0^\star+(\bm{x}-\bm{x}(0))^\top\bm H(\bm{x}-\bm{x}(0))
    \qquad\text{for all }\bm{x}\in B,
\]
where setting $\bm{x}=\bm{x}(0)$ in the first equality of the preceding display identifies the constant $c-\bm{x}(0)^\top\bm H\,\bm{x}(0)$ as $f_0(\bm{x}(0))=f_0^\star$.

For the second step, denote by $q(\bm{x}):=(\bm{x}-\bm{x}(0))^\top\bm H(\bm{x}-\bm{x}(0))+\rho \|\bm{x}\|_2$ the function minimized in~\eqref{eq:dro-quadratic}. In view of the first step and the identity~\eqref{eq:regSAA}, we have $f_\rho(\bm{x})=f_0(\bm{x})+\rho \|\bm{x}\|_2=f_0^\star+q(\bm{x})$ for all $\bm{x}\in B$. Both $f_\rho$ and $q$ attain their minima over $\R^n$ on the compact set $B$: for $\rho=0$, both functions are minimized by $\bm{x}(0)\in B$, while for $\rho>0$, every $\bm{x}\notin B$ satisfies $\|\bm{x}\|_2>\|\bm{x}(0)\|_2$ and is therefore dominated by $\bm{x}(0)$, since the optimality of $\bm{x}(0)$ in the SAA problem~\eqref{eq:extremes:saa} implies
\[
    \mspace{-12mu}
    f_\rho(\bm{x})=f_0(\bm{x})+\rho \|\bm{x}\|_2>f_0(\bm{x}(0))+\rho \|\bm{x}(0)\|_2=f_\rho(\bm{x}(0))
    \quad\text{and}\quad
    q(\bm{x})\ge\rho \|\bm{x}\|_2>\rho \|\bm{x}(0)\|_2=q(\bm{x}(0)).
\]
We thus conclude that
\[
    f_\rho^\star=\min_{\bm{x}\in B}f_\rho(\bm{x})=f_0^\star+\min_{\bm{x}\in B}q(\bm{x})=f_0^\star+\min_{\bm{x}\in\R^n}q(\bm{x}),
\]
which is~\eqref{eq:dro-quadratic}; in particular, the minimizers of the WDRO problem coincide with those of $q$.

For the third step, $\mathcal X(\Xi)=\R^n$ implies that the minimum-norm endpoint is $\bm{x}(\infty)=\bm0$, so the shrinkage path is $\{\lambda \bm{x}(0):\lambda\in[0,1]\}\subseteq B$, where $\lambda$ denotes the coefficient of $\bm{x}(0)$ and thus corresponds to $1-\lambda$ in the parameterization~\eqref{eq:path}. The first step and the identity~\eqref{eq:regSAA} then give
\[
    f_\rho(\lambda \bm{x}(0))
    =f_0^\star+(1-\lambda)^2 \bm{x}(0)^\top\bm H\,\bm{x}(0)
      +\rho \lambda \|\bm{x}(0)\|_2 .
\]
Minimizing this strongly convex univariate quadratic over $\lambda\in[0,1]$ yields the minimizer $\lambda^\star(\rho)$ in~\eqref{eq:shrink-factor-closed-form}, so that $\hat{\bm{x}}(\rho)=\lambda^\star(\rho) \bm{x}(0)$; substituting this minimizer into the path objective gives~\eqref{eq:interp-closed-form}.
\hfill\Halmos

\proofgap
\begin{lemma}\label{lem:quadratic-spectral}
Let the conditions of Lemma~\ref{lem:quadratic-reduction} hold, and let $\rho_{\mathrm{gap}}$ be a maximizer of $\Delta(\rho)$ over $\rho\ge0$ with $\Delta(\rho_{\mathrm{gap}})>0$. Then $\rho_{\mathrm{gap}}\in(0,\rho^{\mathrm h}_0)$, where
\[
    \rho^{\mathrm h}_0:=\frac{2\,\bm{x}(0)^\top\bm H\,\bm{x}(0)}{\|\bm{x}(0)\|_2}.
\]
Moreover,
\begin{equation}\label{eq:stationarity}
\|\hat{\bm{x}}(\rho_{\mathrm{gap}})\|_2=\|\bm{x}^{\star}(\rho_{\mathrm{gap}})\|_2.
\end{equation}
If $\bm H=\bm U\operatorname{diag}(\lambda_1,\ldots,\lambda_n)\bm U^\top$ with $\bm U$ orthogonal, $\tilde{\bm{x}}(0)=\bm U^\top\bm{x}(0)$, $p_i=\tilde x(0)_i^2/\|\bm{x}(0)\|_2^2$, and $\phi_\gamma(\lambda)=(\lambda/(\lambda+\gamma))^2$, then, with $\mathbb P(\Lambda=\lambda_i)=p_i$ and $\gamma_{\mathrm{gap}}=\rho_{\mathrm{gap}}/(2\|\bm{x}^\star(\rho_{\mathrm{gap}})\|_2)$,
\begin{equation}
  \frac{\Delta(\rho_{\mathrm{gap}})}{\|\bm{x}(0)\|_2^{2}}=\mathbb{E}_p[\Lambda\phi_{\gamma_{\mathrm{gap}}}(\Lambda)]-\mathbb{E}_p[\Lambda]\,\mathbb{E}_p[\phi_{\gamma_{\mathrm{gap}}}(\Lambda)],\label{eq:moment-form-quadratic}
\end{equation}
and
\begin{equation}\label{eq:stationarity-spectral}
    \mathbb{E}_p[\phi_{\gamma_{\mathrm{gap}}}(\Lambda)]=\left(\frac{\mathbb E_p[\Lambda]}{\mathbb E_p[\Lambda]+\gamma_{\mathrm{gap}}}\right)^2=\phi_{\gamma_{\mathrm{gap}}}\left(\mathbb{E}_p[\Lambda]\right).
\end{equation}
\end{lemma}

\noindent \textbf{Proof of Lemma~\ref{lem:quadratic-spectral}.} $\;$
We proceed in three steps: first, we establish that $\rho_{\mathrm{gap}}\in(0,\rho^{\mathrm h}_0)$ and the norm identity~\eqref{eq:stationarity}, second, we derive closed forms for $f_{\rho_{\mathrm{gap}}}^\star$ and $\hat f_{\rho_{\mathrm{gap}}}$ from the first-order condition of the quadratic problem in~\eqref{eq:dro-quadratic}, which yield the covariance identity~\eqref{eq:moment-form-quadratic} in the eigenbasis of $\bm H$, and third, we rewrite~\eqref{eq:stationarity} in this eigenbasis to obtain~\eqref{eq:stationarity-spectral}.

For the first step, since $\Delta(0)=0<\Delta(\rho_{\mathrm{gap}})$, the maximizer $\rho_{\mathrm{gap}}$ satisfies $\rho_{\mathrm{gap}}>0$. The objective in~\eqref{eq:dro-quadratic} is strongly convex, and the path objective is strongly convex in $\lambda$, so both $\bm{x}^\star(\rho)$ and $\hat{\bm{x}}(\rho)$ are unique. Danskin's theorem states that a minimum-value function $\rho\mapsto\min_{u}g(u,\rho)$, whose objective is affine in $\rho$ and whose minimum is attained at a unique $u(\rho)$, is differentiable with derivative $(\partial g/\partial\rho)(u(\rho),\rho)$. Since $\rho$ enters both minimization problems only through the term $\rho\|\bm{x}\|_2$, which is affine in $\rho$, we obtain, for every $\rho>0$,
\[
    \Delta'(\rho)=\|\hat{\bm{x}}(\rho)\|_2-\|\bm{x}^\star(\rho)\|_2.
\]
Fermat's rule, which states that the derivative of a differentiable function vanishes at every interior extremum, applied at the maximizer $\rho_{\mathrm{gap}}>0$, gives $\Delta'(\rho_{\mathrm{gap}})=0$ and hence we obtain \eqref{eq:stationarity}.
This rules out every $\rho_{\mathrm{gap}}\ge\rho^{\mathrm h}_0$: at any such radius~\eqref{eq:shrink-factor-closed-form} gives $\hat{\bm{x}}(\rho_{\mathrm{gap}})=\bm0$, so~\eqref{eq:stationarity} would force $\bm{x}^\star(\rho_{\mathrm{gap}})=\bm0$, and then~\eqref{eq:dro-quadratic} and~\eqref{eq:interp-closed-form} would both evaluate to $f_0^\star+\bm{x}(0)^\top\bm H\bm{x}(0)$, making the heuristic exact and contradicting $\Delta(\rho_{\mathrm{gap}})>0$. Thus $\rho_{\mathrm{gap}}\in(0,\rho^{\mathrm h}_0)$; moreover, $\hat{\bm{x}}(\rho_{\mathrm{gap}})=\lambda^\star(\rho_{\mathrm{gap}})\,\bm{x}(0)\neq\bm0$, and so $\bm{x}^\star(\rho_{\mathrm{gap}})\neq\bm0$ by~\eqref{eq:stationarity}.

For the second step, recall from Lemma~\ref{lem:quadratic-reduction} that $\bm{x}^\star(\rho_{\mathrm{gap}})$ minimizes the quadratic problem on the right-hand side of~\eqref{eq:dro-quadratic}; the associated first-order condition at $\bm{x}^\star(\rho_{\mathrm{gap}})\neq\bm0$ is
\begin{equation}
\label{eq:first-order-condition-quadratic}
    2\bm H(\bm{x}^{\star}(\rho_{\mathrm{gap}})-\bm{x}(0))+\rho_{\mathrm{gap}}\frac{\bm{x}^{\star}(\rho_{\mathrm{gap}})}{\|\bm{x}^{\star}(\rho_{\mathrm{gap}})\|_2}=\bm 0.
\end{equation}
With $\gamma_{\mathrm{gap}}:=\rho_{\mathrm{gap}}/(2\|\bm{x}^\star(\rho_{\mathrm{gap}})\|_2)$, this is equivalently
\[
    (\bm H+\gamma_{\mathrm{gap}}\bm I)\bm{x}^\star(\rho_{\mathrm{gap}})=\bm H\bm{x}(0).
\]
Multiplying this equation from the left by $\bm{x}^\star(\rho_{\mathrm{gap}})^\top$ yields $\bm{x}(0)^\top\bm H\,\bm{x}^\star(\rho_{\mathrm{gap}})=\bm{x}^\star(\rho_{\mathrm{gap}})^\top\bm H\,\bm{x}^\star(\rho_{\mathrm{gap}})+\gamma_{\mathrm{gap}} \|\bm{x}^\star(\rho_{\mathrm{gap}})\|_2^2$; substituting this identity and $\rho_{\mathrm{gap}} \|\bm{x}^\star(\rho_{\mathrm{gap}})\|_2=2\gamma_{\mathrm{gap}} \|\bm{x}^\star(\rho_{\mathrm{gap}})\|_2^2$ into the objective of~\eqref{eq:dro-quadratic}, evaluated at $\bm{x}^\star(\rho_{\mathrm{gap}})$, yields
\[
    f_{\rho_{\mathrm{gap}}}^{\star}
    =f_0^{\star}+\bm{x}(0)^\top\bm H\bm{x}(0)
    -\bm{x}^{\star}(\rho_{\mathrm{gap}})^\top\bm H\bm{x}^{\star}(\rho_{\mathrm{gap}}).
\]
Similarly, since $\hat{\bm{x}}(\rho_{\mathrm{gap}})=\lambda^\star(\rho_{\mathrm{gap}}) \bm{x}(0)$ implies $\hat{\bm{x}}(\rho_{\mathrm{gap}})^\top\bm H\,\hat{\bm{x}}(\rho_{\mathrm{gap}})=\lambda^\star(\rho_{\mathrm{gap}})^2 \bm{x}(0)^\top\bm H\,\bm{x}(0)$, the path value~\eqref{eq:interp-closed-form} can be written as
\[
    \hat f_{\rho_{\mathrm{gap}}}
    =f_0^\star+\bm{x}(0)^\top\bm H\bm{x}(0)
    -\hat{\bm{x}}(\rho_{\mathrm{gap}})^\top\bm H\hat{\bm{x}}(\rho_{\mathrm{gap}}).
\]
Subtracting the first identity from the second and using $\hat{\bm{x}}(\rho_{\mathrm{gap}})^\top\bm H\,\hat{\bm{x}}(\rho_{\mathrm{gap}})=\frac{\|\hat{\bm{x}}(\rho_{\mathrm{gap}})\|_2^2}{\|\bm{x}(0)\|_2^2}\,\bm{x}(0)^\top\bm H\,\bm{x}(0)$ together with the norm identity~\eqref{eq:stationarity} gives
\begin{equation}
    \label{eq:spectral-form}
     \Delta(\rho_{\mathrm{gap}})=\bm{x}^{\star}(\rho_{\mathrm{gap}})^\top \bm{H}\,\bm{x}^{\star}(\rho_{\mathrm{gap}})-\frac{\|\bm{x}^{\star}(\rho_{\mathrm{gap}})\|^2_2}{\|\bm{x}(0)\|^2_2}\bm{x}(0)^\top \bm{H}\,\bm{x}(0).
\end{equation}
Diagonalize $\bm H=\bm U\operatorname{diag}(\lambda_1,\ldots,\lambda_n)\bm U^\top$ with $\bm U$ orthogonal, set $\tilde{\bm{x}}(0)=\bm U^\top\bm{x}(0)$, and define $p_i:=\tilde x(0)_i^2/\|\bm{x}(0)\|_2^2$. Then $p_i\ge0$ and $\sum_i p_i=1$. Writing the resolvent equation $(\bm H+\gamma_{\mathrm{gap}}\bm I)\,\bm{x}^\star(\rho_{\mathrm{gap}})=\bm H\bm{x}(0)$ in this eigenbasis gives
\[
    \tilde x_i^\star(\rho_{\mathrm{gap}})=\frac{\lambda_i}{\lambda_i+\gamma_{\mathrm{gap}}}\,\tilde x_i(0).
\]
Substituting this expression into~\eqref{eq:spectral-form} yields
\begin{equation}\label{eq:Delta-spectral}
\Delta(\rho_{\mathrm{gap}})=\|\bm{x}(0)\|_2^{2} \left[\sum_{i=1}^{n}\lambda_i \left(\frac{\lambda_i}{\lambda_i+\gamma_{\mathrm{gap}}}\right)^{2}p_i-\sum_{i=1}^{n} \left(\frac{\lambda_i}{\lambda_i+\gamma_{\mathrm{gap}}}\right)^{2}p_i\sum_{j=1}^{n}\lambda_j p_j\right].
\end{equation}

Define the random variable $\Lambda$ by $\mathbb P(\Lambda=\lambda_i)=p_i$, and set $\phi_\gamma(\lambda):=(\lambda/(\lambda+\gamma))^2$. Then~\eqref{eq:Delta-spectral} becomes the covariance identity
\eqref{eq:moment-form-quadratic}.
Because $\phi_{\gamma_{\mathrm{gap}}}$ is increasing in $\lambda$, this covariance is nonnegative; the gap is precisely the cost of applying one common shrinkage factor instead of the coordinatewise factors $\lambda_i/(\lambda_i+\gamma_{\mathrm{gap}})$.

For the third step, it remains to rewrite the norm identity~\eqref{eq:stationarity} in spectral form. Let $m:=\mathbb E_p[\Lambda]$ and $\lambda_F:=\lambda^\star(\rho_{\mathrm{gap}})$. From the eigenbasis representation,
\[
    \|\bm{x}^{\star}(\rho_{\mathrm{gap}})\|_2^2
    =\|\bm{x}(0)\|_2^2 \mathbb E_p[\phi_{\gamma_{\mathrm{gap}}}(\Lambda)],
\]
and~\eqref{eq:stationarity} gives $\mathbb E_p[\phi_{\gamma_{\mathrm{gap}}}(\Lambda)]=\lambda_F^2$. On the other hand, since $\rho_{\mathrm{gap}}\in(0,\rho^{\mathrm h}_0)$,
\[
    \lambda_F=1-\frac{\rho_{\mathrm{gap}}\,\|\bm{x}(0)\|_2}{2\,\bm{x}(0)^\top\bm H\,\bm{x}(0)}.
\]
Using $\rho_{\mathrm{gap}}=2\gamma_{\mathrm{gap}} \lambda_F \|\bm{x}(0)\|_2$---which follows from the definition of $\gamma_{\mathrm{gap}}$ together with $\|\bm{x}^\star(\rho_{\mathrm{gap}})\|_2=\lambda_F \|\bm{x}(0)\|_2$, that is,~\eqref{eq:stationarity}---and $m=\bm{x}(0)^\top\bm H\,\bm{x}(0)/\|\bm{x}(0)\|_2^2$, this becomes $\lambda_F=1-\gamma_{\mathrm{gap}}\lambda_F/m$, or
\[
    \lambda_F=\frac{m}{m+\gamma_{\mathrm{gap}}}.
\]
Consequently,~\eqref{eq:stationarity-spectral} holds.
\hfill\Halmos

\proofgap
\begin{lemma}\label{lem:quadratic-endpoint-envelope}
Under the conditions of Lemma~\ref{lem:quadratic-spectral}, the gap at the maximizing radius $\rho_{\mathrm{gap}}$ admits a spectral upper bound that depends on $\bm H$ only through its condition number $\kappa$. Specifically,
\begin{equation}\label{eq:normalized-scalar-upper-profile}
   \frac{\Delta(\rho_{\mathrm{gap}})}{\lambda_{\min} \|\bm{x}(0)\|_2^{2}} \leq    \sup_{t\geq0}t^2\left(\frac{1}{\sqrt{1+t}}-\frac{1}{\sqrt{\kappa+t}}\right)^2.
\end{equation}
\end{lemma}

\noindent \textbf{Proof of Lemma~\ref{lem:quadratic-endpoint-envelope}.} $\;$
The proof proceeds in three steps: the first step rewrites the covariance in~\eqref{eq:moment-form-quadratic} as the difference between $\mathbb E_p[r_\gamma(\Lambda)]$ and $r_\gamma(\mathbb E_p[\Lambda])$, that is, the Jensen gap of the resolvent $r_\gamma(\lambda):=1/(\lambda+\gamma)$; the second step bounds this gap by relaxing the law of $\Lambda$ to a mean-constrained distribution on the spectral interval $[\lambda_{\min},\lambda_{\max}]$; and the third step maximizes the resulting endpoint expression over the mean and rescales it to obtain~\eqref{eq:normalized-scalar-upper-profile}.

For the first step, write $\gamma$ for $\gamma_{\mathrm{gap}}$, set $m:=\mathbb E_p[\Lambda]$, and define $r_\gamma(\lambda):=1/(\lambda+\gamma)$. Using the covariance identity~\eqref{eq:moment-form-quadratic}, the spectral stationarity condition~\eqref{eq:stationarity-spectral} in the form $\mathbb E_p[\phi_\gamma(\Lambda)]=\phi_\gamma(m)$, and the identity
\[
    (\lambda+\gamma)\phi_\gamma(\lambda)=\lambda-\gamma+\gamma^2r_\gamma(\lambda),
\]
we obtain the resolvent representation
\begin{equation}
\label{eq:proof-spectral-identity1}
   \frac{\Delta(\rho_{\mathrm{gap}})}{\|\bm{x}(0)\|_2^{2}}=   \mathbb E_p[\Lambda\phi_\gamma(\Lambda)]-m\phi_\gamma(m)=\gamma^2\bigl(\mathbb E_p[r_\gamma(\Lambda)]-r_{\gamma}(m)\bigr).
\end{equation}
Indeed, the middle expression of~\eqref{eq:proof-spectral-identity1} equals
\[
    m-\gamma+\gamma^2\mathbb E_p[r_\gamma(\Lambda)]-(m+\gamma)\phi_\gamma(m),
\]
and $(m+\gamma)\phi_\gamma(m)=m-\gamma+\gamma^2r_\gamma(m)$.
Thus the stationarity condition~\eqref{eq:stationarity-spectral} has reduced the covariance to a resolvent Jensen gap. From this point on, the argument uses only the support interval $[\lambda_{\min},\lambda_{\max}]$ and the mean $m$ of $\Lambda$.

For the second step, we relax the requirements on the law of $\Lambda$ and merely require it to be supported on $[\lambda_{\min},\lambda_{\max}]$ with mean $\mathbb E_p[\Lambda]=m$. Since $r_\gamma$ is convex on $[\lambda_{\min},\lambda_{\max}]$, every such law satisfies the secant bound
\begin{equation}
\label{eq:resolvent-secant-sum}
    \mathbb E_p[r_{\gamma}(\Lambda)]\leq\frac{\lambda_{\max}-m}{\lambda_{\max}-\lambda_{\min}}r_{\gamma}(\lambda_{\min})+\frac{m-\lambda_{\min}}{\lambda_{\max}-\lambda_{\min}}r_{\gamma}(\lambda_{\max}).
\end{equation}
The right-hand side of~\eqref{eq:resolvent-secant-sum} equals $\mathbb E_p[r_\gamma(\Lambda)]$ when $\Lambda$ follows the two-point law on $\{\lambda_{\min},\lambda_{\max}\}$ with mean $m$. The secant bound therefore holds with equality for this law.
Substituting~\eqref{eq:resolvent-secant-sum} into~\eqref{eq:proof-spectral-identity1} gives the endpoint envelope
\begin{equation*}
\begin{aligned}
   \frac{\Delta(\rho_{\mathrm{gap}})}{\|\bm{x}(0)\|_2^{2}}  \leq\frac{\gamma^2(m-\lambda_{\min})(\lambda_{\max}-m)}{(\lambda_{\min}+\gamma)(\lambda_{\max}+\gamma)(m+\gamma)}.
\end{aligned}
\end{equation*}
The resulting expression is nevertheless an upper bound, not an exact reduction of the stationarity-constrained problem, because stationarity is no longer imposed after the resolvent rewrite. Taking the supremum over $m\in[\lambda_{\min},\lambda_{\max}]$ and $\gamma\ge0$ yields
\begin{equation}\label{eq:endpoint-envelope-problem}
   \frac{\Delta(\rho_{\mathrm{gap}})}{\|\bm{x}(0)\|_2^{2}} \leq \sup_{m\in[\lambda_{\min},\lambda_{\max}],\;\gamma\geq0}\frac{\gamma^2(m-\lambda_{\min})(\lambda_{\max}-m)}{(\lambda_{\min}+\gamma)(\lambda_{\max}+\gamma)(m+\gamma)}.
\end{equation}

For the third step, fix $\gamma>0$. The only $m$-dependent factor in~\eqref{eq:endpoint-envelope-problem} is $(m-\lambda_{\min})(\lambda_{\max}-m)/(m+\gamma)$. Its derivative is
\[
    \frac{(\lambda_{\min}+\gamma)(\lambda_{\max}+\gamma)-(m+\gamma)^2}{(m+\gamma)^2},
\]
so the maximizing mean is
\[
    m^\star(\gamma)=\sqrt{(\lambda_{\min}+\gamma)(\lambda_{\max}+\gamma)}-\gamma,
\]
which lies in $[\lambda_{\min},\lambda_{\max}]$. Evaluating~\eqref{eq:endpoint-envelope-problem} at $m^\star(\gamma)$ reduces the bound to
\[
    \frac{\Delta(\rho_{\mathrm{gap}})}{\|\bm{x}(0)\|_2^2}
    \leq
    \sup_{\gamma\ge0}\gamma^2
    \left(\frac1{\sqrt{\lambda_{\min}+\gamma}}
    -\frac1{\sqrt{\lambda_{\max}+\gamma}}\right)^2.
\]
With $\gamma=\lambda_{\min}t$ and $\kappa=\lambda_{\max}/\lambda_{\min}$, this is
\eqref{eq:normalized-scalar-upper-profile}.
\hfill\Halmos

\proofgap
\begin{lemma}\label{lem:quadratic-scalar-envelope}
The scalar envelope in Lemma~\ref{lem:quadratic-endpoint-envelope} admits the following closed-form bound: for every $\kappa\ge1$ and $t\ge0$,
\begin{equation}\label{eq:pointwise-normalized-envelope}
    t^2\left(\frac{1}{\sqrt{1+t}}-\frac{1}{\sqrt{\kappa+t}}\right)^2\leq\frac{(\kappa-1)^2}{\varphi^5(\kappa-1)+27}.
\end{equation}
\end{lemma}

\noindent \textbf{Proof of Lemma~\ref{lem:quadratic-scalar-envelope}.} $\;$
The claim is immediate when $\kappa=1$, so take $\kappa>1$. The proof proceeds in two steps: the first step rationalizes the difference of square roots, then rescales the desired inequality through variables $u$ and $z$; the second step maximizes the resulting left-hand side over $u$ and compares the maximum with the right-hand side.

For the first step, we invoke the identity $\bigl(\sqrt{\kappa+t}-\sqrt{1+t}\bigr)\bigl(\sqrt{\kappa+t}+\sqrt{1+t}\bigr)=\kappa-1$ to rewrite the left-hand side of~\eqref{eq:pointwise-normalized-envelope} as $t^2(\kappa-1)^2/\bigl[(1+t)(\kappa+t)\bigl(\sqrt{1+t}+\sqrt{\kappa+t}\bigr)^2\bigr]$. Clearing the positive denominators and dividing by $(\kappa-1)^2>0$, we obtain the equivalent inequality
\[
    \left(\varphi^5(\kappa-1)+27\right)t^2
    \leq
    (1+t)(\kappa+t)\left(\sqrt{1+t}+\sqrt{\kappa+t}\right)^2.
\]
Now set $u:=t/(1+t)\in[0,1)$ and $z:=(\kappa-1)/(1+t)\ge0$, and define $R(z):=(1+z)\left(1+\sqrt{1+z}\right)^2$. Dividing the preceding inequality by $(1+t)^3$ shows that it suffices to prove, for all $u\in[0,1]$ and $z\ge0$,
\begin{equation}\label{eq:transformed-normalized-envelope}
    u^2\left[27(1-u)+\varphi^5 z\right]\leq R(z).
\end{equation}

For the second step, fix $z$. The derivative of the left-hand side with respect to $u$ is $u(54+2\varphi^5 z-81u)$. Hence, with $z_0:=27/(2\varphi^5)$,
\begin{equation*}
    \max_{u\in[0,1]}u^2\left[27(1-u)+\varphi^5 z\right]=
    \begin{cases}
        4\left(1+\varphi^5 z/27\right)^3, & 0\leq z\leq z_0,\\
        \varphi^5 z, & z\geq z_0,
    \end{cases}
\end{equation*}
where the first branch is attained at the interior critical point and the second at $u=1$. We now compare both branches with $R(z)$. For the second branch, writing $s=\sqrt{1+z}\ge1$ gives
\[
    R(z)-\varphi^5z=(s+1)(s-\varphi)^2(s+2+\sqrt5)\ge0.
\]
For the first branch, set $g(z):=(R(z)/4)^{1/3}$. We claim that $g$ is concave. Indeed, writing $s=\sqrt{1+z}$ gives $g(z)=2^{-2/3}(s(s+1))^{2/3}$ and
\[
    g''(z)=-\frac{2^{-2/3}(4s^2+7s+4)}{18s^{10/3}(s+1)^{4/3}}<0.
\]
Moreover $g(0)=1$, and the preceding branch bound at $z_0$ gives $g(z_0)\ge(\varphi^5z_0/4)^{1/3}=3/2$. Therefore, for $z\in[0,z_0]$,
\[
    g(z)\geq\left(1-\frac{z}{z_0}\right)g(0)+\frac{z}{z_0}g(z_0)
    \geq 1+\frac{\varphi^5 z}{27}.
\]
Cubing both sides and invoking $g(z)=(R(z)/4)^{1/3}$ yields $R(z)/4\ge(1+\varphi^5z/27)^3$, hence $R(z)\ge4(1+\varphi^5z/27)^3$ on $[0,z_0]$. The two branch bounds $R(z)\ge4(1+\varphi^5z/27)^3$ on $[0,z_0]$ and $R(z)\ge\varphi^5z$ on $[z_0,\infty)$ together give $\max_{u\in[0,1]}u^2\left[27(1-u)+\varphi^5 z\right]\le R(z)$ for every $z\ge0$, which proves~\eqref{eq:transformed-normalized-envelope} and hence~\eqref{eq:pointwise-normalized-envelope}.
\hfill\Halmos

\proofgap
\noindent \textbf{Proof of Theorem~\ref{thm:quadratic}.} $\;$
We use the notation introduced before Lemma~\ref{lem:quadratic-reduction}. The proof proceeds in three steps: the first step handles the trivial case $\bm{x}(0)=\bm0$ and otherwise combines the auxiliary lemmas to prove the additive bound; the second step builds two-dimensional instances and stationary radii at which the gap can be evaluated explicitly; and the third step evaluates these stationary radii in the regimes $\kappa\to\infty$ and $\kappa\to1^+$ to prove asymptotic tightness.

For the first step, if $\bm{x}(0)=\bm0$, then $\bm0$ minimizes both the empirical loss and the norm regularizer in~\eqref{eq:regSAA}, so it is a WDRO optimizer for every $\rho\ge0$; the shrinkage path is also identically zero, and hence $\Delta(\rho)=0$ for every $\rho$. We therefore assume $\bm{x}(0)\neq\bm0$ in the remainder of the proof and, as in Lemma~\ref{lem:quadratic-reduction}, absorb the constant $\Lip(L)$ in~\eqref{eq:regSAA} into the radius $\rho$, which leaves the worst-case suboptimality $\sup_{\rho\ge0}\Delta(\rho)$ unchanged.

Define $\rho^{\mathrm h}_0:=2\,\bm{x}(0)^\top\bm H\,\bm{x}(0)/\|\bm{x}(0)\|_2$, and let $\rho_{\mathrm{gap}}$ maximize $\Delta(\rho)$ over $\rho\ge0$. Such a maximizer exists because Lemma~\ref{lem:quadratic-reduction} gives continuity of both $\hat f_\rho$ and $f_\rho^\star$ in $\rho$, while $\hat f_\rho$ is constant for $\rho\ge\rho^{\mathrm h}_0$ by~\eqref{eq:interp-closed-form} and $f_\rho^\star$ is nondecreasing in $\rho$, so that $\Delta$ is nonincreasing on $[\rho^{\mathrm h}_0,\infty)$ and $\sup_{\rho\ge0}\Delta(\rho)$ is attained on the compact interval $[0,\rho^{\mathrm h}_0]$. If $\Delta(\rho_{\mathrm{gap}})=0$, the theorem is immediate. Otherwise, Lemmas~\ref{lem:quadratic-spectral}, \ref{lem:quadratic-endpoint-envelope}, and~\ref{lem:quadratic-scalar-envelope} give
\[
    \Delta(\rho_{\mathrm{gap}})
    \leq
    \lambda_{\min} \|\bm{x}(0)\|_2^2 \frac{(\kappa-1)^2}{\varphi^5(\kappa-1)+27}
    =
    \lambda_{\max} \|\bm{x}(0)\|_2^2 \frac{(\kappa-1)^2}{\kappa[\varphi^5(\kappa-1)+27]},
\]
which is the bound~\eqref{eq:quadratic-additive-bound}.

For the second step, we set up the tightness construction. Consider the matrix $\bm H_\kappa$ and the SAA solution $\bm{x}_\kappa(0)$ given by
\begin{equation}\label{eq:two-point-instance}
    \bm H_\kappa=\operatorname{diag}(1,\kappa),\qquad \bm{x}_\kappa(0)=\|\bm{x}_\kappa(0)\|_2\,(\sqrt{1-\theta},\sqrt{\theta})^\top,\qquad\theta\in(0,1),
\end{equation}
with spectral law $p_\theta=(1-\theta)\delta_1+\theta\delta_\kappa$ and mean $m=1+\theta(\kappa-1)$, where $\delta_a$ denotes the unit point mass at $a$. These reduced data---in fact, any $\bm H\succ\bm0$ and any target $\bm{x}(0)\neq\bm0$---are realized by empirical samples satisfying Assumption~\ref{ass:quadratic}. Fix $\beta>2\sqrt{\bm{x}(0)^\top\bm H\bm{x}(0)}$ and set $\bar{\bm\xi}:=-2\bm H\bm{x}(0)/\beta$. Then $\bar{\bm\xi}^\top\bm H^{-1}\bar{\bm\xi}=4\,\bm{x}(0)^\top\bm H\bm{x}(0)/\beta^2<1$, so $\bm H-\bar{\bm\xi}\bar{\bm\xi}^\top\succ\bm0$ admits a factorization $\bm H-\bar{\bm\xi}\bar{\bm\xi}^\top=\bm L\bm L^\top$. The $N=2n$ equally weighted samples $\bar{\bm\xi}\pm\sqrt n\,\bm L\bm e_j$, $j\in[n]$, have empirical mean $\bar{\bm\xi}$ and empirical second-moment matrix $\bar{\bm\xi}\bar{\bm\xi}^\top+\bm L\bm L^\top=\bm H$. Choose for $L$ a Huber-type loss that agrees with $z^2+\beta z+c$ on an open interval containing $\cZ$ and continues linearly, with matching value and slope, outside that interval. Here $c$ is large enough to ensure $L\ge0$. This loss is convex, globally Lipschitz, differentiable, and quadratic on $\cZ$, so Assumptions~\ref{ass:apriori} and~\ref{ass:quadratic}\emph{(ii)} hold with $b=\beta$. Finally, the choice of $\bar{\bm\xi}$ yields the first-order condition $\nabla f_0(\bm{x}(0))=2\bm H\bm{x}(0)+\beta\bar{\bm\xi}=\bm0$, and $f_0$ coincides with the strongly convex quadratic $\bm{x}\mapsto\bm{x}^\top\bm H\bm{x}+\beta\bar{\bm\xi}^\top\bm{x}+c$ on a neighborhood of $\bm{x}(0)$. Since a convex function with a strong local minimizer admits no other minimizers, $\bm{x}(0)$ is the unique SAA solution.

To evaluate the gaps produced by this construction, we use the same covariance derivation as in Lemma~\ref{lem:quadratic-spectral}: the derivation only relies on the first-order condition~\eqref{eq:first-order-condition-quadratic} and the stationarity property~\eqref{eq:stationarity}, and is therefore not specific to the maximizing radius $\rho_{\mathrm{gap}}$. We now construct radii with these two properties. For any $\gamma>0$, the point $\bm{x}_\gamma:=(\bm H_\kappa+\gamma\bm I)^{-1}\bm H_\kappa\bm{x}_\kappa(0)$ is nonzero because $\bm H_\kappa\succ\bm0$ and $\bm{x}_\kappa(0)\neq\bm0$. It satisfies the first-order condition~\eqref{eq:first-order-condition-quadratic} with $\rho(\gamma):=2\gamma\|\bm{x}_\gamma\|_2$ in place of $\rho_{\mathrm{gap}}$, hence $\bm{x}^\star(\rho(\gamma))=\bm{x}_\gamma$ by strong convexity of~\eqref{eq:dro-quadratic}. The dual multiplier $\rho(\gamma)/(2\|\bm{x}^\star(\rho(\gamma))\|_2)$ equals $\gamma$ itself. If, in addition, $\gamma$ satisfies the spectral stationarity condition~\eqref{eq:stationarity-spectral} for the law $p_\theta$ with mean $m=\mathbb E_{p_\theta}[\Lambda]$---in which case we call $(\theta,\gamma)$ a \emph{stationary pair} and $\rho(\gamma)$ its \emph{stationary radius}---then $\|\bm{x}_\gamma\|_2=\|\bm{x}_\kappa(0)\|_2\sqrt{\mathbb E_{p_\theta}[\phi_\gamma(\Lambda)]}=m(m+\gamma)^{-1}\|\bm{x}_\kappa(0)\|_2$, so that $\rho(\gamma)=2\gamma m(m+\gamma)^{-1}\|\bm{x}_\kappa(0)\|_2$. Moreover, since $\rho(\gamma)\|\bm{x}_\kappa(0)\|_2/(2\bm{x}_\kappa(0)^\top\bm H_\kappa\bm{x}_\kappa(0))=\gamma/(m+\gamma)<1$, the positive part in~\eqref{eq:shrink-factor-closed-form} is inactive and $\lambda^\star(\rho(\gamma))=m/(m+\gamma)$. Thus $\|\hat{\bm{x}}(\rho(\gamma))\|_2=\|\bm{x}_\gamma\|_2$, and the stationarity property~\eqref{eq:stationarity} indeed holds at $\rho(\gamma)$. For the two-point law $p_\theta$, the covariance formula at a stationary radius therefore reduces to
\begin{equation}\label{eq:two-point-gap}
    \frac{\Delta(\rho(\gamma))}{\|\bm{x}_\kappa(0)\|_2^2}=\mathbb E_{p_\theta}[\Lambda\phi_\gamma(\Lambda)]-\mathbb E_{p_\theta}[\Lambda]\,\mathbb E_{p_\theta}[\phi_\gamma(\Lambda)]=\theta(1-\theta)(\kappa-1)\bigl[\phi_\gamma(\kappa)-\phi_\gamma(1)\bigr].
\end{equation}
Since $\lambda_{\max}=\kappa$, the upper bound~\eqref{eq:quadratic-additive-bound} becomes $\|\bm{x}_\kappa(0)\|_2^2(\kappa-1)^2/[\varphi^5(\kappa-1)+27]$ on these instances. Thus it suffices to exhibit stationary radii whose gaps match this expression asymptotically.

For the third step, first let $\kappa\to\infty$ and set $\alpha:=1/\varphi$, so $\alpha^2+\alpha=1$ and $1+\alpha=1/\alpha=\varphi$. Take $\gamma_\kappa=\alpha\kappa$ and choose $\theta_\kappa$ to satisfy~\eqref{eq:stationarity-spectral}, which for the law $p_\theta$ reads $(1-\theta)\,\phi_{\alpha\kappa}(1)+\theta\,\phi_{\alpha\kappa}(\kappa)=\phi_{\alpha\kappa}(m)$ with $m=1+\theta(\kappa-1)$. Since $\phi_{\alpha\kappa}(\kappa)=(1+\alpha)^{-2}=\alpha^2$ for every $\kappa$, while $\phi_{\alpha\kappa}(1)\to0$ and $\phi_{\alpha\kappa}(m)\to\theta^2/(\theta+\alpha)^2$ as $\kappa\to\infty$, the stationarity equation converges to
\[
    \alpha^2\theta=\frac{\theta^2}{(\theta+\alpha)^2},
\]
whose roots in $(0,1]$ are the interior point $\theta=\alpha^2$---by the defining relation $\alpha^2+\alpha=1$---and the boundary point $\theta=1$. Define the limiting stationarity defect $D(\theta):=\alpha^2\theta-\theta^2/(\theta+\alpha)^2$. Its derivative at the interior root is $D'(\alpha^2)=\alpha^2(1-2\alpha)\neq0$, so $D$ changes sign at $\alpha^2$. The intermediate value theorem furnishes, for all sufficiently large $\kappa$, a root $\theta_\kappa$ of the exact stationarity equation with $\theta_\kappa\to\alpha^2$ and $1-\theta_\kappa\to\alpha$. Let $\bar\rho_\kappa:=\rho(\gamma_\kappa)=2\gamma_\kappa m_\kappa(m_\kappa+\gamma_\kappa)^{-1}\|\bm{x}_\kappa(0)\|_2$ be the stationary radius of the pair $(\theta_\kappa,\gamma_\kappa)$, where $m_\kappa=1+\theta_\kappa(\kappa-1)$. Since $\phi_{\alpha\kappa}(\kappa)=\alpha^2$ and $\phi_{\alpha\kappa}(1)\to0$,~\eqref{eq:two-point-gap} gives
\[
    \frac{\Delta(\bar\rho_\kappa)}{\kappa\|\bm{x}_\kappa(0)\|_2^2}
    \longrightarrow
    \alpha^2 \alpha \alpha^2
    =\frac1{\varphi^5}.
\]
The upper bound divided by $\kappa\|\bm{x}_\kappa(0)\|_2^2$ has the same limit, proving tightness as $\kappa\to\infty$.

Finally let $\kappa=1+\epsilon$ with $\epsilon\downarrow0$ and choose $\theta=1/2$. The stationarity defect
\[
    D_\epsilon(\gamma)
    :=\mathbb E_{p_{1/2}}[\phi_\gamma(\Lambda)]
      -\phi_\gamma\left(1+\frac{\epsilon}{2}\right)
    =
    \frac{\epsilon^2}{8}\frac{2\gamma(\gamma-2)}{(1+\gamma)^4}
    +O(\epsilon^3)
\]
vanishes at $\gamma=2$ to leading order in $\epsilon$, and the factor $2\gamma(\gamma-2)/(1+\gamma)^4$ has derivative $4/81\neq0$ at $\gamma=2$. For all sufficiently small $\epsilon$, the defect therefore changes sign near $\gamma=2$, and the intermediate value theorem furnishes a stationary dual $\bar\gamma_\epsilon\to2$. Let $\bar\rho_\epsilon:=\rho(\bar\gamma_\epsilon)$ be the corresponding stationary radius. Since $\phi_2'(1)=4/27$,~\eqref{eq:two-point-gap} gives
\[
    \frac{\Delta(\bar\rho_\epsilon)}{\|\bm{x}_\kappa(0)\|_2^2}
    =
    \frac14\,\epsilon\left[\phi_{\bar\gamma_\epsilon}(1+\epsilon)-\phi_{\bar\gamma_\epsilon}(1)\right]
    =
    \frac{\epsilon^2}{27}+O(\epsilon^3),
\]
which matches $\epsilon^2/[\varphi^5\epsilon+27]=\epsilon^2/27+O(\epsilon^3)$. This proves tightness as $\kappa\to1^+$ and completes the proof.
\hfill\Halmos

\proofgap
\noindent \textbf{Proof of Corollary~\ref{cor:golden-ratio-geometry}.} $\;$
For each $\kappa$, we use the two-point instance~\eqref{eq:two-point-instance} with $\theta=\theta_\kappa$, where $\theta_\kappa$ is the parameter constructed in the $\kappa\to\infty$ tightness step of the proof of Theorem~\ref{thm:quadratic} and satisfies $\theta_\kappa\to\alpha^2$ for $\alpha:=1/\varphi$. The proof proceeds in three steps: the first step shows that the maximal gap of the $\kappa$th instance asymptotically attains the worst-case bound~\eqref{eq:quadratic-additive-bound}; the second step identifies the asymptotic scale of the dual multiplier $\gamma_{\mathrm{gap},\kappa}$ at a gap-maximizing radius; and the third step converts that scale into the three radii in the statement.

For the first step, let $\rho_{\mathrm{gap},\kappa}$ be a radius that maximizes the gap $\Delta(\rho)$ for the $\kappa$th instance. The stationary radius $\bar\rho_\kappa$ constructed in the proof of Theorem~\ref{thm:quadratic} is an admissible candidate radius for the same instance, so $\Delta(\rho_{\mathrm{gap},\kappa})\ge\Delta(\bar\rho_\kappa)$. The lower bound from that construction and the upper bound~\eqref{eq:quadratic-additive-bound} therefore squeeze the maximizing gap:
\begin{equation}\label{eq:cor-geometry-gap-limit}
    \frac{\Delta(\rho_{\mathrm{gap},\kappa})}{\kappa\|\bm{x}_\kappa(0)\|_2^2}\longrightarrow \alpha^5=\frac1{\varphi^5}.
\end{equation}
For all sufficiently large $\kappa$, this gap is positive. Lemma~\ref{lem:quadratic-spectral} then applies at $\rho_{\mathrm{gap},\kappa}$. We write $\gamma_{\mathrm{gap},\kappa}:=\rho_{\mathrm{gap},\kappa}/(2\|\bm{x}_\kappa^\star(\rho_{\mathrm{gap},\kappa})\|_2)$ for the dual multiplier $\gamma_{\mathrm{gap}}$ from that lemma, specialized to the $\kappa$th~instance.

For the second step, we claim that $\gamma_{\mathrm{gap},\kappa}/\kappa\to\alpha$. For the two-point instances, $\lambda_{\min}=1$ and $\lambda_{\max}=\kappa$. Write the scalar endpoint envelope in Lemma~\ref{lem:quadratic-endpoint-envelope} as $G_\kappa(t):=t^2\big((1+t)^{-1/2}-(\kappa+t)^{-1/2}\big)^2$. The proof of that lemma gives $\Delta(\rho_{\mathrm{gap},\kappa})/\|\bm{x}_\kappa(0)\|_2^2\le G_\kappa(\gamma_{\mathrm{gap},\kappa})$.
Together with Lemma~\ref{lem:quadratic-scalar-envelope} and~\eqref{eq:cor-geometry-gap-limit}, this implies the squeeze
\begin{equation}\label{eq:cor-geometry-envelope-limit}
    \alpha^5
    \;\longleftarrow\;
    \frac{\Delta(\rho_{\mathrm{gap},\kappa})}{\kappa\|\bm{x}_\kappa(0)\|_2^2}
    \;\le\; \frac{G_\kappa(\gamma_{\mathrm{gap},\kappa})}{\kappa}
    \;\le\; \frac{(\kappa-1)^2}{\kappa[\varphi^5(\kappa-1)+27]}
    \;\longrightarrow\;\alpha^5.
\end{equation}

Now set $\beta_\kappa:=\gamma_{\mathrm{gap},\kappa}/\kappa$ and define
\[
    h_\kappa(\beta):=\frac{G_\kappa(\beta\kappa)}{\kappa}
    =\beta^2\left(\frac{1}{\sqrt{\beta+1/\kappa}}-\frac{1}{\sqrt{1+\beta}}\right)^2,
    \qquad
    w(\beta):=\beta^2\left(\frac{1}{\sqrt{\beta}}-\frac{1}{\sqrt{1+\beta}}\right)^2.
\]
Since $1/\sqrt{\beta+1/\kappa}\to1/\sqrt{\beta}$ uniformly on every compact subset of $(0,\infty)$, we have that $h_\kappa\to w$ uniformly on every such subset. Also, the conjugate form of $G_\kappa$ and the inequalities $(\sqrt{1+t}+\sqrt{\kappa+t})^2\ge\kappa+t$ and $(\kappa-1)^2\le\kappa^2$ yield
\begin{equation*}
    h_\kappa(\beta)\le
    \frac{\beta^2\kappa}{(1+\beta\kappa)(1+\beta)^2}
    \le \frac{\beta}{(1+\beta)^2}\qquad\forall \beta>0.
\end{equation*}
Since the right-hand side tends to zero as $\beta\downarrow0$ and as $\beta\uparrow\infty$,~\eqref{eq:cor-geometry-envelope-limit} rules out $\beta_\kappa\to0$ and $\beta_\kappa\to\infty$ along any subsequence. Along any subsequence with $\beta_\kappa\to\bar\beta\in(0,\infty)$, the local uniform convergence gives $w(\bar\beta)=\lim_{\kappa\to\infty}h_\kappa(\beta_\kappa)=\alpha^5$.

It remains only to identify the maximizer of $w$. Since $\sqrt{w(\beta)}=\sqrt{\beta}-\beta/\sqrt{1+\beta}$, the first-order condition for an interior maximizer of $w$ is $(1+\beta)^3=(2+\beta)^2\beta$, equivalently $\beta^2+\beta-1=0$. Its unique positive solution is $\beta=\alpha$. Also $w(\beta)\to0$ as $\beta\downarrow0$ and as $\beta\uparrow\infty$, so $\alpha$ is the unique maximizer of $w$, with $w(\alpha)=\alpha^5$. Hence every convergent subsequence of $\{\beta_\kappa\}$ has limit $\alpha$, and the claim~follows.

For the third step, let $m_\kappa:=\bm{x}_\kappa(0)^\top\bm H_\kappa\bm{x}_\kappa(0)/\|\bm{x}_\kappa(0)\|_2^2=1+\theta_\kappa(\kappa-1)$. Since $\theta_\kappa\to\alpha^2$, we have $m_\kappa/\kappa\to\alpha^2$. The stationarity relation in Lemma~\ref{lem:quadratic-spectral} and the limit $\gamma_{\mathrm{gap},\kappa}/\kappa\to\alpha$ give
\[
    \lambda^\star(\rho_{\mathrm{gap},\kappa})
    =\frac{m_\kappa}{m_\kappa+\gamma_{\mathrm{gap},\kappa}}
    \longrightarrow\frac{\alpha^2}{\alpha^2+\alpha}=\alpha^2.
\]
The three characteristic radii are
\[
    \begin{aligned}
    \rho_{\mathrm{gap},\kappa}&=2\gamma_{\mathrm{gap},\kappa}\lambda^\star(\rho_{\mathrm{gap},\kappa})\|\bm{x}_\kappa(0)\|_2,\qquad
    \rho^{\mathrm h}_{0,\kappa}=2m_\kappa\|\bm{x}_\kappa(0)\|_2,\\
    \rho^{\star}_{0,\kappa}&=2\|\bm H_\kappa\bm{x}_\kappa(0)\|_2
    =2\|\bm{x}_\kappa(0)\|_2\sqrt{(1-\theta_\kappa)+\theta_\kappa\kappa^2}.
    \end{aligned}
\]
The first identity combines the definition of $\gamma_{\mathrm{gap},\kappa}$ with the norm identity~\eqref{eq:stationarity} and $\hat{\bm{x}}_\kappa(\rho)=\lambda^\star(\rho) \bm{x}_\kappa(0)$; the second is the radius at which the shrinkage factor~\eqref{eq:shrink-factor-closed-form} reaches zero; and the third holds because $\bm{x}_\kappa^\star(\rho)=\bm0$ if and only if $\bm0$ satisfies the subgradient optimality condition $\rho\,\partial\|\cdot\|_2(\bm0)\ni2\bm H_\kappa\bm{x}_\kappa(0)$ in the quadratic problem of~\eqref{eq:dro-quadratic}, that is, if and only if $\rho\ge2\|\bm H_\kappa\bm{x}_\kappa(0)\|_2$. After dividing by $2\kappa\|\bm{x}_\kappa(0)\|_2$, these identities yield
\[
    \frac{\rho_{\mathrm{gap},\kappa}}{2\kappa\|\bm{x}_\kappa(0)\|_2}\to\alpha^3,\qquad
    \frac{\rho^{\mathrm h}_{0,\kappa}}{2\kappa\|\bm{x}_\kappa(0)\|_2}\to\alpha^2,\qquad
    \frac{\rho^{\star}_{0,\kappa}}{2\kappa\|\bm{x}_\kappa(0)\|_2}\to\alpha.
\]
Since $1/\alpha=\varphi$, we conclude that
\[
    \rho_{\mathrm{gap},\kappa}:\rho^{\mathrm h}_{0,\kappa}:\rho^{\star}_{0,\kappa}
    \longrightarrow \alpha^3:\alpha^2:\alpha=1:\varphi:\varphi^2.
\]
\hfill \Halmos

\proofgap
The proof of Theorem~\ref{thm:endpoint} relies on the following auxiliary results, which we state and prove first. Throughout, we use the notation introduced before Lemma~\ref{lem:quadratic-reduction} and the spectral notation of Lemma~\ref{lem:quadratic-spectral}, and we write $\Delta^{0,\infty}(\rho):=\min\{f_\rho(\bm{x}(0)),f_\rho(\bm{x}(\infty))\}-f_\rho^\star$ for the additive suboptimality of the SAA--RO heuristic at the radius $\rho$. The first lemma locates $\rho_{\mathrm{gap}}$, a maximizer of the suboptimality $\Delta(\rho)$ of the shrinkage path heuristic. The second lemma provides a correlation inequality in the spirit of Cassels' classical reverse Cauchy--Schwarz inequality \citep[Appendix~1]{watson1955serial}.

\begin{lemma}\label{lem:endpoint-localization}
Let the conditions of Lemma~\ref{lem:quadratic-spectral} hold. Then the gap-maximizing radius satisfies $\rho_{\mathrm{gap}}>\rho_{\mathrm{end}}$, where $\rho_{\mathrm{end}}:=\bm{x}(0)^\top \bm{H}\,\bm{x}(0)/\|\bm{x}(0)\|_2$.
\end{lemma}
\noindent \textbf{Proof of Lemma~\ref{lem:endpoint-localization}.} $\;$
Set $m:=\mathbb E_p[\Lambda]$. Then $m=\bm{x}(0)^\top\bm H\bm{x}(0)/\|\bm{x}(0)\|_2^2=\rho_{\mathrm{end}}/\|\bm{x}(0)\|_2$. By Lemmas~\ref{lem:quadratic-reduction} and~\ref{lem:quadratic-spectral}, $\|\bm{x}^\star(\rho_{\mathrm{gap}})\|_2=\|\hat{\bm{x}}(\rho_{\mathrm{gap}})\|_2=\lambda^\star(\rho_{\mathrm{gap}})\|\bm{x}(0)\|_2$ with $\lambda^\star(\rho_{\mathrm{gap}})=m/(m+\gamma_{\mathrm{gap}})$. Hence
\[
    \rho_{\mathrm{gap}}=2\gamma_{\mathrm{gap}}\lambda^\star(\rho_{\mathrm{gap}})\|\bm{x}(0)\|_2,\qquad
    \rho_{\mathrm{end}}=m\|\bm{x}(0)\|_2,\qquad
    \frac{\rho_{\mathrm{gap}}}{\rho_{\mathrm{end}}}=\frac{2\gamma_{\mathrm{gap}}}{m+\gamma_{\mathrm{gap}}}.
\]
It is therefore enough to show that $\gamma_{\mathrm{gap}}>m$.

Suppose, toward a contradiction, that $\gamma_{\mathrm{gap}}\le m$. Let $\chg{a}(\lambda):=\phi_{\gamma_{\mathrm{gap}}}(m)+\phi_{\gamma_{\mathrm{gap}}}'(m)(\lambda-m)$ be the tangent line to $\phi_{\gamma_{\mathrm{gap}}}$ at $m$, and set $h:=\chg{a}-\phi_{\gamma_{\mathrm{gap}}}$. Then $h(m)=h'(m)=0$ and
\[
    h''(\lambda)=-\phi_{\gamma_{\mathrm{gap}}}''(\lambda)=\frac{2\gamma_{\mathrm{gap}}(2\lambda-\gamma_{\mathrm{gap}})}{(\lambda+\gamma_{\mathrm{gap}})^4}.
\]
Thus $h$ is concave on $(0,\gamma_{\mathrm{gap}}/2)$ and strictly convex on $(\gamma_{\mathrm{gap}}/2,\infty)$. Since $m\ge\gamma_{\mathrm{gap}}$, convexity and $h(m)=h'(m)=0$ imply $h\ge0$ on $[\gamma_{\mathrm{gap}}/2,\infty)$, with equality only at $m$. Moreover,
\begin{equation*}
    h(0)
    = \phi_{\gamma_{\mathrm{gap}}}(m)-\phi_{\gamma_{\mathrm{gap}}}'(m)m-\phi_{\gamma_{\mathrm{gap}}}(0)
    = \frac{m^2}{(m+\gamma_{\mathrm{gap}})^2}-\frac{2\gamma_{\mathrm{gap}} m^2}{(m+\gamma_{\mathrm{gap}})^3}
    =\frac{m^2(m-\gamma_{\mathrm{gap}})}{(m+\gamma_{\mathrm{gap}})^3}\ge0.
\end{equation*}
Since also $h(\gamma_{\mathrm{gap}}/2)>0$ (were it zero, convexity and $h \geq 0$ on $[\gamma_{\rm gap}/2, m]$ together with $h(m)=0$ would force $h=0$ on this interval, contradicting strict convexity), concavity on $[0,\gamma_{\mathrm{gap}}/2]$ gives
\begin{equation*}
    h(\lambda) \ge \frac{\gamma_{\mathrm{gap}}/2 - \lambda}{\gamma_{\mathrm{gap}}/2}\,h(0)
    +\frac{\lambda}{\gamma_{\mathrm{gap}}/2}\,h(\gamma_{\mathrm{gap}}/2)>0,
    \qquad \lambda\in(0,\gamma_{\mathrm{gap}}/2].
\end{equation*}
Thus $h(\lambda)\ge0$ for all $\lambda>0$, with equality only at $\lambda=m$. Taking expectations and using the stationarity condition~\eqref{eq:stationarity-spectral} gives
\[
    \mathbb E_p[h(\Lambda)]
    =\mathbb E_p[\chg{a}(\Lambda)]-\mathbb E_p[\phi_{\gamma_{\mathrm{gap}}}(\Lambda)]
    =\chg{a}(m)-\phi_{\gamma_{\mathrm{gap}}}(m)=0.
\]
Hence $\Lambda=m$ $p$-almost surely. The covariance formula~\eqref{eq:moment-form-quadratic} would then give $\Delta(\rho_{\mathrm{gap}})=0$, contradicting $\Delta(\rho_{\mathrm{gap}})>0$. Therefore $\gamma_{\mathrm{gap}}>m$, and so $\rho_{\mathrm{gap}}>\rho_{\mathrm{end}}$.
\hfill\Halmos

\proofgap
\begin{lemma}\label{lem:cassels-ratio}
Let $X$ and $Y$ be random variables satisfying $X\in[a,b]$ and $Y\in[c,d]$ almost surely, where $0<a\leq b$ and $0<c\leq d$. Then
\begin{equation}\label{eq:endpoint-cassels-bound}
    1-\frac{\mathbb E[X]\,\mathbb E[Y]}{\mathbb E[XY]}\leq\frac{(b-a)(d-c)}{(\sqrt{ac}+\sqrt{bd})^2}.
\end{equation}
\end{lemma}

\noindent \textbf{Proof of Lemma~\ref{lem:cassels-ratio}.} $\;$
If $a=b$ or $c=d$, then $X$ or $Y$ is almost surely constant and both sides of~\eqref{eq:endpoint-cassels-bound} vanish; we therefore assume $a<b$ and $c<d$. Let $x:=\mathbb E[X]$ and $y:=\mathbb E[Y]$. Since $(b-X)(Y-c)\ge0$ and $(X-a)(d-Y)\ge0$, taking expectations and dividing by $xy$ gives
\[
    \frac{\mathbb E[XY]}{xy}
    \le B_1(x,y):=\frac{by+cx-bc}{xy},
    \qquad
    \frac{\mathbb E[XY]}{xy}
    \le B_2(x,y):=\frac{dx+ay-ad}{xy}.
\]
Thus $\mathbb E[XY]/xy\le\min\{B_1(x,y),B_2(x,y)\}$. Write $u:=(x-a)/(b-a)$ and $v:=(y-c)/(d-c)$. A direct comparison shows that $B_1(x,y)\le B_2(x,y)$ if $v\le u$, and $B_2(x,y)\le B_1(x,y)$ if $u\le v$. Moreover, $B_1$ is nondecreasing in $y$ for fixed $x$, since $\partial_y B_1=c(b-x)/(xy^2)\ge0$, and $B_2$ is nondecreasing in $x$ for fixed $y$, since $\partial_x B_2=a(d-y)/(x^2y)\ge0$. Hence the largest possible value of $\min\{B_1,B_2\}$ occurs on the diagonal $u=v=t$.

On this diagonal, $x=a+(b-a)t$ and $y=c+(d-c)t$, and the two bounds coincide:
\[
    \frac{\mathbb E[XY]}{\mathbb E[X]\,\mathbb E[Y]}
    \le g(t):=\frac{ac+(bd-ac)t}{\bigl(a+(b-a)t\bigr)\bigl(c+(d-c)t\bigr)}.
\]
Differentiating gives
\[
    g'(t)=
    \frac{(b-a)(d-c)\bigl(ac(1-t)^2-bd\,t^2\bigr)}
         {\bigl(a+(b-a)t\bigr)^2\bigl(c+(d-c)t\bigr)^2}.
\]
Thus $g$ is maximized at $t^\star=\sqrt{ac}/(\sqrt{ac}+\sqrt{bd})$, where
\[
    g(t^\star)=\frac{(\sqrt{ac}+\sqrt{bd})^2}{(\sqrt{ad}+\sqrt{bc})^2}.
\]
Since $\mathbb E[XY]/(\mathbb E[X]\,\mathbb E[Y])\le g(t^\star)$, we have $\mathbb E[X]\,\mathbb E[Y]/\mathbb E[XY]\ge1/g(t^\star)$. Therefore
\begin{equation*}
    1-\frac{\mathbb E[X]\,\mathbb E[Y]}{\mathbb E[XY]}
    \le
    1-\frac{1}{g(t^\star)}
    =
    \frac{(\sqrt{ac}+\sqrt{bd})^2-(\sqrt{ad}+\sqrt{bc})^2}
         {(\sqrt{ac}+\sqrt{bd})^2}
    =
    \frac{(b-a)(d-c)}{(\sqrt{ac}+\sqrt{bd})^2}.
\end{equation*}
\hfill\Halmos

\proofgap

\noindent \textbf{Proof of Theorem~\ref{thm:endpoint}.} $\;$
We use the notation introduced before Lemma~\ref{lem:endpoint-localization}. The proof proceeds in two steps: the first step establishes a spectral bound on the multiplicative suboptimality using Lemmas~\ref{lem:quadratic-spectral} and~\ref{lem:endpoint-localization}; and the second step maximizes this spectral bound over admissible spectra using Lemma~\ref{lem:cassels-ratio}, which yields the bound~\eqref{eq:ratio-bound}.

For the first step, recall that the multiplicative suboptimality is the ratio of the worst-case additive suboptimality of the shrinkage path heuristic to that of the SAA--RO heuristic. Since $\rho_{\mathrm{gap}}$ maximizes $\Delta(\rho)$ over $\rho\ge0$, this ratio is
\begin{equation}
\label{eq:multiplicative-suboptimality-ratio}
\frac{\sup_{\rho\geq0}\Delta(\rho)}{\sup_{\rho\geq0} \Delta^{0,\infty}(\rho)}=\frac{\Delta(\rho_{\mathrm{gap}})}{{\sup_{\rho\geq0} \Delta^{0,\infty}(\rho)}},
\end{equation}
where the denominator is shown below to be positive whenever $\bm{x}(0)\neq\bm0$.

If $\bm{x}(0)=\bm0$, then the WDRO solution, the shrinkage path, and the SAA--RO heuristic are all identically $\bm0$, so both additive gaps are zero for every $\rho\ge0$; we interpret the multiplicative suboptimality as zero in this case, and the bound holds trivially. We therefore assume $\bm{x}(0)\neq\bm0$ in the remainder.

Lemma~\ref{lem:quadratic-reduction} gives $\bm{x}(\infty)=\bm0$ and the quadratic representation of $f_\rho$. Evaluating it at the two endpoints gives the SAA--RO value
\begin{equation}\label{eq:saa-ro-value-closed-form}
    \min\{f_\rho(\bm{x}(0)),f_\rho(\bm{x}(\infty))\}
    =f_0^\star+\min\{\rho\|\bm{x}(0)\|_2,\ \bm{x}(0)^\top\bm H\bm{x}(0)\}.
\end{equation}
In particular, for $\rho_{\mathrm{end}}=\bm{x}(0)^\top\bm H\bm{x}(0)/\|\bm{x}(0)\|_2$, the endpoint heuristic value is $f_0^\star+\bm{x}(0)^\top\bm H\bm{x}(0)$, while~\eqref{eq:interp-closed-form} gives $\hat f_{\rho_{\mathrm{end}}}=f_0^\star+\tfrac34\bm{x}(0)^\top\bm H\bm{x}(0)$. Hence the denominator in~\eqref{eq:multiplicative-suboptimality-ratio} is positive:
\begin{equation*}
\begin{aligned}
        \sup_{\rho\geq0}\Delta^{0,\infty}(\rho) &\ge \Delta^{0,\infty}(\rho_{\mathrm{end}})=\min\{f_{\rho_{\mathrm{end}}}(\bm{x}(0)),f_{\rho_{\mathrm{end}}}(\bm{x}(\infty))\}-f_{\rho_{\mathrm{end}}}^\star
        \\&\geq\min\{f_{\rho_{\mathrm{end}}}(\bm{x}(0)),f_{\rho_{\mathrm{end}}}(\bm{x}(\infty))\}-\hat{f}_{\rho_{\mathrm{end}}}=\tfrac14\bm{x}(0)^\top\bm{H}\bm{x}(0) >0,
\end{aligned}
\end{equation*}
where the last inequality uses $\bm H\succ\bm0$ and $\bm{x}(0)\neq\bm0$.

If $\Delta(\rho_{\mathrm{gap}})=0$, then the bound is immediate. We therefore assume $\Delta(\rho_{\mathrm{gap}})>0$. Lemma~\ref{lem:endpoint-localization} gives $\rho_{\mathrm{gap}}>\rho_{\mathrm{end}}$, so~\eqref{eq:saa-ro-value-closed-form} gives $\min\{f_{\rho_{\mathrm{gap}}}(\bm{x}(0)),f_{\rho_{\mathrm{gap}}}(\bm{x}(\infty))\}=f_0^\star+\bm{x}(0)^\top\bm H\bm{x}(0)$. Combining this with the identity $f_{\rho_{\mathrm{gap}}}^{\star}=f_0^{\star}+\bm{x}(0)^\top\bm H\bm{x}(0)-\bm{x}^{\star}(\rho_{\mathrm{gap}})^\top\bm H\bm{x}^{\star}(\rho_{\mathrm{gap}})$ from Lemma~\ref{lem:quadratic-spectral} lower bounds the denominator of~\eqref{eq:multiplicative-suboptimality-ratio}:
\begin{equation*}
    \sup_{\rho\geq0}\Delta^{0,\infty}(\rho)\geq\Delta^{0,\infty}(\rho_{\mathrm{gap}})=\bm{x}^{\star}(\rho_{\mathrm{gap}})^\top\bm H\bm{x}^{\star}(\rho_{\mathrm{gap}}).
\end{equation*}
Combining this lower bound with~\eqref{eq:spectral-form} gives
\begin{equation*}
\frac{\sup_{\rho\geq0}\Delta(\rho)}{\sup_{\rho\geq0} \Delta^{0,\infty}(\rho)}
\leq\frac{\Delta(\rho_{\mathrm{gap}})}{\Delta^{0,\infty}(\rho_{\mathrm{gap}})}
=1 - \frac{\|\bm{x}^{\star}(\rho_{\mathrm{gap}})\|^2_2}{\|\bm{x}(0)\|^2_2}
\frac{\bm{x}(0)^\top\bm{H}\,\bm{x}(0)}
{\bm{x}^\star(\rho_{\mathrm{gap}})^\top \bm{H}\, \bm{x}^\star(\rho_{\mathrm{gap}})}.
\end{equation*}
Using the spectral notation of Lemma~\ref{lem:quadratic-spectral} and $\tilde{x}^{\star}_i(\rho_{\mathrm{gap}})=\lambda_i(\lambda_i+\gamma_{\mathrm{gap}})^{-1}\tilde x_i(0)$, the last bound becomes
\begin{equation}\label{eq:ratio-relax}
\frac{\sup_{\rho\geq0}\Delta(\rho)}{\sup_{\rho\geq0} \Delta^{0,\infty}(\rho)}
\le
1-\frac{\mathbb E_p[\Lambda]\,\mathbb E_p[\phi_{\gamma_{\mathrm{gap}}}(\Lambda)]}
        {\mathbb E_p[\Lambda\phi_{\gamma_{\mathrm{gap}}}(\Lambda)]}.
\end{equation}

For the second step, we maximize the spectral representation in~\eqref{eq:ratio-relax}. Since $\phi_{\gamma_{\mathrm{gap}}}$ is increasing, Lemma~\ref{lem:cassels-ratio} applies with $X=\Lambda$ and $Y=\phi_{\gamma_{\mathrm{gap}}}(\Lambda)$. Thus, with $a=\lambda_{\min}$, $b=\lambda_{\max}$, $c=\phi_{\gamma_{\mathrm{gap}}}(\lambda_{\min})$, and $d=\phi_{\gamma_{\mathrm{gap}}}(\lambda_{\max})$, we obtain
\begin{equation}\label{eq:ratio-fixed-gamma-bound}
\frac{\sup_{\rho\geq0}\Delta(\rho)}{\sup_{\rho\geq0} \Delta^{0,\infty}(\rho)}
\le
\frac{(\lambda_{\max}-\lambda_{\min})\bigl[\phi_{\gamma_{\mathrm{gap}}}(\lambda_{\max})-\phi_{\gamma_{\mathrm{gap}}}(\lambda_{\min})\bigr]}
     {\bigl(\lambda_{\min}^{3/2}/(\lambda_{\min}+\gamma_{\mathrm{gap}})+\lambda_{\max}^{3/2}/(\lambda_{\max}+\gamma_{\mathrm{gap}})\bigr)^2}
=\frac{\eta(\kappa-1)^2\bigl(2\kappa+\eta(\kappa+1)\bigr)}
       {\bigl(\kappa+\eta+\kappa^{3/2}(1+\eta)\bigr)^2},
\end{equation}
where $\eta:=\gamma_{\mathrm{gap}}/\lambda_{\min}$. It remains to maximize the $\eta$-parameterized bound over $\eta\ge0$. For fixed $\kappa\ge1$, write this bound as
$R_\kappa(\eta):=\frac{\eta(\kappa-1)^2\bigl(2\kappa+\eta(\kappa+1)\bigr)}{\bigl(\kappa+\eta+\kappa^{3/2}(1+\eta)\bigr)^2}$. Differentiating $R_\kappa(\eta)$ gives
\[
\frac{\mathrm{d}R_\kappa(\eta)}{\mathrm{d}\eta}
=\frac{(\kappa-1)^2\left(2\kappa(\kappa+\kappa^{3/2})+2\kappa(\kappa+\sqrt{\kappa})\eta\right)}{\bigl(\kappa+\eta+\kappa^{3/2}(1+\eta)\bigr)^3}
\geq0\qquad \forall \kappa\geq1,\ \eta\geq0.
\]
Thus $R_\kappa(\eta)$ is nondecreasing on $[0,\infty)$, and hence
\begin{equation*}
    \sup_{\eta\ge0}R_\kappa(\eta)=\lim_{\eta\to\infty}R_\kappa(\eta)=\frac{(\kappa-1)^2(\kappa+1)}{(\kappa^{3/2}+1)^2}.
\end{equation*}
Combining this supremum with~\eqref{eq:ratio-fixed-gamma-bound} yields the bound \eqref{eq:ratio-bound} in Theorem~\ref{thm:endpoint}. Finally, since the denominator of \eqref{eq:ratio-bound} is larger than the numerator by
\[
(\kappa^{3/2}+1)^2-(\kappa-1)^2(\kappa+1)=\kappa(\sqrt{\kappa}+1)^2>0,
\]
the bound \eqref{eq:ratio-bound} is strictly less than one for every $\kappa\ge1$.
\hfill \Halmos

\proofgap
\noindent \textbf{Proof of Theorem~\ref{thm:lb-algorithm}.} $\;$
Throughout, write $D := W(\Pnom,\Qprob(\infty))$ and $\Qprob_\lambda := (1-\lambda) \Pnom+\lambda \Qprob(\infty)$, so that $\widehat{\mathcal{B}}(\Pnom)=\{\Qprob_\lambda:\lambda\in[0,1]\}$. The existence and finite support of $\Qprob(\infty)$ are established in the GitHub~companion. The proof proceeds in three steps: the first step handles the trivial case $\Pnom = \Qprob(\infty)$ and otherwise shows that~\eqref{eq:lower-bound} is a maximization over those $\lambda\in[0,1]$ for which $\Qprob_\lambda$ is feasible; the second step shows that the objective of this maximization is nondecreasing in $\lambda$, so that the largest feasible $\lambda$ is optimal and yields $\widehat{\Qprob}(\rho)$; and the third step verifies that the infimum in~\eqref{eq:Qhat} is attained.

For the first step, we begin with the trivial case $\Pnom = \Qprob(\infty)$. In this case, $\Qprob_\lambda = \Pnom$ for every $\lambda \in [0,1]$, so the dual shrinkage path collapses to $\{\Pnom\}$ and $\widehat{\Qprob}(\rho) = \Pnom$ for every $\rho \ge 0$. Since $W(\Pnom,\Pnom) = 0 \le \rho$, the intersection in~\eqref{eq:lower-bound} equals $\{\Pnom\}$, and thus $\underline{J}(\rho) = \min_{\bm{x}\in\mathcal{X}(\Xi)} \E_{\Pnom}[\Loss(\bm{x},\chg{\bm{\tilde{\xi}}})]$, attained at the SAA solution $\bm{x}(0)$. We therefore assume $\Pnom \neq \Qprob(\infty)$ in the remainder of the proof, so that $D > 0$ since the Wasserstein distance is a metric.

Kantorovich--Rubinstein duality for the type-1 Wasserstein distance gives, for every $\lambda\in[0,1]$,
\begin{equation*}
    W(\Pnom,\Qprob_\lambda)
    = \sup_{\|f\|_{\Lip}\le 1}\left\{\int f\,\mathrm{d}\Pnom - \int f\,\mathrm{d}\Qprob_\lambda\right\}
    = \lambda\sup_{\|f\|_{\Lip}\le 1}\left\{\int f\,\mathrm{d}\Pnom - \int f\,\mathrm{d}\Qprob(\infty)\right\}
    = \lambda D.
\end{equation*}
Since $\mathcal{P}(\Xi)$ is convex and $\Pnom, \Qprob(\infty) \in \mathcal{P}(\Xi)$, every $\Qprob_\lambda$ belongs to $\mathcal{P}(\Xi)$. Hence $\Qprob_\lambda \in \Amb(\Pnom)$ if and only if $W(\Pnom,\Qprob_\lambda)\le\rho$, which is equivalent to $\lambda D\le\rho$, that is, to $\lambda\le\rho/D$. Intersecting with $\lambda\in[0,1]$ yields $\widehat{\mathcal{B}}(\Pnom)\cap\Amb(\Pnom) = \{\Qprob_\lambda : \lambda\in[0,\overline\lambda]\}$ with $\overline\lambda := \min\{\rho/D,\,1\}$. Therefore, the a posteriori lower bound~\eqref{eq:lower-bound} is equivalent to the univariate maximization problem $\underline{J}(\rho)=\max_{\lambda\in[0,\overline\lambda]}\phi(\lambda)$, where $\phi(\lambda) := \inf_{\bm{x}\in\mathcal{X}(\Xi)} \E_{\Qprob_\lambda}[\Loss(\bm{x},\chg{\bm{\tilde{\xi}}})]$.

For the second step, fix $\bm{x}$. The map $\lambda\mapsto\E_{\Qprob_\lambda}[\Loss(\bm{x},\chg{\bm{\tilde{\xi}}})]$ is affine, so $\phi$ is concave on $[0,1]$ as a pointwise infimum of affine functions. Also, $\Qprob_1=\Qprob(\infty)$ maximizes $\inf_{\bm{x}\in\mathcal{X}(\Xi)}\E_{\Qprob}[\Loss(\bm{x},\chg{\bm{\tilde{\xi}}})]$ over $\mathcal{P}(\Xi)$ by~\eqref{eq:Qinf}; hence $\phi(\lambda)\le\phi(1)$ for every $\lambda\in[0,1]$. If $0\le\lambda_1<\lambda_2\le1$ and $t:=(\lambda_2-\lambda_1)/(1-\lambda_1)$, then $\lambda_2=(1-t)\lambda_1+t$ and concavity gives $\phi(\lambda_2)\ge(1-t)\phi(\lambda_1)+t\phi(1)\ge\phi(\lambda_1)$. Thus $\phi$ is nondecreasing, and the maximum is attained at $\overline\lambda=\min\{\rho/D,1\}=1-\mu(\rho)$. Substituting this value into $\Qprob_{\overline\lambda}$ gives $\Qprob_{\overline\lambda}=\mu(\rho) \Pnom+(1-\mu(\rho)) \Qprob(\infty)=\widehat{\Qprob}(\rho)$, and therefore $\underline{J}(\rho)=\inf_{\bm{x}\in\mathcal{X}(\Xi)}\E_{\widehat{\Qprob}(\rho)}[\Loss(\bm{x},\chg{\bm{\tilde{\xi}}})]$.

For the third step, it remains to replace the infimum in this representation of $\underline{J}(\rho)$ by a minimum. Since $\Qprob(\infty)$ is an RO distribution, an auxiliary result in the GitHub companion gives $\inf_{\bm{x}\in\mathcal{X}(\Xi)}\E_{\Qprob(\infty)}[\Loss(\bm{x},\chg{\bm{\tilde{\xi}}})]=J(\infty)$, where $J(\infty)$ is the finite optimal value of the RO problem~\eqref{eq:extremes:ro}. Thus $\E_{\Qprob(\infty)}[\Loss(\bm{x},\chg{\bm{\tilde{\xi}}})]\ge J(\infty)$ for every $\bm{x}\in\mathcal{X}(\Xi)$. If $\overline\lambda = 1$, then $\widehat{\Qprob}(\rho) = \Qprob(\infty)$, and the RO solution $\bm{x}(\infty)$ satisfies $\E_{\Qprob(\infty)}[\Loss(\bm{x}(\infty),\chg{\bm{\tilde{\xi}}})]\le\sup_{\bm{\xi}\in\Xi}\Loss(\bm{x}(\infty),\bm{\xi}) = J(\infty)$, so $\bm{x}(\infty)$ attains the infimum. If $\overline\lambda<1$, then $\mu(\rho)>0$. For every $\bm{x}\in\mathcal{X}(\Xi)$,
\begin{equation*}
    \begin{aligned}
        \E_{\widehat{\Qprob}(\rho)}[\Loss(\bm{x},\chg{\bm{\tilde{\xi}}})]
        &= \mu(\rho)\E_{\Pnom}[\Loss(\bm{x},\chg{\bm{\tilde{\xi}}})] + (1-\mu(\rho))\E_{\Qprob(\infty)}[\Loss(\bm{x},\chg{\bm{\tilde{\xi}}})] \\
        &\ge \mu(\rho)\E_{\Pnom}[\Loss(\bm{x},\chg{\bm{\tilde{\xi}}})] + (1-\mu(\rho))J(\infty).
    \end{aligned}
\end{equation*}
Assumption~\ref{ass:aposteriori}\emph{(i)} makes the right-hand side coercive on $\mathcal{X}(\Xi)$, and hence the objective is coercive as well. The distribution $\widehat{\Qprob}(\rho)$ is finitely supported, so the objective is a finite nonnegative weighted sum of lower semicontinuous functions $\Loss(\cdot,\bm{\xi})$ and is therefore lower semicontinuous. The set $\mathcal{X}(\Xi)$ is closed by the closedness of $\mathcal{X}$ and of each constraint $g_m(\cdot,\bm{\xi})\le\gamma_m$. Weierstrass' theorem therefore gives an optimizer, and the infimum is a minimum.\hfill \Halmos

\proofgap

\noindent \textbf{Proof of Theorem~\ref{thm:posteriori}.} $\;$
We prove the three claims in order. Throughout, write $D:=W(\Pnom,\Qprob(\infty))$.

For monotonicity, let $0\le\rho_1\le\rho_2$. Then $\mathbb{B}_{\rho_1}(\Pnom)\subseteq\mathbb{B}_{\rho_2}(\Pnom)$. For every fixed $\bm{x}$, the inner worst-case expectation $\sup_{\Qprob\in\mathbb{B}_\rho(\Pnom)}\E_{\Qprob}[\Loss(\bm{x},\chg{\bm{\tilde{\xi}}})]$ is therefore nondecreasing in $\rho$. Taking the minimum over the $\rho$-independent sets $\mathcal{X}(\Xi)$ and $\widehat{\mathcal{X}}(\Xi)$ shows that $J^\star(\rho)$ and $\overline{J}(\rho)$ are nondecreasing. For the lower bound, the feasible sets in~\eqref{eq:lower-bound} also nest:
$\widehat{\mathcal{B}}(\Pnom)\cap\mathbb{B}_{\rho_1}(\Pnom)\subseteq\widehat{\mathcal{B}}(\Pnom)\cap\mathbb{B}_{\rho_2}(\Pnom)$. Maximizing the same objective over a larger set cannot decrease the value, so $\underline{J}(\rho)$ is nondecreasing.

For concavity, we first consider $J^\star$ and $\overline{J}$. The strong duality of the Wasserstein worst-case expectation \citep[Theorem~4.2]{mohajerin2018data} gives, for every fixed $\bm{x}$ and every radius $\rho>0$,
\begin{equation*}
    \sup_{\Qprob\in\mathbb{B}_\rho(\Pnom)}\E_{\Qprob}[\Loss(\bm{x},\chg{\bm{\tilde{\xi}}})]=\inf_{\eta\ge0}\Big\{\eta\rho + \tfrac{1}{N}\sum_{i=1}^N\sup_{\bm{\xi}\in\Xi}\big[\Loss(\bm{x},\bm{\xi})-\eta\,d(\bm{\xi},\chg{\bm{\hat{\xi}}}_i)\big]\Big\}.
\end{equation*}
Thus the worst-case expectation is an infimum of functions affine in $\rho$. Taking the further infimum over $\bm{x}\in\mathcal{X}(\Xi)$ gives $J^\star(\rho)$, and taking it over $\bm{x}\in\widehat{\mathcal{X}}(\Xi)$ gives $\overline{J}(\rho)$. Since both feasible sets are independent of $\rho$, both value functions are pointwise infima of affine functions of $\rho$, and hence concave on $(0,\infty)$. Concavity on $[0,\infty)$ follows because both functions are nondecreasing by claim~\emph{(i)}, and a nondecreasing function on $[0,\infty)$ that is concave on $(0,\infty)$ is concave throughout. For $\underline{J}$, if $D=0$, then $\underline{J}$ is constant---and hence concave---by Theorem~\ref{thm:lb-algorithm}. If $D>0$, the proof of Theorem~\ref{thm:lb-algorithm} shows that $\underline{J}(\rho)=\phi(\min\{\rho/D,1\})$, where $\phi(\lambda):=\inf_{\bm{x}\in\mathcal{X}(\Xi)}\E_{(1-\lambda) \Pnom+\lambda \Qprob(\infty)}[\Loss(\bm{x},\chg{\bm{\tilde{\xi}}})]$ is concave and nondecreasing on $[0,1]$. The map $\rho\mapsto\min\{\rho/D,1\}$ is concave, and composing a nondecreasing concave function with a concave function preserves concavity. Hence $\underline{J}$ is concave.

For endpoint tightness, first take $\rho=0$. The Wasserstein ball is then the singleton $\mathbb{B}_0(\Pnom)=\{\Pnom\}$, so $J^\star(0)=\min_{\bm{x}\in\mathcal{X}(\Xi)}\E_{\Pnom}[\Loss(\bm{x},\chg{\bm{\tilde{\xi}}})]=J(0)$, which equals the optimal value of the SAA problem~\eqref{eq:extremes:saa}. Since $\bm{x}(0)\in\widehat{\mathcal{X}}(\Xi)$ is SAA-optimal over the larger set $\mathcal{X}(\Xi)$, we also have $\overline{J}(0)=J(0)$. The dual shrinkage path intersects $\mathbb{B}_0(\Pnom)$ only at $\Pnom$, so $\underline{J}(0)=J(0)$ as well. Thus $\overline{J}(0)=J^\star(0)=\underline{J}(0)=J(0)$.

Finally, let $\rho\ge D$. Theorem~\ref{thm:lb-algorithm} gives $\widehat{\Qprob}(\rho)=\Qprob(\infty)$, with the degenerate case $D=0$ covered by $\Pnom=\Qprob(\infty)$, so $\underline{J}(\rho)=\min_{\bm{x}\in\mathcal{X}(\Xi)}\E_{\Qprob(\infty)}[\Loss(\bm{x},\chg{\bm{\tilde{\xi}}})]$, which equals $J(\infty)$ by the same auxiliary result from the GitHub companion invoked in the proof of Theorem~\ref{thm:lb-algorithm}. Note that $\mathbb{B}_\rho(\Pnom)\subseteq\mathcal{P}(\Xi)$ implies
\[
    J^\star(\rho)\le\min_{\bm{x}\in\mathcal{X}(\Xi)}\sup_{\bm{\xi}\in\Xi}\Loss(\bm{x},\bm{\xi})=J(\infty).
\]
Since also $\underline{J}(\rho)\le J^\star(\rho)$ by~\eqref{eq:lower-bound} and $\underline{J}(\rho)=J(\infty)$, we have $J^\star(\rho)=J(\infty)$. For the upper bound, evaluating the path problem at $\bm{x}(\infty)\in\widehat{\mathcal{X}}(\Xi)$ gives $\overline{J}(\rho)\le\sup_{\Qprob\in\mathbb{B}_\rho(\Pnom)}\E_{\Qprob}[\Loss(\bm{x}(\infty),\chg{\bm{\tilde{\xi}}})]\le J(\infty)$. The opposite inequality $\overline{J}(\rho)\ge J^\star(\rho)=J(\infty)$ follows from $\widehat{\mathcal{X}}(\Xi)\subseteq\mathcal{X}(\Xi)$. Hence $\overline{J}(\rho)=J^\star(\rho)=\underline{J}(\rho)=J(\infty)$ for every $\rho\ge D$.
\hfill \Halmos

\proofgap
\noindent The proof of Proposition~\ref{prop:saddle-summary} relies on the following three lemmas, one for each of the settings~\emph{(i)}--\emph{(iii)}. Each lemma exhibits a finite convex problem whose solution yields a finitely supported RO distribution $\Qprob(\infty)$ defined in~\eqref{eq:Qinf}.

\begin{lemma}[Vertex Construction]\label{lem:ro-vertex-construction}
Let the support $\Xi = \conv(\bm{v}_1,\ldots,\bm{v}_K) + \cone(\bm{r}_1,\ldots,\bm{r}_R)$ be a pointed polyhedron with vertices $\bm{v}_1,\ldots,\bm{v}_K$ and extreme rays $\bm{r}_1,\ldots,\bm{r}_R$, and let the loss $\Loss(\bm{x},\cdot)$ be convex on $\Xi$ with nonpositive recession along every extreme ray, that is,
\[
    \chg{\Loss^\infty}(\bm{x},\bm{r}_r) \coloneqq \lim_{\alpha\to\infty}\alpha^{-1}\big[\Loss(\bm{x},\bm{\xi}+\alpha\bm{r}_r)-\Loss(\bm{x},\bm{\xi})\big] \le 0
    \qquad\forall\,\bm{x}\in\mathcal{X}(\Xi),\ \bm{\xi}\in\Xi,\ r\in[R].
\]
Then the RO problem~\eqref{eq:extremes:ro} admits the finite vertex reformulation
\begin{equation}\label{eq:vertex-reformulation}
    J(\infty)=\min_{\bm{x}\in\mathcal{X}(\Xi),\,t}\;\big\{\,t \;:\; \Loss(\bm{x},\bm{v}_j)\le t,\ j\in[K]\,\big\},
\end{equation}
and a finitely supported RO distribution can be obtained from its dual.
\end{lemma}
\noindent \textbf{Proof of Lemma~\ref{lem:ro-vertex-construction}.} $\;$
Fix $\bm{x}\in\mathcal{X}(\Xi)$ and write any $\bm{\xi}\in\Xi$ as $\bm{\xi}=\bar{\bm{\xi}}+\sum_{r}\beta_r\bm{r}_r$ with $\bar{\bm{\xi}}=\sum_j\alpha_j\bm{v}_j$, $\bm{\alpha}\in\Delta_K$, and $\beta_r\ge0$, where $\Delta_K$ denotes the probability simplex on $[K]$. Nonpositive recession gives $\Loss(\bm{x},\bm{\xi})\le\Loss(\bm{x},\bar{\bm{\xi}})$, while convexity gives $\Loss(\bm{x},\bar{\bm{\xi}})\le\sum_j\alpha_j\Loss(\bm{x},\bm{v}_j)\le\max_{j\in[K]}\Loss(\bm{x},\bm{v}_j)$; since each vertex lies in $\Xi$, the reverse inequality is immediate, so $\sup_{\bm{\xi}\in\Xi}\Loss(\bm{x},\bm{\xi})=\max_{j\in[K]}\Loss(\bm{x},\bm{v}_j)$. The RO problem~\eqref{eq:extremes:ro} therefore reduces to the finite vertex reformulation~\eqref{eq:vertex-reformulation}.

Dualizing the epigraph constraints of~\eqref{eq:vertex-reformulation} with multipliers $\bm{\pi}\ge\bm{0}$ and minimizing the Lagrangian over $t$ forces $\sum_j\pi_j=1$, giving the dual $\max_{\bm{\pi}\in\Delta_K}\inf_{\bm{x}\in\mathcal{X}(\Xi)}\sum_j\pi_j\Loss(\bm{x},\bm{v}_j)$. As any $t>J(\infty)$ is strictly feasible at $\bm{x}(\infty)$, strong duality holds and an optimal $\bm{\pi}^\star\in\Delta_K$ exists with $\inf_{\bm{x}\in\mathcal{X}(\Xi)}\sum_j\pi_j^\star\Loss(\bm{x},\bm{v}_j)=J(\infty)$. Placing the optimal mass on the active vertices yields the candidate RO distribution
\[
    \Qprob(\infty)=\sum_{j:\,\pi_j^\star>0}\pi_j^\star\,\delta_{\bm{v}_j}.
\]
This is indeed an RO distribution: $\inf_{\bm{x}\in\mathcal{X}(\Xi)}\E_{\Qprob(\infty)}[\Loss(\bm{x},\chg{\bm{\tilde{\xi}}})]=\inf_{\bm{x}\in\mathcal{X}(\Xi)}\sum_j\pi_j^\star\Loss(\bm{x},\bm{v}_j)=J(\infty)$, whereas every $\Qprob\in\mathcal{P}(\Xi)$ satisfies $\inf_{\bm{x}\in\mathcal{X}(\Xi)}\E_{\Qprob}[\Loss(\bm{x},\chg{\bm{\tilde{\xi}}})]\le\E_{\Qprob}[\Loss(\bm{x}(\infty),\chg{\bm{\tilde{\xi}}})]\le\sup_{\bm{\xi}\in\Xi}\Loss(\bm{x}(\infty),\bm{\xi})=J(\infty)$, so $\Qprob(\infty)$ attains the maximum in~\eqref{eq:Qinf} and is finitely supported.
\hfill\Halmos

\proofgap
\begin{lemma}[Finite-Max Construction]\label{lem:ro-finite-max-construction}
In addition to the assumptions of Lemma~\ref{lem:ro-vertex-construction}, let the feasible set be the polyhedron $\mathcal{X}(\Xi)=\{\bm{x}\in\R^n:C\bm{x}\le\bm{d}\}$, and let $\Loss(\bm{x},\bm{\xi})=\max_{l\in[L]}\chg{\Loss_l}(\bm{x},\bm{\xi})$, where each $\chg{\Loss_l}(\cdot,\bm{\xi})$ is convex and differentiable. Define the active vertices, the active loss pieces at each vertex, and the active constraints at the RO solution $(\bm{x}(\infty),J(\infty))$ as
\[
    \mathcal{A}=\{j:\Loss(\bm{x}(\infty),\bm{v}_j)=J(\infty)\},\quad
    \mathcal{L}_j=\{l:\chg{\Loss_l}(\bm{x}(\infty),\bm{v}_j)=J(\infty)\},\quad
    \mathcal{I}_\infty=\{m:C_m\bm{x}(\infty)=d_m\}.
\]
Then the feasibility linear program
\begin{equation}\label{eq:feasibility-lp-lifting}
    \sum_{j\in\mathcal{A}}\sum_{l\in\mathcal{L}_j}\pi_{jl}\,\nabla_{\bm{x}}\chg{\Loss_l}(\bm{x}(\infty),\bm{v}_j)
    +C_{\mathcal{I}_\infty}{}^{\top}\bm{\mu}=\bm{0},\qquad
    \sum_{j\in\mathcal{A}}\sum_{l\in\mathcal{L}_j}\pi_{jl}=1,
    \qquad
    \pi_{jl}\ge0,\;\bm{\mu}\ge\bm{0}
\end{equation}
is feasible, and every basic feasible solution induces an RO distribution supported on at most $n+1$~atoms.
\end{lemma}
\noindent \textbf{Proof of Lemma~\ref{lem:ro-finite-max-construction}.} $\;$
The argument in Lemma~\ref{lem:ro-vertex-construction} gives $\sup_{\bm{\xi}\in\Xi}\Loss(\bm{x},\bm{\xi})=\max_{j\in[K]}\max_{l\in[L]}\chg{\Loss_l}(\bm{x},\bm{v}_j)$, so $\bm{x}(\infty)$ minimizes this finite maximum over $\mathcal{X}(\Xi)$, and the active terms of the maximum at $\bm{x}(\infty)$ are exactly the pairs $\{(j,l):j\in\mathcal{A},\,l\in\mathcal{L}_j\}$.

The linear program~\eqref{eq:feasibility-lp-lifting} is feasible. Indeed, since $\bm{x}(\infty)$ minimizes the finite maximum over the polyhedron $\mathcal{X}(\Xi)$, first-order optimality gives $\bm{0}\in\partial_{\bm{x}}\big(\max_{j,l}\chg{\Loss_l}(\cdot,\bm{v}_j)\big)(\bm{x}(\infty))+N_{\mathcal{X}(\Xi)}(\bm{x}(\infty))$, where $N_{\mathcal{X}(\Xi)}(\bm{x})\coloneqq\{\bm{z}\in\R^n:\bm{z}^\top(\bm{y}-\bm{x})\le0\;\;\forall\bm{y}\in\mathcal{X}(\Xi)\}$ denotes the normal cone to $\mathcal{X}(\Xi)$ at $\bm{x}$. The subdifferential of a finite maximum of differentiable convex functions is the convex hull of the active gradients, $\operatorname{conv}\{\nabla_{\bm{x}}\chg{\Loss_l}(\bm{x}(\infty),\bm{v}_j):j\in\mathcal{A},\,l\in\mathcal{L}_j\}$, and the normal cone to the polyhedron is $\{C_{\mathcal{I}_\infty}{}^\top\bm{\mu}:\bm{\mu}\ge\bm{0}\}$; combining the two identities produces nonnegative $\pi_{jl}$ summing to one together with a $\bm{\mu}\ge\bm{0}$ that satisfy~\eqref{eq:feasibility-lp-lifting}.

Any feasible solution of~\eqref{eq:feasibility-lp-lifting} induces an RO distribution. Let $(\pi_{jl}^S,\bm{\mu}^S)$ be such a solution, put $\pi_j^S:=\sum_{l\in\mathcal{L}_j}\pi_{jl}^S$, and define $\Qprob^S(\infty):=\sum_{j\in\mathcal{A}:\,\pi_j^S>0}\pi_j^S\,\delta_{\bm{v}_j}$. For each such $j$ the averaged gradient $\bm{g}_j^S:=(\pi_j^S)^{-1}\sum_{l\in\mathcal{L}_j}\pi_{jl}^S\nabla_{\bm{x}}\chg{\Loss_l}(\bm{x}(\infty),\bm{v}_j)$ is a convex combination of active gradients and hence lies in $\partial_{\bm{x}}\Loss(\bm{x}(\infty),\bm{v}_j)$, while $\bm{z}^S:=C_{\mathcal{I}_\infty}{}^\top\bm{\mu}^S\in N_{\mathcal{X}(\Xi)}(\bm{x}(\infty))$, and~\eqref{eq:feasibility-lp-lifting} reads $\sum_{j\in\mathcal{A}}\pi_j^S\bm{g}_j^S+\bm{z}^S=\bm{0}$. For any $\bm{x}\in\mathcal{X}(\Xi)$ the subgradient inequality gives $\Loss(\bm{x},\bm{v}_j)\ge J(\infty)+(\bm{g}_j^S)^\top(\bm{x}-\bm{x}(\infty))$ for every $j\in\mathcal{A}$; weighting by $\pi_j^S$ and summing gives
\[
    \E_{\Qprob^S(\infty)}[\Loss(\bm{x},\chg{\bm{\tilde{\xi}}})]
    \;\ge\; J(\infty)+\Big(\textstyle\sum_{j\in\mathcal{A}}\pi_j^S\bm{g}_j^S\Big)^{\!\top}(\bm{x}-\bm{x}(\infty))
    \;=\; J(\infty)-(\bm{z}^S)^\top(\bm{x}-\bm{x}(\infty))
    \;\ge\; J(\infty),
\]
the last step because $\bm{z}^S\in N_{\mathcal{X}(\Xi)}(\bm{x}(\infty))$. Equality at $\bm{x}=\bm{x}(\infty)$ gives $\inf_{\bm{x}\in\mathcal{X}(\Xi)}\E_{\Qprob^S(\infty)}[\Loss(\bm{x},\chg{\bm{\tilde{\xi}}})]=J(\infty)$, and since every $\Qprob\in\mathcal{P}(\Xi)$ obeys $\inf_{\bm{x}\in\mathcal{X}(\Xi)}\E_{\Qprob}[\Loss(\bm{x},\chg{\bm{\tilde{\xi}}})]\le\E_{\Qprob}[\Loss(\bm{x}(\infty),\chg{\bm{\tilde{\xi}}})]\le\sup_{\bm{\xi}\in\Xi}\Loss(\bm{x}(\infty),\bm{\xi})=J(\infty)$, the distribution $\Qprob^S(\infty)$ is an RO distribution.

Finally, the bound on the number of atoms follows from the structure of a basic feasible solution. The system~\eqref{eq:feasibility-lp-lifting} has $n+1$ equality constraints---$n$ stationarity equations and one normalization---so a basic feasible solution has at most $n+1$ positive entries among the variables $\{\pi_{jl}\}\cup\{\bm{\mu}\}$, hence at most $n+1$ positive $\pi_{jl}^S$ and therefore at most $n+1$ atoms $\bm{v}_j$ in $\Qprob^S(\infty)$, matching the bound established in the GitHub companion.
\hfill\Halmos

\proofgap
\begin{lemma}[Convex--Concave Construction]\label{lem:ro-convex-concave-construction}
Let the support $\Xi$ be convex and compact, and let $\Loss(\bm{x},\bm{\xi})=\max_{j\in[K]}\Loss_j(\bm{x},\bm{\xi})$, where each $\Loss_j$ is closed convex with full domain in $\bm{x}$ and concave and upper semicontinuous in $\bm{\xi}$ on $\Xi$. Write $\Loss_j^\ast(\bm{s},\bm{\xi})\coloneqq\sup_{\bm{x}\in\R^n}\{\bm{s}^\top\bm{x}-\Loss_j(\bm{x},\bm{\xi})\}$ for the Fenchel conjugate of each piece in its first argument, $\sigma_{\mathcal{X}(\Xi)}(\bm{z})\coloneqq\sup_{\bm{x}\in\mathcal{X}(\Xi)}\bm{z}^\top\bm{x}$ for the support function of $\mathcal{X}(\Xi)$, and $\Delta_K$ for the probability simplex on $[K]$. Then the RO problem~\eqref{eq:extremes:ro} admits the finite convex reformulation
\begin{equation}\label{eq:dual-piecewise-concave}
    J(\infty)
    =
    \max_{\bm{\pi}\in\Delta_K,\,\bm{\zeta}_j/\pi_j\in\Xi,\,\bm{s}_j\in\R^n}
    \left\{-\sigma_{\mathcal{X}(\Xi)}\Big(-\sum_{j=1}^K\bm{s}_j\Big)-\sum_{j=1}^K\pi_j\,\Loss_j^\ast\big(\bm{s}_j/\pi_j,\bm{\zeta}_j/\pi_j\big)\right\},
\end{equation}
with the standard closed-perspective convention at $\pi_j=0$, under which the $j$th term of the sum equals $0$ if $(\bm{s}_j,\bm{\zeta}_j)=(\bm{0},\bm{0})$ and $+\infty$ otherwise, and the constraint $\bm{\zeta}_j/\pi_j\in\Xi$ reduces to $\bm{\zeta}_j=\bm{0}$, and every maximizer of~\eqref{eq:dual-piecewise-concave} induces a finitely supported RO distribution.
\end{lemma}
\noindent \textbf{Proof of Lemma~\ref{lem:ro-convex-concave-construction}.} $\;$
For each $\bm{x}$ we have $\sup_{\bm{\xi}\in\Xi}\Loss(\bm{x},\bm{\xi})=\max_{j\in[K]}\phi_j(\bm{x})$ with $\phi_j(\bm{x})\coloneqq\max_{\bm{\xi}\in\Xi}\Loss_j(\bm{x},\bm{\xi})$, where the maxima are attained because $\Xi$ is compact and $\Loss_j(\bm{x},\cdot)$ is upper semicontinuous, and each $\phi_j$ is real-valued and convex in $\bm{x}$, hence continuous. Lifting the finite maximum to the simplex gives $\max_j\phi_j(\bm{x})=\max_{\bm{\pi}\in\Delta_K}\sum_j\pi_j\phi_j(\bm{x})$, and Sion's minimax theorem \citep{sion1958minimax} over the convex set $\mathcal{X}(\Xi)$ and the convex compact set $\Delta_K$ yields $J(\infty)=\max_{\bm{\pi}\in\Delta_K}\inf_{\bm{x}\in\mathcal{X}(\Xi)}\sum_j\pi_j\phi_j(\bm{x})$. For fixed $\bm{\pi}$, separability gives $\sum_j\pi_j\phi_j(\bm{x})=\max_{\bm{\xi}_1,\ldots,\bm{\xi}_K\in\Xi}\sum_j\pi_j\Loss_j(\bm{x},\bm{\xi}_j)$, an objective convex in $\bm{x}$ and concave and upper semicontinuous in $(\bm{\xi}_1,\ldots,\bm{\xi}_K)$. A second application of Sion over the convex compact set $\Xi^K$ swaps the inner infimum and maximum, and combining the two interchanges yields
\begin{equation}\label{eq:piecewise-outer-max}
    J(\infty)=\max_{\bm{\pi}\in\Delta_K,\,\bm{\xi}_1,\ldots,\bm{\xi}_K\in\Xi}\;
    \inf_{\bm{x}\in\mathcal{X}(\Xi)}\;
    \sum_{j=1}^K\pi_j\Loss_j(\bm{x},\bm{\xi}_j).
\end{equation}
At any maximizer $(\bm{\pi}^\star,\bm{\xi}_1^\star,\ldots,\bm{\xi}_K^\star)$ the inner infimum equals $J(\infty)$ and is attained at $\bm{x}(\infty)$, since $\sum_j\pi_j^\star\Loss_j(\bm{x}(\infty),\bm{\xi}_j^\star)\le\sup_{\bm{\xi}\in\Xi}\Loss(\bm{x}(\infty),\bm{\xi})=J(\infty)$.

Fix $(\bm{\pi},\bm{\xi}_1,\ldots,\bm{\xi}_K)\in\Delta_K\times\Xi^K$ with $\bm{\pi}>\bm{0}$ and dualize the inner minimization of $\sum_j\pi_j\Loss_j(\bm{x},\bm{\xi}_j)$ over $\mathcal{X}(\Xi)$; zero-weight pieces drop from this sum and are reinstated below via the closed-perspective convention with $(\bm{s}_j,\bm{\zeta}_j)=(\bm{0},\bm{0})$. Each summand is proper closed convex with full domain $\R^n$, so the objective function $\sum_{j}\pi_j\Loss_j(\cdot,\bm{\xi}_j)$ is finite convex on all of $\R^n$. The indicator function of $\mathcal{X}(\Xi)$ equals $0$ on $\mathcal{X}(\Xi)$ and $+\infty$ outside. Adding the indicator function to the objective function enforces the constraint $\bm{x}\in\mathcal{X}(\Xi)$. The inner infimum in~\eqref{eq:piecewise-outer-max} therefore equals the minimization of this sum over all of $\R^n$. Fenchel--Rockafellar duality applies to this sum of the objective function and the indicator function defined by $\mathcal X(\Xi)$. Its qualification requires the relative interiors of their domains, $\R^n$ and $\mathcal{X}(\Xi)$, to intersect. This holds since $\mathcal{X}(\Xi)$ is nonempty, and weak duality reads
\[
    -\sigma_{\mathcal{X}(\Xi)}\Big(-\textstyle\sum_{j}\bm{s}_j\Big)-\sum_{j}\pi_j\Loss_j^\ast\big(\bm{s}_j/\pi_j,\bm{\xi}_j\big)\;\le\;\inf_{\bm{x}\in\mathcal{X}(\Xi)}\sum_{j}\pi_j\Loss_j(\bm{x},\bm{\xi}_j)
    \qquad\forall\,\bm{s}_1,\ldots,\bm{s}_K\in\R^n,
\]
with equality and dual attainment whenever the inner infimum is finite---in particular at a maximizer of~\eqref{eq:piecewise-outer-max}. Because $\Loss_j(\bm{x},\cdot)$ is concave, $\Loss_j^\ast(\bm{s},\bm{\xi})$ is jointly convex in $(\bm{s},\bm{\xi})$, so its perspective $(\pi_j,\bm{s}_j,\bm{\zeta}_j)\mapsto\pi_j\Loss_j^\ast(\bm{s}_j/\pi_j,\bm{\zeta}_j/\pi_j)$ is jointly convex; substituting $\bm{\zeta}_j=\pi_j\bm{\xi}_j$ turns the constraint $\bm{\xi}_j\in\Xi$ into $\bm{\zeta}_j/\pi_j\in\Xi$ and, maximizing the dual jointly with~\eqref{eq:piecewise-outer-max}, produces the convex reformulation~\eqref{eq:dual-piecewise-concave}. The maximization problem~\eqref{eq:dual-piecewise-concave} has a concave objective and a convex feasible set. Weak duality bounds its objective at every feasible point by the corresponding inner infimum in~\eqref{eq:piecewise-outer-max}, which is at most $J(\infty)$. At a maximizer of~\eqref{eq:piecewise-outer-max}, this inner infimum is finite and equals $J(\infty)$. Fenchel--Rockafellar duality therefore yields multipliers $\bm{s}_1,\ldots,\bm{s}_K$ for which the objective in~\eqref{eq:dual-piecewise-concave} attains $J(\infty)$. Thus, the optimal value of~\eqref{eq:dual-piecewise-concave} is $J(\infty)$ and is attained.

Let $(\bm{\pi}^\star,\bm{\zeta}_1^\star,\ldots,\bm{\zeta}_K^\star,\bm{s}_1^\star,\ldots,\bm{s}_K^\star)$ maximize~\eqref{eq:dual-piecewise-concave}, and for $\pi_j^\star>0$ set $\bm{\xi}_j^\star\coloneqq\bm{\zeta}_j^\star/\pi_j^\star\in\Xi$. Fenchel--Rockafellar duality gives $\inf_{\bm{x}\in\mathcal{X}(\Xi)}\sum_j\pi_j^\star\Loss_j(\bm{x},\bm{\xi}_j^\star)=J(\infty)$. Define
\[
    \Qprob(\infty)=\sum_{j:\,\pi_j^\star>0}\pi_j^\star\,\delta_{\bm{\xi}_j^\star}.
\]
For any $\bm{x}\in\mathcal{X}(\Xi)$, $\Loss=\max_r\Loss_r\ge\Loss_j$ yields $\E_{\Qprob(\infty)}[\Loss(\bm{x},\chg{\bm{\tilde{\xi}}})]=\sum_j\pi_j^\star\Loss(\bm{x},\bm{\xi}_j^\star)\ge\sum_j\pi_j^\star\Loss_j(\bm{x},\bm{\xi}_j^\star)$. Moreover, $\sum_j\pi_j^\star\Loss_j(\bm{x},\bm{\xi}_j^\star)\ge\inf_{\bm{y}\in\mathcal{X}(\Xi)}\sum_j\pi_j^\star\Loss_j(\bm{y},\bm{\xi}_j^\star)=J(\infty)$. Thus, $\E_{\Qprob(\infty)}[\Loss(\bm{x},\chg{\bm{\tilde{\xi}}})]\ge J(\infty)$ for every $\bm{x}\in\mathcal{X}(\Xi)$. Taking the infimum over $\bm{x}\in\mathcal{X}(\Xi)$ yields $\inf_{\bm{x}\in\mathcal{X}(\Xi)}\E_{\Qprob(\infty)}[\Loss(\bm{x},\chg{\bm{\tilde{\xi}}})]\ge J(\infty)$. Conversely, the infimum is at most $\E_{\Qprob(\infty)}[\Loss(\bm{x}(\infty),\chg{\bm{\tilde{\xi}}})]\le\sup_{\bm{\xi}\in\Xi}\Loss(\bm{x}(\infty),\bm{\xi})=J(\infty)$, so it equals $J(\infty)$. Since every $\Qprob\in\mathcal{P}(\Xi)$ obeys $\inf_{\bm{x}\in\mathcal{X}(\Xi)}\E_{\Qprob}[\Loss(\bm{x},\chg{\bm{\tilde{\xi}}})]\le\E_{\Qprob}[\Loss(\bm{x}(\infty),\chg{\bm{\tilde{\xi}}})]\le J(\infty)$, the distribution $\Qprob(\infty)$ attains the maximum in~\eqref{eq:Qinf} and is a finitely supported RO distribution.
\hfill\Halmos

\proofgap
\noindent \textbf{Proof of Proposition~\ref{prop:saddle-summary}.} $\;$
Lemmas~\ref{lem:ro-vertex-construction}, \ref{lem:ro-finite-max-construction}, and~\ref{lem:ro-convex-concave-construction} construct a finitely supported RO distribution from the solution of a finite convex problem under the assumptions of settings~\emph{(i)}, \emph{(ii)}, and~\emph{(iii)}, respectively. \hfill \Halmos

\end{document}